\documentclass[a4paper,11pt]{article}

\input{styleart.sty}

\title{Fields definable in the free group}
\date{\today}
\author{Ayala Byron \and Rizos Sklinos}

\begin{document}

\maketitle

\begin{abstract} 
We prove that no infinite field is definable in the theory of the free group.

\end{abstract}

\section{Introduction} 
After the work of Sela \cite{Sel7} leading to the positive answer to Tarski's question (see also \cite{KharlampovichMyasnikov}), 
there is an increasing model theoretic interest in the first-order theory of non abelian free groups.

Although, Sela proved that any definable set is equivalent to a boolean combination of $\forall\exists$-definable sets 
we are far from understanding these ``basic'' sets. According to Sela these sets admit a natural geometric interpretation but 
admittedly neither geometers nor logicians have absorbed the sophisticated techniques that occur in his voluminous work. 
Thus, in principle, it is hard to determine whether a subset 
of some cartesian power of a non abelian free group is definable or not.

Moreover, starting from Zilber's seminal work towards understanding uncountably categorical theories via some naturally defined pregeometries (see \cite{ZilberGeomStab}), questions 
about what kind of groups are definable or whether an infinite field is definable in a given first order theory have become important within the community of model theorists.

Recently, some positive results in this line of thought appeared. The following theorem has been proved 
independently in \cite{MalcevKM} and \cite{ForeignPPST}. 

\begin{theorem}
The only definable proper subgroups of a torsion-free hyperbolic group are cyclic.
\end{theorem}

When it comes to infinite definable fields in some non abelian free group, intuitively speaking, one expects to find none.   
To the best of our knowledge this has been first posed as a conjecture in \cite{PillayForking}. 
This problem proved very hard to tackle and the only positive result towards its solution had been the following theorem proved in the thesis 
\cite{ThesisSklinos} of the second named author: 

\begin{theorem}
Let $\F_n$ be the free group of rank $n$. Let $\phi$ be a formula over $\F_n$. Suppose $\phi(\F_n)\neq\phi(\F_{\omega})$ 
then $\phi$ cannot be given definably an abelian group structure.
\end{theorem}

On the other hand, in many model theoretic questions concerning existing ``configurations'' in a first order theory one does not need to understand the 
exact set of solutions of a formula, but just its rough ``shape''. Indeed in this vein there has been progress. A definable set (or a parametric family) can be endowed with 
an envelope that contains the definable set and moreover it carries a natural geometric structure from which we can read properties that when hold ``generically'' 
for the envelope, they also hold for the definable set. 
The method of envelopes has been proved very useful in proving (weak) elimination of imaginaries (see \cite{SelaIm}) but also in proving that the first 
order theory of the free group does not have the finite cover property (see \cite{FCP}). We will utilise 
envelopes once more in order to confirm the above conjecture.

The main theorem of this paper is:

\begin{thmIntro}
Let $\F$ be a non abelian free group. Then no infinite field is definable in $\F$.
\end{thmIntro}

Our proof is based on the following theorem which is a consequence of the quantifier elimination procedure. 
We record it next in the simplest possible form.

\begin{theorem}[Sela]
Let $\F:=\F(\bar{a})$. Let $\mathcal{T}(G,\F)$ be a hyperbolic tower where $G:=\langle \bar{u},\bar{x},\bar{a}\ | \ $ $\Sigma(\bar{u},\bar{x},\bar{a})\rangle$ 
and $\phi(\bar{x},\bar{a})$ be a first order formula over $\F$. 
Suppose there exists a test sequence, $(h_n)_{n<\omega}:G\rightarrow\F$ for $\mathcal{T}(G,\F)$, such that $\F\models\phi(h_n(\bar{x}),\bar{a})$. 

Then for any test sequence, $(h'_n)_{n<\omega}:G\rightarrow\F$, for $\mathcal{T}(G,\F)$ 
there is $n_0$ such that $\F\models \phi(h'_n(\bar{x}),\bar{a})$ for all $n>n_0$.
\end{theorem}

The notions of a {\em hyperbolic tower} and of a {\em test sequence} over it will be defined in section \ref{Towers} and section \ref{Envelopes} respectively.

On our way of proving the main theorem, we prove various Merzlyakov-type theorems. We recall 
that Merzlyakov's original theorem stated:

\begin{theorem}[Merzlyakov]
Let $\F:=\F(\bar{a})$. Let $\Sigma(\bar{x},\bar{y},\bar{a})\subset_{\textrm{finite}}\langle \bar{x},\bar{y}\rangle*\F$ be a finite set of words. Suppose 
$\F\models \forall\bar{x}\exists\bar{y}(\Sigma(\bar{x},\bar{y},\bar{a})=1)$. Then there exists 
a ``formal solution'' $\bar{w}(\bar{x},\bar{a})\subseteq\langle\bar{x}\rangle*\F$ such that $\Sigma(\bar{x},\bar{w}(\bar{x},\bar{a}),\bar{a})$ 
is trivial in $\langle\bar{x}\rangle*\F$.
\end{theorem}

We obtain the following ``generalisation'' to an arbitrary formula, 
but only after strengthening the hypothesis and weakening the conclusion of Merzlyakov's theorem as follows: 

\begin{thmIntro}
Let $\F:=\F(\bar{a})$. Let $\F\models \forall\bar{x}\exists^{<\infty}\bar{y}\phi(\bar{x},\bar{y},\bar{a})$ and assume there 
exists a test sequence $(\bar{b}_n)_{n<\omega}$ and a sequence of tuples $(\bar{c}_n)_{n<\omega}$ such that $\F\models\phi(\bar{b}_n,\bar{c}_n,\bar{a})$. 
Then there exists a tuple of words $\bar{w}(\bar{x},\bar{a})$ in $\langle \bar{x}\rangle*\F$ 
such that for any test sequence $(\bar{b}'_n)_{n<\omega}$ in $\F$ we have that 
there exists $n_0$ (that depends on the test sequence) with $\F\models\phi(\bar{b}'_n,\bar{w}(\bar{b}'_n),\bar{a})$ 
for all $n>n_0$.
\end{thmIntro}

We remark that our main theorem implies, together with the elimination of the ``exists infinitely many'' quantifier $\exists^{\infty}$,  
that no infinite field is definable in any model of this theory. 

The proof splits in two parts. Roughly speaking for any definable set $X$ we prove that either $X$ is internal to a finite set of centralizers 
or it cannot be given definably an abelian group operation, i.e. there is no definable set $Y\subset X\times X\times X$ such that 
$Y$ is the graph of an abelian group operation on $X$. To conclude that no infinite field is definable we prove that 
centralizers of non trivial elements are one-based.

The paper is structured as follows: in the next section we take the opportunity to recall some basic geometric stability theory 
and introduce the reader to the results that will allow us to conclude our theorem in the case a definable set is ``coordinated'' 
by a finite set of centralizers. 

The following section contains introductory material that concerns Bass-Serre theory as well as results for a special class of groups called limit groups. The material 
here is by no means original and certainly well known. Since in many of our arguments we will use actions on trees or normal forms for groups that admit 
a {\em graph of groups} splitting, we hope that this section will provide an adequate background for the uninitiated.

Section \ref{Towers} contains many of the core notions which are important in this paper. We start by explaining when a group 
admits the structure of a {\em tower} and then we continue by introducing a construction that leads to the notion of a {\em twin tower}. 
Twin towers would play a fundamental role in our main proof.

In sections \ref{Completions} and \ref{Envelopes} we record and extend some constructions and results of Sela that appear in \cite{Sel2}, \cite{Sel3} and \cite{SelaIm}. 
Here the reader will find all the technical apparatus that makes our main proof possible. Theorems \ref{RigidGradedTower}, 
\ref{PrepareExtFormalSolutions} and \ref{ExtFormalSolutions}  lie in the core of our result.

Finally, in the last section we bring everything together and we prove the main result. We split the proof in two cases: the abelian case and the 
non abelian case. The abelian case is resolved using geometric stability, while the non abelian case is resolved 
using geometric group theory. We have also added an example, which we call the hyperbolic case, 
where our proof is free of certain technical phenomena, so the reader could clearly see the idea behind it.

\section{Some geometric stability}
In this section we provide some quick model theoretic background on stable theories. A gentle introduction to stability and forking independence 
has been given in \cite{ForkingPerinSklinos}, so to avoid repetition we refer the reader there. Our main focus in this paper will be on geometric stability and 
in particular on the notion of one-basedness. For more details the reader can consult \cite{PillayStability}. 
We work in the monster model $\mathbb{M}$ of a stable theory $T$. 

\begin{definition}
A definable set $X$ (in $\mathbb{M}$) is called weakly normal if for every $a\in X$ only 
finitely many translates of $X$ under $Aut(\mathbb{M})$ contain $a$.
\end{definition}

\begin{definition}
The first order theory $T$ is one-based, if every definable set (in $\mathbb{M})$ is a boolean combination 
of weakly normal definable sets.
\end{definition}

The simplest example of an one-based theory is the theory of a vector space $(V,+,0, \{r_k\}_{k\in K})$ over a field $K$, where for each $k\in K$, $r_k$ is a function symbol which 
is interpreted in the structure as scalar multiplication by the element $k\in K$. In the same vein we have:

\begin{fact} 
The theory of any abelian group (in the group language) is one-based.
\end{fact}

For the purposes of our paper one-based theories have an important property.

\begin{fact}[Pillay]
Let $T$ be one-based. Then no infinite field is interpretable in $T$. 
\end{fact}

A set is {\em interpretable} if it definable up to a definable equivalence relation, thus, 
in particular no infinite field is definable in an one-based theory. 
\\

For any definable set $X$ in $\mathbb{M}$ we can define the {\em induced structure} on $X$, 
$X^{ind}$, in the following way: the universe of the structure will be $X$ and 
for every definable set $Y$ in $\mathbb{M}$ we add a predicate $P_Y$ that 
corresponds to the intersection $Y\cap X$. 

\begin{definition}
Let $X$ be a definable set in $\mathbb{M}$. Then $X$ is one-based, if the 
first order theory of $X^{ind}$ is one-based.
\end{definition}

We say that a family of definable sets $\mathcal{P}$ is {\em $\emptyset$-invariant},  
if the image of any definable set in $\mathcal{P}$ by an automorphism in $\Aut(\mathbb{M})$ is still in $\mathcal{P}$.
Moreover, if $B$ is a small subset of $\mathbb{M}$, we say that $\bar{c}\models \mathcal{P}\upharpoonright B$ 
if $tp(\bar{c}/B)$ contains some definable set that belongs to $\mathcal{P}$.

\begin{definition} 
A definable set $X$ (over some small subset $A\subset \mathbb{M}$) is $\mathcal{P}$-internal for some 
$\emptyset$-invariant family of definable sets $\mathcal{P}$, if for any $\bar{c}\in X$ there 
exists $B\supseteq A$ with $\bar{c} \underset{A}{\forkindep} B$ and $\bar{b}_1,\ldots,\bar{b}_k$ 
with $\bar{b}_i\models\mathcal{P}\upharpoonright B$ such that $\bar{c}\in dcl(B,\bar{b}_1,\ldots,\bar{b}_k)$
\end{definition}

The following theorem has been proved by F. Wagner in \cite{WagnerInternal}.

\begin{theorem}[Wagner]
Let $X$ be a definable set and $\mathcal{P}$ be an $\emptyset$-invariant family of one-based sets. If $X$ is $\mathcal{P}$-internal, 
then $X$ is one-based.
\end{theorem}

\section{Actions on trees}
The goal of this section is to present a structure theorem for groups acting on trees. 
In the first subsection we will be interested in group actions on simplicial trees. These actions 
can be analysed using Bass-Serre theory, and we will explain the duality between the notion 
of a graph of groups and these actions. 

In the second subsection we 
will record the notion of a real tree and quickly describe some natural group actions on real trees.

\subsection{Bass-Serre Theory}\label{BS}
Bass-Serre theory gives a structure theorem for groups acting on (simplicial) trees, i.e. acyclic connected graphs. 
It describes a group (that acts on a tree) 
as a series of {\em amalgamated free products} and {\em HNN extensions}. The mathematical notion that 
contains these instructions is called a graph of groups. For a complete treatment we refer the reader to 
\cite{SerreTrees}.

We start with the definition of a graph.

\begin{definition}
A graph $G(V,E)$ is a collection of data that consists of two sets $V$ (the set of vertices) and $E$ (the set of edges) 
together with three maps:
\begin{itemize}
 \item an involution $\bar{\phantom{1}}:E\to E$, where $\bar{e}$ is called the inverse of $e$;
 \item $\alpha:E\to V$, where $\alpha(e)$ is called the initial vertex of $e$; and
 \item $\tau:E\to V$, where $\tau(e)$ is called the terminal vertex of $e$. 
\end{itemize}
so that $\bar{e}\neq e$, and $\alpha(e)=\tau(\bar{e})$ for every $e\in E$.
\end{definition}

\begin{definition}[Graph of Groups]
A graph of groups $\mathcal{G}:=(G(V,E), \{G_u\}_{u\in V}, \{G_e\}_{e\in E},$ $\{f_e\}_{e\in E})$ consists of the following data:
\begin{itemize}
 \item a graph $G(V,E)$;
  \item a family of groups $\{G_u\}_{u\in V}$, i.e. a group is attached to each vertex of the graph;
  \item a family of groups $\{G_e\}_{e\in E}$, i.e. a group is attached to each edge of the graph. Moreover, $G_{e}=G_{\bar{e}}$;
  \item a collection of injective morphisms $\{f_e:G_e\to G_{\tau(e)} \ | \ e\in E\}$, i.e. each edge group comes 
  equipped with two embeddings to the incident vertex groups. 
\end{itemize}

\end{definition}

The fundamental group of a graph of groups is defined as follows.

\begin{definition}
Let $\mathcal{G}:=(G(V,E), \{G_u\}_{u\in V}, \{G_e\}_{e\in E}, \{f_e\}_{e\in E})$ be a graph of groups. Let $T$ be a maximal subtree of $G(V,E)$. 
Then the fundamental group, $\pi_1(\mathcal{G},T)$, of $\mathcal{G}$ with respect to $T$ is the group given by the following presentation:
$$\langle \{G_u\}_{u\in V}, \{t_e\}_{e\in E} \ | \ t^{-1}_e=t_{\bar{e}} \ \textrm{for} \ e\in E, t_e=1 \ \textrm{for} \ e\in T, 
f_e(a)=t_ef_{\bar{e}}(a)t_{\bar{e}} \ \textrm{for} \ e\in E \ a\in G_e\rangle$$
\end{definition}

\begin{remark}
It is not hard to see that the fundamental group of a graph of groups does not depend on the 
choice of the maximal subtree up to isomorphism (see \cite[Proposition 20, p.44]{SerreTrees}).
\end{remark}

In order to give the main theorem of Bass-Serre theory we need the following definition.

\begin{definition}  
Let $G$ be a group acting on a simplicial tree $T$ without inversions,
denote by $\Lambda$ the corresponding quotient graph and by $p$ the quotient map $T \to \Lambda$. 
A Bass-Serre presentation for the action of $G$ on $T$ is a triple $(T^1, T^0, \{\gamma_e\}_{e\in E(T^1)\setminus E(T^0)})$ consisting of
\begin{itemize}
\item a subtree $T^1$ of $T$ which contains exactly one edge of $p^{-1}(e)$ for each edge $e$ of $\Lambda$;
\item a subtree $T^0$ of $T^1$ which is mapped injectively by $p$ onto a maximal subtree of $\Lambda$;
\item a collection of elements of $G$, $\{\gamma_e\}_{e\in E(T^1)\setminus E(T^0)}$, such that if $e=(u,v)$ with $v\in T^1\setminus T^0$, 
then $\gamma_e\cdot v$ belongs to $T^0$.
\end{itemize}
\end{definition}

\begin{theorem}
Suppose $G$ acts on a simplicial tree $T$ without inversions. Let $(T^1,T^0, \{\gamma_e\})$ be a Bass-Serre presentation for the action. 
Let $\mathcal{G}:=(G(V,E), \{G_u\}_{u\in V}, \{G_e\}_{e\in E}, \{f_e\}_{e\in E})$ be the following graph of groups:
\begin{itemize}
 \item $G(V,E)$ is the quotient graph given by $p:T\to \Gamma$;
 \item if $u$ is a vertex in $T^0$, then $G_{p(u)}=Stab_G(u)$;
 \item if $e$ is an edge in $T^1$, then $G_{p(e)}=Stab_G(e)$;
 \item if $e$ is an edge in $T^1$, then $f_{p(e)}:G_{p(e)}\to G_{\tau(p(e))}$ is given by the identity 
 if $e\in T^0$ and by conjugation by $\gamma_e$ if not.
\end{itemize}

Then $G$ is isomorphic to $\pi_1(\mathcal{G})$. 
\end{theorem}

\begin{remark}
The other direction of the above theorem also holds. Whenever a group $G$ is isomorphic to the fundamental group of a graph of groups 
then there is a natural way to obtain a simplicial tree $T$ and an action of $G$ on $T$ (see \cite[section 5.3, p.50]{SerreTrees}).
\end{remark}

Among splittings of groups we will distinguish those 
with some special type vertex groups called {\em surface type vertex groups}.

\begin{definition}
Let $G$ be a group acting on a tree $T$ without inversions and $(T_1,T_0,\{\gamma_e\})$ be a Bass-Serre presentation for this action. 
Then a vertex $v\in T^0$ is called a surface type vertex if the following conditions hold:
\begin{itemize}
 \item $\Stab_G(v)=\pi_1(\Sigma)$ for a connected compact surface $\Sigma$ with non-empty boundary;
 \item For every edge $e\in T_1$ adjacent to $v$, $\Stab_G(e)$ embeds onto a maximal boundary 
 subgroup of $\pi_1(\Sigma)$, and this induces a one-to-one correspondence between the 
 set of edges (in $T^1$) adjacent to $v$ and the set of boundary components of $\Sigma$.
\end{itemize}

\end{definition}

We next follow \cite{BFNotesOnSela} and define the notion of a Generalized Abelian Decomposition (GAD). 

\begin{definition}
A GAD of a group $G$ is a graph of groups $\mathcal{G}(V,E)$ with abelian edge groups and such that $V$ is partitioned as $V_S\cup V_A \cup V_R$ where:
\begin{itemize}
 \item each vertex in $V_S$ is a vertex of surface type for the corresponding action on a tree;
 \item each vertex group for a vertex in $V_A$ is non-cyclic abelian; and
 \item each vertex group for a vertex $V_R$ is called rigid.
\end{itemize}

\end{definition}

\begin{definition}[Peripheral subgroup]
Let $A$ be a vertex group of a GAD of $G$, $(\mathcal{G}(V,E), (V_S,$ $V_A,V_R))$, whose vertex is in $V_A$. Then the we denote 
by $P(A)$ the subgroup of $A$ generated by all incident edges groups. Moreover the subgroup of $A$ that dies under every morphism 
$h:A\rightarrow \Z$ that kills $P(A)$ is called the peripheral subgroup and denoted by $\bar{P}(A)$.
\end{definition}

\begin{definition}[Dehn twists]
Let $H$ be a subgroup of $G$. Suppose $G$ splits as an: 
\begin{itemize} 
 \item amalgamated free product $A*_CB$ so that $H$ is a subgroup of $A$. Let $g$ be an element in the centralizer of $C$ in $G$. Then 
 a Dehn twist in $g$ is the automorphism fixing $A$ pointwise and sending each element $b$ of $B$ to $gbg^{-1}$;
 \item $HNN$-extension $A*_C$ so that $H$ is a subgroup of $A$. Let $g$ be an element in the centralizer of $C$ in $G$. Then 
 a Dehn twist in $g$ is the automorphism fixing $A$ and sending the Bass-Serre element $t$ to $tg$.
\end{itemize}

\end{definition}

\begin{definition}[Relative modular automorphisms]
Let $H$ be a subgroup of $G$. Let $\Delta:=(\mathcal{G}(V,E), (V_S,V_A,V_R))$ be a GAD of $G$ in which $H$ can be conjugated into a vertex group. 
Then $Mod_H(\Delta)$ is the subgroup of $Aut_H(G)$ generated by:
\begin{itemize}
 \item inner automorphisms;
 \item unimodular automorphisms of $G_u$ for $u\in V_A$ that fix the peripheral subgroup of $G_u$ and every other vertex group;
 \item automorphisms of $G_u$ for $u\in V_S$ coming from homeomorphisms of the corresponding surface that fix all boundary components;
 \item Dehn twists in elements of centralizers of edge groups, after collapsing the GAD to a one edge splitting in which $H$ 
 is a subgroup of a vertex group. 
\end{itemize}
Moreover we define the modular group of $G$ relative to $H$, $Mod_H(G)$, 
to be the group generated by $Mod_H(\Delta)$ for every $GAD$ $\Delta$ of $G$.  
\end{definition}

\subsection{Actions on real trees}
Real trees (or $\R$-trees) generalize simplicial trees in the following way. 

\begin{definition}
A real tree is a geodesic metric space in which 
for any two points there is a unique arc that connects them.
\end{definition}

When we say that a group $G$ acts on an real tree $T$ we will always mean an action by isometries.

Moreover, an action $G\curvearrowright T$ of a group $G$ on a real tree $T$ is called {\em non-trivial} if there is no 
globally fixed point and {\em minimal} if there is no proper $G$-invariant subtree. Lastly, an action is called {\em free} 
if for any $x\in T$ and any non trivial $g\in G$ we have that $g\cdot x\neq x$.


We next collect some families of group actions on real trees.

\begin{definition}
Let $G\curvearrowright^{\lambda} T$ be a minimal action of a finitely generated group $G$ on a real tree $T$. Then we say:
\begin{itemize}
 \item[(i)] $\lambda$ is of discrete (or simplicial) type, if every orbit $G.x$ is discrete in $T$. In this case $T$ is simplicial 
 and the action can be analysed using Bass-Serre theory;
 \item[(ii)] $\lambda$ is of axial (or toral) type, if $T$ is isometric to the real line $\R$ and $G$ acts with dense orbits, 
 i.e. $\overline{G.x}=T$ for every $x\in T$;
 \item[(iii)] $\lambda$ is of surface (or IET) type, if $G=\pi_1(\Sigma)$ where $\Sigma$ is a surface with (possibly empty) boundary 
 carrying an arational measured foliation and $T$ is dual to $\tilde{\Sigma}$, i.e. $T$ is the lifted leaf space in $\tilde{\Sigma}$ 
 after identifying leaves of distance $0$ (with respect to the pseudo-metric induced by the measure);
\end{itemize}
\end{definition}


We will use the notion of a graph of actions in order to glue real trees equivariantly. 
We follow the exposition in \cite[Section 1.3]{GuirardelRTrees}.

\begin{definition}[Graph of actions]
A graph of actions $(G\curvearrowright T,\{Y_u\}_{u\in V(T)},\{p_e\}_{e\in E(T)})$ consists of the following data:
\begin{itemize}
 \item A simplicial type action $G\curvearrowright T$; 
 \item for each vertex $u$ in $T$ a real tree $Y_u$;
 \item for each edge $e$ in $T$, an attaching point $p_e$ in $Y_{\tau(e)}$.
\end{itemize}

Moreover:
\begin{enumerate}
 \item $G$ acts on $R:=\{\coprod Y_u : u\in V(T)\}$ so that $q:R\to V(T)$ with $q(Y_u)=u$ is $G$-equivariant;
 \item for every $g\in G$ and $e\in E(T)$, $p_{g\cdot e}=g\cdot p_e$.
\end{enumerate}

\end{definition}

To a graph of actions $\mathcal{A}:=(G\curvearrowright T,\{Y_u\}_{u\in V(T)},\{p_e\}_{e\in E(T)})$ we can assign an $\R$-tree $Y_{\mathcal{A}}$ endowed with 
a $G$-action. Roughly speaking this tree will be $\coprod_{u\in V(T)} Y_{u}/\sim$, where the equivalence relation $\sim$ identifies $p_e$ with 
$p_{\bar{e}}$ for every $e\in E(T)$. We say that a real $G$-tree $Y$ {\em decomposes as a graph of actions} $\mathcal{A}$, if there is 
an equivariant isometry between $Y$ and $Y_{\mathcal{A}}$. 






Interesting actions on real trees can be obtained by sequences of morphisms from a finitely generated group to a free group. 
We explain how in the next subsection.

\subsection{The Bestvina-Paulin method}
The construction we are going to record is credited to Bestvina \cite{BesLimit} and Paulin \cite{PaulinGromov} independently. 

We fix a finitely generated group $G$ and we consider the set of non-trivial 
equivariant pseudometrics $d:G\times G\to \R^{\geq 0}$, denoted by $\mathcal{ED}(G)$. 
We equip $\mathcal{ED}(G)$ with the compact-open topology (where $G$ is given the discrete topology). Note 
that convergence in this topology is given by: 
$$(d_i)_{i<\omega}\to d\ \ \textrm{if and only if} \ \ d_i(1,g)\to d(1,g)\ \ (\textrm{in $\R$}) \ \ \textrm{for any}\ \ g\in G$$

Is not hard to see that $\R^+$ acts cocompactly on $\mathcal{ED}(G)$ by rescaling, thus 
the space of {\em projectivised equivariant pseudometrics} on $G$ is compact.

We also note that any based $G$-space $(X,*)$ (i.e. a metric space with a distinguished point equipped with an action of 
$G$ by isometries) gives rise to an equivariant pseudometric on $G$ as follows: $d(g,h)=d_X(g\cdot *,h\cdot *)$. 

We say that a sequence of $G$-spaces $(X_i,*_i)_{i<\omega}$ converges to a $G$-space $(X,*)$, if the 
corresponding pseudometrics induced by $(X_i,*_i)$ converge to the pseudometric induced by $(X,*)$ in $\mathcal{PED}(G)$. 

A morphism $h:G\to H$ where $H$ is a finitely generated group induces an action of $G$ on  
$\mathcal{X}_H$ (the Cayley graph of $H$) in the obvious way, thus making $\mathcal{X}_H$ a $G$-space. We have:

\begin{lemma}\label{LimitAction} 
Let $\F$ be a non abelian free group. Let $(h_n)_{n<\omega}:G\to\F$ be a sequence of pairwise non-conjugate morphisms. 
Then for each $n<\omega$ there exists a base point $*_n$ 
in $\mathcal{X}_{\F}$ such that the sequence of $G$-spaces $(\mathcal{X}_{\F},*_n)_{n<\omega}$ has a convergent subsequence to 
a real $G$-tree $(T,*)$, where the action of $G$ on $T$ is non-trivial.
\end{lemma}

\subsection{Limit groups}

\begin{definition}
Let $G$ be a group and $(h_n)_{n<\omega}:G\rightarrow \F$ be a sequence of morphisms. Then the sequence $(h_n)_{n<\omega}$ is 
called convergent if for every $g\in G$, there exists $n_g$ such that either $h_n(g)\neq 1$ for all $n>n_g$ or $h_n(g)=1$ for all $n>n_g$. 

Moreover, if $(h_n)_{n<\omega}:G\rightarrow\F$ is a convergent sequence, then we define its stable kernel 
$\Ker h_n:=\{ g\in G \ | \ g \ \textrm{is eventually killed by}\ h_n \}$  
\end{definition}

\begin{definition}
Let $G$ be a finitely generated group. Then $G$ is a limit group if there exists a convergent sequence $(h_n)_{n<\omega}:G\rightarrow\F$ 
with trivial stable kernel.
\end{definition}

Limit groups can be given a constructive definition. To this end we define:

\begin{definition}
Let $\Delta:=(\mathcal{G}(V,E), (V_S,V_A,V_R))$ be a $GAD$ of a group $G$. For each $G_v$ with $v\in V_R$ we define its envelope 
$\tilde{G_v}$ in $\Delta$ in the following way: for every $a\in V_A$ we replace $G_a$ in $\mathcal{G}(V,E)$ by its peripheral subgroup, then $\tilde{G_v}$ 
is the group generated by $G_v$ together with the centralizers of incident edge groups.
\end{definition}

\begin{definition}[Strict morphisms]
Let $\eta:G\twoheadrightarrow L$ be an epimorphism and $\Delta:=(\mathcal{G}(V,E),$  $(V_S,V_A,V_R))$ be a GAD of $G$ in which every edge group is maximal abelian 
in at least one vertex group of the one edged splitting induced by the edge. Then $\eta$ is strict with respect to $\Delta$ if the following hold:
\begin{itemize}
 \item $\eta$ is injective on each edge group;
 \item $\eta$ is injective on $\tilde{G_v}$ for every $v\in V_R$;
 \item $\eta$ is injective on the peripheral subgroup of each abelian vertex group;
 \item $\eta(G_s)$ is not abelian for every $s\in V_S$.
\end{itemize}

\end{definition}

\begin{definition}
A group $L$ is a constructive limit group if it belongs to the following hierarchy of groups defined recursively:
\\
{\bf Base step.} Level $0$ consists of finitely generated free groups;\\ 
{\bf Recursive step.} A group $G$ belongs to level $i+1$ if it is either the free product of two groups that belong to 
 level $i$, or there exists a $GAD$, $\Delta$, for $G$ and a strict map $\eta:G\twoheadrightarrow H$ with respect to $\Delta$ 
 onto some $H$ that belongs to level $i$.

\end{definition}

\begin{theorem}[Sela]
Let $L$ be a finitely generated group. Then $L$ is a limit group if and only if it is a constructive limit group.
\end{theorem}

\subsection{Graded limit groups}

A {\em graded limit group} is a limit group together with a distinguished finitely generated subgroup. We will 
be interested in a special kind of graded limit groups called {\em solid limit groups}.

\begin{definition}
Let $G$ be a limit group which is freely indecomposable with respect to a finitely generated subgroup $H$. 
We fix a finite generating set, $\Sigma$ for $G$ and a basis $\bar{a}$ for $\F:=\F(\bar{a})$.  
A morphism $h:G\rightarrow \F$ is short with respect to $Mod_H(G)$ if for every $\sigma\in Mod_H(G)$ and 
every $g\in \F$ that commutes with $h(H)$ we have that $max_{s\in\Sigma}\abs{h(s)}_{\F}\leq max_{s\in\Sigma}\abs{Conj(g)\circ h\circ \sigma}_{\F}$.
\end{definition}

\begin{definition}
Let $G$ be a limit group which is freely indecomposable with respect to a finitely generated subgroup $H$. 
Let $(h_n)_{n<\omega}:G\rightarrow \F$ be a convergent sequence of short morphisms with respect to $Mod_H(G)$. 
Then we call $G/\Ker h_n$ a shortening quotient of $G$ with respect to $H$.
\end{definition}

\begin{definition}[Solid limit group]
Let $S$ be a limit group and $H$ be a finitely generated subgroup of $S$. 
Suppose $S$ is freely indecomposable with respect to $H$. Then 
$S$ is solid with respect to $H$ if there exists a shortening quotient  
of $S$ with respect to $H$ which is isomorphic to $S$.
\end{definition}

\begin{example}
The surface group $\langle x_1,x_2,\ldots,x_{2k},e_1,e_2,\ldots,e_{2m} \ | \ [x_1,x_2][x_3,x_4]
 \ldots[x_{2k-1},$ $x_{2k}][e_1,e_2]$ $\ldots[e_{2m-1},e_{2m}]\rangle$ is a solid limit group with respect to the subgroup 
 $\langle e_1,\ldots,e_{2m}\rangle$.
\end{example}

\begin{definition}
Let $Sld$ be a solid limit group with respect to a finitely generated subgroup $H$ with generating set $\Sigma_H$. 
Let $(h_n)_{n<\omega}:Sld\rightarrow \F$. We call $(h_n)_{n<\omega}$ a flexible sequence if for every $n$, either: 
\begin{itemize}
 \item the morphism $h_n=h'_n\circ\eta_n$, where $\eta_n:Sld\twoheadrightarrow \F*\Gamma$ for some group $\Gamma$, 
 $H$ is mapped onto $\F$ by $\eta_n$, $h'_n:\F*\Gamma\rightarrow\F$ stays the identity on $\F$, and 
 $\eta_n$ is short with respect to $Mod_H(Sld)$; or   
 \item the morphism $h_n$ is short with respect to $Mod_H(Sld)$ and moreover 
 $$max_{g\in B_n}\abs{h_n(g)}_{\F}> 2^n(1+max_{s\in\Sigma_H}\abs{h_n(s)}_{\F})$$ 
 where $B_n$ is the ball of radius $n$ in the Cayley graph of $Sld$.
\end{itemize}
If $(h_n)_{n<\omega}:Sld\rightarrow\F$ is a convergent flexible sequence, then we call $Sld/\Ker h_n$ a flexible quotient of $Sld$.
\end{definition}

It is not hard to see, using the shortening argument, that flexible quotients are proper. Moreover one can 
define a partial order and an equivalence relation on the class of flexible quotients of a solid limit group. 
Let $Sld$ be a solid limit group with respect to a finitely generated subgroup $H$ 
and $\eta_i:Sld\twoheadrightarrow Q_i$ for $i\leq 2$ be flexible quotients with their canonical quotient maps. 
Then $Q_2\leq Q_1$ if $ker\eta_1\subseteq ker\eta_2$. And $Q_1\sim Q_2$ if there exists $\sigma\in Mod_H(Sld)$ 
such that $ker(\eta_1\circ\sigma)=ker\eta_2$.

\begin{theorem}[Sela]
Let $Sld$ be a solid limit group with respect to a finitely generated subgroup $H$. Assume that $Sld$ admits a flexible quotient. Then there exist  
finitely many classes of maximal flexible quotients. 
\end{theorem}

A morphism from a solid limit group to a free group that does not factor through one of the maximal flexible quotients (after precomposition by a modular automorphism) 
is called a {\em solid morphism}.

\section{Towers}\label{Towers}
In this section we are interested in limit groups that have a very special structure, namely the structure of a {\em tower}. A 
tower is built recursively adding {\em floors} to a given basis, which is taken to be a free product of fundamental groups of surfaces with free abelian 
groups. Each floor is built by ``gluing'' a finite set of {\em surface flats} and {\em abelian flats} to the previous one following specific rules, 
to be made precise in the next subsection.  

Limit groups that admit the structure of a tower play a significant role in 
the proof of the elementary equivalence of non abelian free groups. This class of limit groups is connected with Merzlyakov-type 
theorems as proved in \cite{Sel2}. We will analyse and further expand this connection in section \ref{Envelopes}.

\subsection{The construction of a tower}

We start with defining the notion of a surface flat. 

\begin{definition}[Surface flat]
Let $G$ be a group and $H$ be a subgroup of $G$. Then $G$ has the structure of a surface flat over $H$,  
if $G$ acts minimally on a tree $T$ and the action admits a Bass-Serre presentation 
$(T^1, T^0, \{\gamma_e\})$ such that:
\begin{itemize}
\item the set of vertices of $T^1$ is partitioned in two sets, $\{v\}$ and $V$, where $v$ is a surface type vertex; 
\item $T^1$ is bipartite between $v$ and $V(T^1)\setminus \{v\}$; 
\item $H$ is the free product of the stabilizers of vertices in $V$;
\item either there exists a retraction $r:G\to H$ that sends $\Stab_G(v)$ to a non abelian image  
or $H$ is cyclic and there exists a retraction $r': G * \Z \to H * \Z$ which sends $\Stab_G(v)$ to a non abelian image.
\end{itemize}

\end{definition}

 \begin{figure}[ht!]
\centering
\includegraphics[width=.7\textwidth]{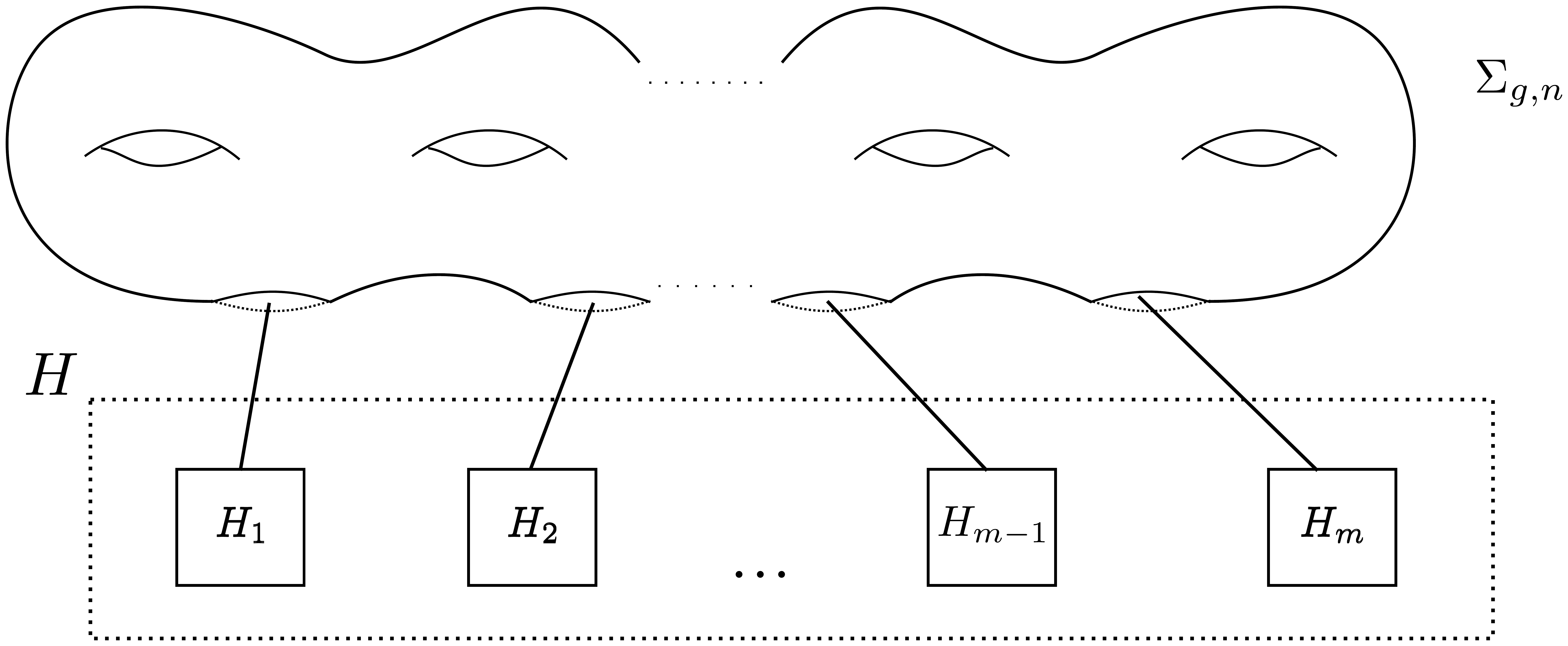}
\caption{A surface flat.}
\end{figure}

\begin{remark}
A more concise way to refer to a group $G$ that has the structure of a surface flat over a subgroup $H$ is to say 
that $G$ is obtained from $H$ by gluing a surface $\Sigma_{g,n}$ along its boundary onto the subgroups $\{E_i\ :\ i\leq n\}$ of $H$.
\end{remark}

\begin{example}
The surface group $\pi_1(\Sigma_{g})=\langle x_1,\ldots,x_{2g} \ | \ [x_1,x_2]\cdot\ldots\cdot[x_{2g-1},x_{2g}]\rangle$ 
has the structure of a surface flat over $\F_g:=\langle x_1,\ldots,x_g\rangle$. 

Concisely, $\pi_1(\Sigma_g)$ is obtained from $\F_g$ by gluing $\langle b,x_{g+1},\ldots,x_{2g} \ | 
\ b^{-1}\cdot[x_{g+1},x_{g+2}]\cdot\ldots\cdot$ $[x_{2g-1},x_{2g}]\rangle$ along its boundary onto $\langle [x_1,x_2]\ldots[x_{g-1},x_g]\rangle$

\begin{figure}[ht!]
\centering
\includegraphics[width=.9\textwidth]{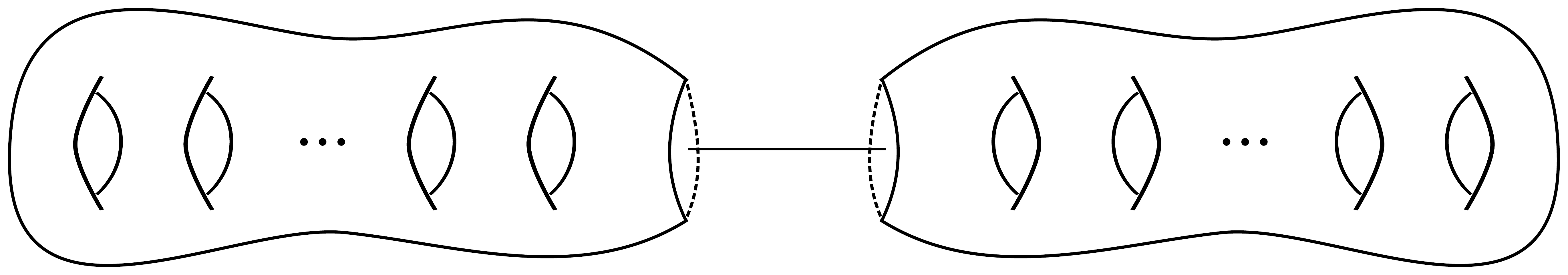}
\caption{An example of a surface group that has the structure of a surface flat over a free group.}
\end{figure}
\end{example}

We similarly define the notion of an abelian flat. 

\begin{definition}
Let $G$ be a group and $H$ be a subgroup of $G$. Then $G$ has the structure of an abelian flat over $H$,  
if $G$ acts minimally on a tree $T$ with a single orbit of edges, 
and the action admits a Bass-Serre presentation $(T^1=T^0, T^0)$ so that if $e=\{u,v\}$ is the unique edge in $T^0$, 
then $\Stab_G(u)=H$, $\Stab_G(e)$ is a maximal abelian subgroup of $H$, which we call the peg of the abelian flat, and 
$\Stab_G(v)=\Stab_G(e)\oplus\Z^m$ for some $m<\omega$. 
\end{definition}

\begin{figure}[ht!]
\centering
\includegraphics[width=.5\textwidth]{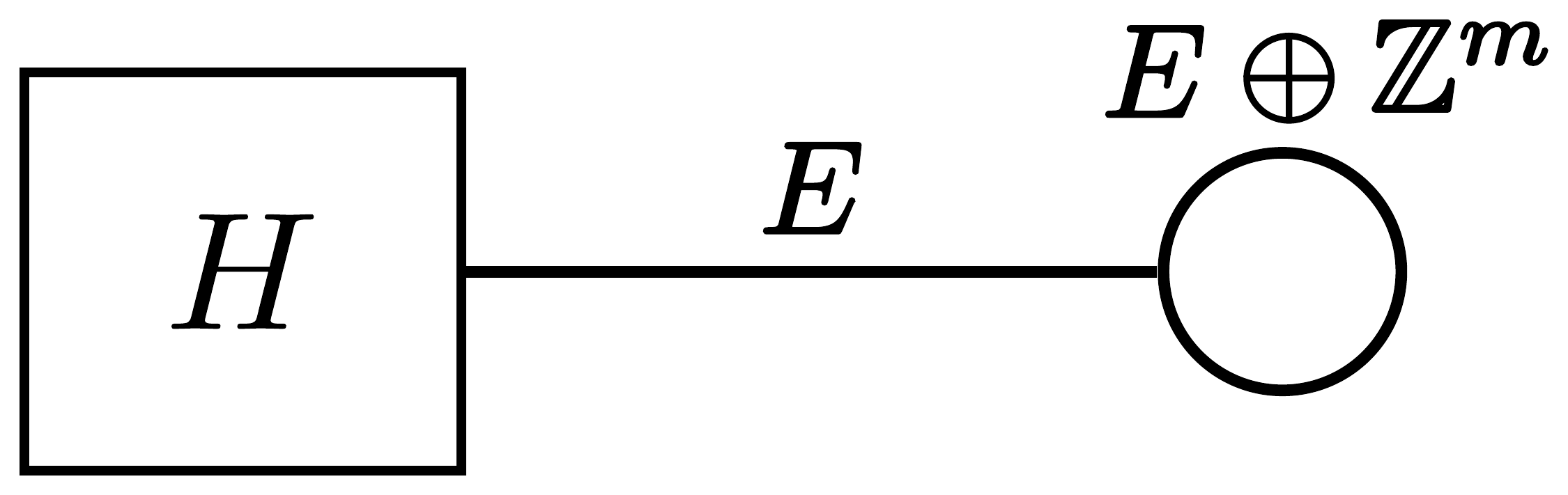}
\caption{An abelian flat.}
\end{figure}

\begin{remark}
A more concise way to refer to a group $G$ that has the structure of an abelian flat over a subgroup $H$, is to say that 
$G$ is obtained from $H$ by gluing a free abelian group $\Z^n$ along the (maximal abelian) subgroup $E$ of $H$.
\end{remark}

We note that when we say ``gluing $\Z^n$ along the subgroup $E$ of $H$'', the outcome will really be the amalgamated free product $(E\oplus\Z^n)*_EH$, 
but we keep this terminology as it will be convenient in the sequel.

We observe that if $G$ has the structure of an abelian flat over a subgroup $H$, then it is not hard to find a retraction $r:G\rightarrow H$:  
one can use the projection of $E\oplus\Z^m$ to $E$ and extend this to a morphism from $G$ to $H$ which 
fixes $H$.

\begin{example}
The group $G:=\langle x_1,\ldots,x_n,z_1,\ldots,z_k \ | \ x_1^2\cdot\ldots\cdot x_n^2=1, [x_1,z_i], [z_i,z_j], i,j\leq k\rangle$ 
has the structure of a free abelian flat over the subgroup $\langle x_1,\ldots, x_n \ | \ x_1^2\cdot\ldots\cdot x_n^2 \rangle$.

Concisely, $G$ is obtained from $\langle x_1,\ldots, x_n \ | \ x_1^2\ldots x_n^2\rangle$ by gluing the free 
abelian group $\Z^k$ along the subgroup $\langle x_1\rangle$ of $\langle x_1,\ldots, x_n \ | \ x_1^2\ldots x_n^2\rangle$.
\end{example}

We can combine surface and abelian flats in order to obtain the ``floors'' of a tower.

\begin{definition}[Floor]
Let $G$ be a group and $H$ be a subgroup of $G$. Then $G$ has the structure of a floor over $H$,  
if $G$ acts minimally on a tree $T$ and the action admits a Bass-Serre presentation 
$(T^1, T^0, \{\gamma_e\})$, where the set of vertices of $T^1$ is partitioned in three subsets, $V_S, V_A$ and $V_R$, 
such that:
\begin{itemize}
\item each vertex in $V_S$ is a surface type vertex; 
\item for each vertex $u\in V_2$, its stabilizer $G_{u}$ is a free abelian group;  
\item the tree $T^1$ is bipartite between $V_S\cup V_A$ and $V_R$; 
\item the subgroup $H$ of $G$ is the free product of the stabilizers of vertices in $V_R$;
\item for each $u\in V_A$, there is a unique edge $e=\{v,u\}$ connecting $u$ to a vertex $v$ in $V_R$. Moreover, 
$\Stab_G(e)$ is maximal abelian in $G$ and $Stab_G(u)=Stab_G(e)\oplus\Z^m$. In addition, the stabilizer $Stab_{G}(e)$ of $e$ 
cannot be conjugated to any other stabilizer $\Stab_G(e')$ for $e'\neq e$ an edge connecting  
a vertex in $V_A$ to a vertex in $V_R$;
\item either there exists a retraction $r:G\to H$ that, for each $v\in V_S$, sends $\Stab_G(v)$ to a non abelian image  
or $H$ is cyclic and there exists a retraction $r': G * \Z \to H * \Z$ that, for each $v\in V_S$, sends $\Stab_G(v)$ to a non abelian image.
\end{itemize}
\end{definition}

 \begin{figure}[ht!]
\centering
\includegraphics[width=.9\textwidth]{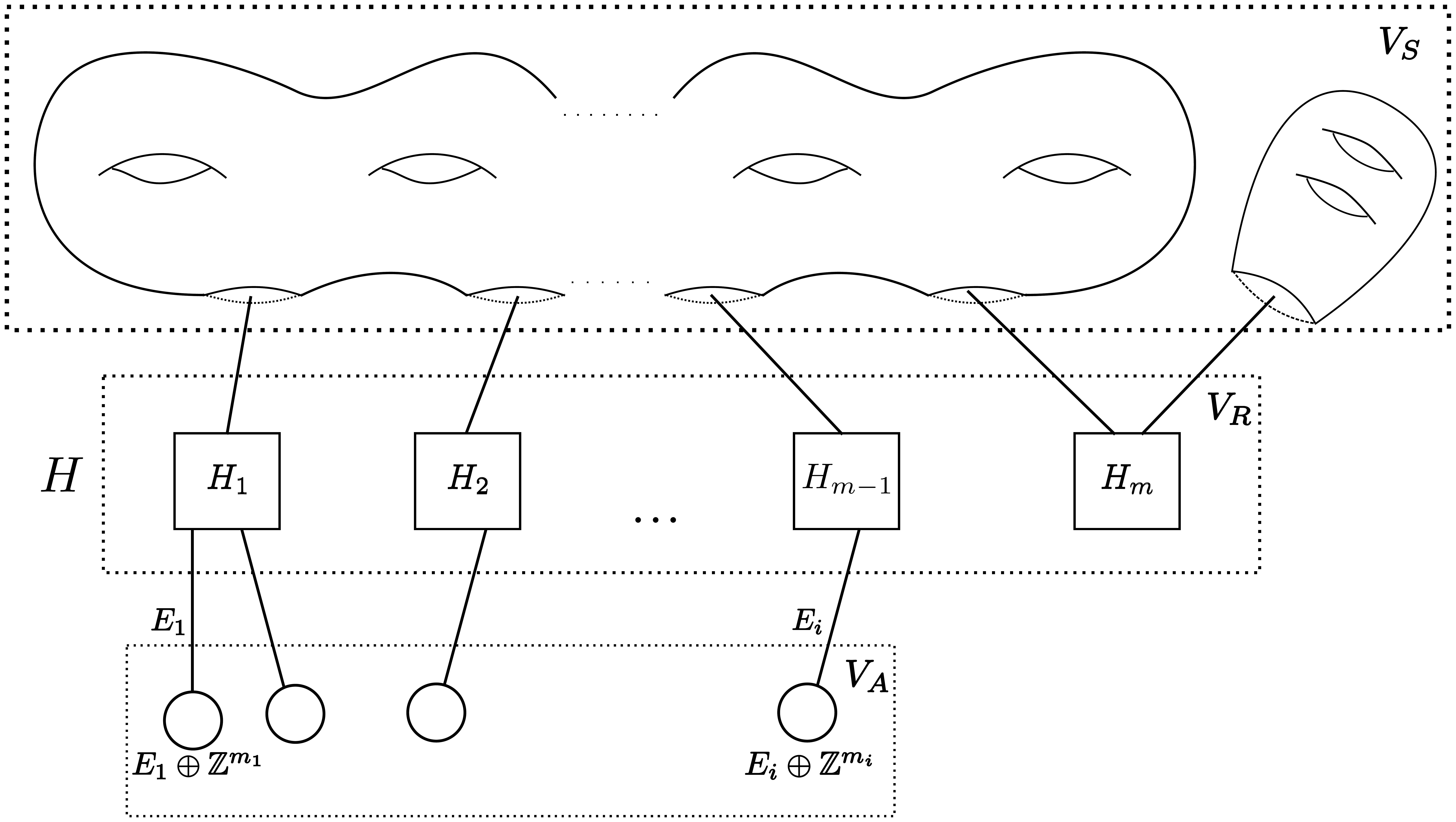}
\caption{A floor together with the partition of its vertices}
\end{figure}


In the opposite direction a floor can be {\em decomposed into flats} in many possible ways, i.e. a floor can be seen as 
a sequence of surface and abelian flats, and we will often see such a sequence as giving a preferred order to the flats of the floor. 

We can now bring everything together to define:

\begin{definition}
A group $G$ has the structure of a tower (of height $m$) over a subgroup $H$ if there 
exists a sequence $G=G^m>G^{m-1}>\ldots>G^0=H$ such that for each $i$, $0\leq i<m$, 
one of the following holds:
\begin{itemize}
 \item[(i)] the group $G^{i+1}$ has the structure of a floor over $G^i$, in which $H$ is contained in one of the vertex groups that generate $G_i$ in the 
 floor decomposition of $G^{i+1}$ over $G^i$. Moreover, the pegs of the abelian flats of the floor
 are glued along (maximal abelian) groups that are not conjugates of each other and 
 they cannot be conjugated into groups which correspond to abelian flats of any previous floor;
 \item[(ii)] the group $G^{i+1}$ is a free product of $G^{i}$ with a finitely generated free group.
\end{itemize}

\end{definition}

The next lemma follows from the definition of a constructible limit group.

\begin{lemma}
If $G$ has the structure of a tower over a limit group, then $G$ is a limit group.
\end{lemma}

If $G$ has the structure of a tower over a subgroup $H$ it will be useful to collect the information witnessing it, thus we 
define:

\begin{definition}
Suppose $G$ has the structure of a tower (of height $m$) over $H$. Then the tower corresponding to $G$, denoted by $\mathcal{T}(G,H)$, 
is the following collection of data:
$$((\mathcal{G}(G^1,G^0),r_1),(\mathcal{G}(G^2,G^1),r_2),\ldots,(\mathcal{G}(G^m,G^{m-1}),r_m))$$
where:
\begin{itemize}
 \item the splitting $\mathcal{G}(G^{i+1},G^i)$ is the splitting that witnesses that $G^{i+1}$ has the structure of a floor over $G^i$ (respectively the 
  free splitting $G^{i}*\F_n$ for some finitely generated free group $\F_n$);
  \item the morphism $r_{i+1}:G^{i+1}\rightarrow G^i$ (or $r_{i+1}:G^{i+1}\rightarrow G^{i}*\Z$) is the retraction that witnesses 
  that $G^{i+1}$ has the structure of a floor over $G^{i}$ (respectively the retraction $r_{i+1}:G^i*\F_n\rightarrow G^i$).
\end{itemize}
\end{definition}

\begin{remark} The notation $\mathcal{G}(G)$ will refer to a splitting of $G$ as a graph of groups. The notation 
 $\mathcal{G}(G,H)$ will refer either to a free splitting of $G$ as $H*\F_n$ or to a splitting that 
 corresponds to a floor structure of $G$ over $H$.
\end{remark}

A tower in which no abelian flat occurs in some (any) decomposition of its floors into flats is called a {\em hyperbolic tower}. Furthermore, 
if a floor consists only of abelian flats we call it an {\em abelian floor}. 

\noindent
For the rest of the paper we assume the following:
\\ \\
\noindent
{\bf Convention:} Suppose $\mathcal{T}(G,\F)$ is a tower. Let $\{E_j\}_{j\in J}$  
be the collection of pegs that correspond to the abelian flats that occur along the floors of the tower. 
Let $\{E_j\}_{j\in J'}$, with $J'\subseteq J$, be the subcollection  
of the pegs that can be conjugated into a subgroup of the base floor $\F$, 
i.e. there is $\gamma_j\in G$ such that $E_j^{\gamma_j}\leq \F$ for every $j\in J'$. 

Then, we assume that:
\begin{enumerate}
 \item when the above subcollection is not empty, the first floor $\mathcal{G}(G^1,G^0)$ of the tower $\mathcal{T}(G,\F)$ 
  consists only of the abelian flats corresponding to the above subcollection and glued along $E_j^{\gamma_j}$ 
  to $\F$, and each floor above the first (abelian) floor is either a free product or it consists of a single flat (abelian or surface);
 \item when the above subcollection is empty, we assume that each floor is either a free product or it consists of a single flat (abelian or surface). 
\end{enumerate}

\subsection{Twin towers}

We next work towards constructing a tower by ``gluing'' two copies of a given tower together.  

\begin{definition}\label{FloorDouble}
Suppose $H$ has the structure of an abelian floor over $\F$ and $\mathcal{G}(H,\F)$ is the splitting witnessing it. 
Let $\{\Z^{m_i}\}_{i\in I}$ be the collection of the free abelian groups that we glue along the corresponding pegs $\{E_i\}_{i\in I}$ 
in forming the abelian flats of the floor.

Then the double of $H$ with respect to $\mathcal{G}(H,\F)$, denoted by $H_{Db}:=H_{Db}(\mathcal{G}(G,\F))$, is the group obtained 
as the fundamental group of a graph of groups in which 
all the data is as in $\mathcal{G}(H,\F)$ apart from replacing $\{\Z^{m_i}\}_{i\in I}$ 
by their doubles $\{\Z^{m_i}\oplus\Z^{m_i}\}_{i\in I}$. The above graph of groups $Db(\mathcal{G}(H,\F)):=\mathcal{G}(H_{Db},\F)$ 
is called the floor double of $\mathcal{G}(H,\F)$ and naturally witnesses that $H_{Db}$ is an abelian floor over $\F$.

\end{definition}

 \begin{figure}[ht!]
\centering
\includegraphics[width=.7\textwidth]{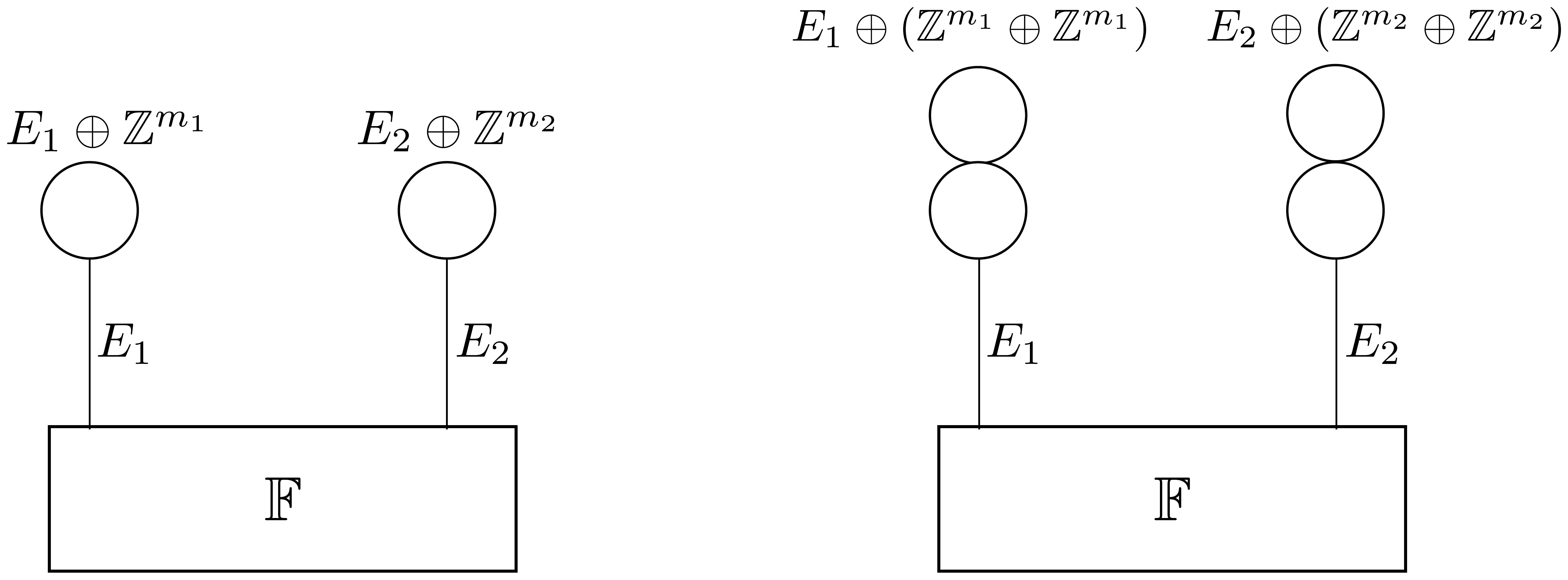}
\caption{An abelian floor over a free group and its double.}
\end{figure}

\begin{lemma}\label{DoubleEmbedd}
Suppose $H$ has the structure of an abelian floor 
$\mathcal{G}(H,\F)$ over $\F$. Then $H$ admits two natural embeddings $f_1,f_2$ into $H_{Db}$. 

Moreover, for each $i\leq 2$, the group $H_{Db}$ admits an abelian floor structure over $f_i(H)$. 
\end{lemma}
\begin{proof} 
The first embedding $f_1$ can be taken to be the identity since $H$ is a subgroup of $H_{Db}$ and clearly 
$H_{Db}$ has an abelian floor structure over $H$ with the pegs corresponding to the maximal abelian groups of $H$ that 
contain the pegs of $\mathcal{G}(H,\F)$. 
The second embedding $f_2$ is obtained as follows: 
\begin{itemize}
\item it agrees with $Id$ on $\F$; and
\item it sends each free abelian group $\Z^m$ that is glued along a peg $E$ in $\F$ in forming the abelian flats of the abelian floor $\mathcal{G}(H,\F)$ 
isomorphically onto the corresponding free abelian group glued along the peg in $H$ that contains $E$ 
in forming the abelian flats of the abelian floor $\mathcal{G}(H_{Db},H)$. 
\end{itemize} 

\end{proof}

The following lemmata are immediate.

\begin{lemma}\label{FreeProdDouble}
Let $H$ be a subgroup of a group $G$. Suppose $H$ has the structure of an abelian floor 
$\mathcal{G}(H,\F)$ over $\F$. Let $G_{Db}:=H_{Db}*_HG$.  
Let $\mathcal{G}(G)$ be a splitting of $G$ in which $H$ is a subgroup of a vertex group $G_v$. Then $G_{Db}$ 
is isomorphic to the fundamental group of the graph of groups that has the same data as $\mathcal{G}(G)$ apart 
from replacing $G_v$ by $(G_v)_{Db}:=H_{Db}*_HG_v$.
\end{lemma}

\begin{lemma}\label{MaxAb0}
Suppose $G$ has an abelian floor structure over a limit group $L$. Let $B_1, B_2$ be non conjugate (in $L$) maximal abelian subgroups of $L$ 
that cannot be conjugated into any of the pegs of the abelian floor. Then $B_1,B_2$ are non conjugate maximal abelian subgroups of $G$.
\end{lemma}

We now pass to proving that replacing the first (abelian) floor by its double yields a natural tower structure for the corresponding group.  

\begin{lemma}\label{ToweroverDouble}
Suppose $G$ has the structure of a tower (of height $m$)
$\mathcal{T}(G,\F):=((\mathcal{G}(G^1,G^0),$ $r_1),(\mathcal{G}(G^2,G^1),$ $r_2),\ldots,(\mathcal{G}(G^m,G^{m-1}),r_m))$ over $\F$ 
and that the first floor $\mathcal{G}(G^1,G^0)$ is an abelian floor.
Let, for each $i\leq m$, the group $G^i_{Db}$ be the amalgamated free product $G^1_{Db}*_{G^1}G^i$. 

Then $G_{Db}:=G^m_{Db}$ admits a structure of a tower over $\F$ 
witnessed by $G^m_{Db}>G^{m-1}_{Db}>\ldots> G^1_{Db}>\F$ and splittings 
$\mathcal{G}(G^{i+1}_{Db},G^i_{Db})$, which are naturally inherited from the corresponding splittings in $\mathcal{T}(G,\F)$.
\end{lemma}

\begin{proof}
The proof is by induction on the height $m$ of the tower. 
\\
{\bf Base step.} The group $G^1_{Db}$ has a natural abelian floor structure over $\F$ as observed in Definition \ref{FloorDouble}; 
\\
{\bf Inductive step.} We will assume that the result holds for any tower of height at most $i$ and we show it for towers of height $i+1$. 
We take cases according to whether $G^{i+1}$ is a free product over $G^i$, or has a surface flat 
structure over $G^i$, or has an abelian flat structure over $G^i$:
\begin{itemize}
   \item assume that $G^{i+1}=G^i*\F_n$, then, by Lemma \ref{FreeProdDouble}, $G^{i+1}_{Db}=G^i_{Db}*\F_n$. So, 
      by the induction hypothesis $G^{i+1}_{Db}$ has a tower structure over $\F$ 
      corresponding naturally to the tower structure of $G^{i+1}$ over $\F$;
      
   \item assume that $G^{i+1}$ has a surface flat structure over $G^i$ witnessed by $(\mathcal{G}(G^{i+1},G^i),r_{i+1})$. 
   We consider the rigid vertex group, say $G_u$, of the above graph of groups that contains $\F$. Since $G^1$ is 
   freely indecomposable with respect to $\F$, the vertex group $G_u$ must contain $G^1$. We consider the graph of groups 
   with the same data as $(\mathcal{G}(G^{i+1},G^i),r_{i+1})$ apart from replacing $G_u$ by $G_u*_{G^1}G^1_{Db}$. Then by Lemma \ref{FreeProdDouble} 
   the fundamental 
   group of this latter graph of groups is isomorphic to $G^{i+1}_{Db}$ and together with the retraction 
   $r_{i+1}': G^{i+1}_{Db}\rightarrow G^i_{Db}$ that agrees with $r_{i+1}$ on $G^{i+1}$ and stays the identity on $G^1_{Db}$ 
   they witness that $G^{i+1}_{Db}$ has a surface flat structure over $G^i_{Db}$;
   
   \item assume that $G^{i+1}$ has an abelian flat structure over $G^i$, i.e. $G^{i+1}=G^i*_A(A\oplus \Z^n)$. 
    We consider the splitting $(G^i*_{G^1}G^1_{Db})*_A(A\oplus\Z^n)$. By Lemma \ref{FreeProdDouble} this is 
    a splitting of $G^{i+1}_{Db}$. It is not hard to see that the group $A$ is maximal abelian in $G^i*_{G^1}G^1_{Db}$: 
    indeed since $A$ cannot be conjugated to any of the pegs of the first abelian floor $\mathcal{G}(G^1,\F)$, and 
    $A$ is maximal abelian in $G^i$ it must be maximal abelian in $G^i_{Db}$. 
    Moreover, by Lemma \ref{MaxAb0}, it cannot be conjugated to any other peg in $G^i_{Db}$. Thus, together with 
    the retraction $r_{i+1}':G^{i+1}_{Db}\rightarrow G^i_{Db}$ that agrees with $r_{i+1}$ on $G^{i+1}$ and stays 
    the identity on $G^1_{Db}$ they witness that $G^{i+1}$ has an abelian flat structure over $G^i_{Db}$.

      \end{itemize}
 
\end{proof}

Changing slightly the hypothesis of the previous lemma yields the following remark.

\begin{remark}\label{ToweroverfDouble}
Suppose $G$ has the structure of a tower (of height $m$)
$\mathcal{T}(G,\F):=((\mathcal{G}(G^1,G^0),$ $r_1),(\mathcal{G}(G^2,G^1),$ $r_2),\ldots,(\mathcal{G}(G^m,G^{m-1}),r_m))$ over $\F$ 
and that the first floor $\mathcal{G}(G^1,G^0)$ is an abelian floor.
Let $f:G^1\hookrightarrow G^1_{Db}$ be the non-identity natural map from $G^1$ to its double (see Lemma \ref{DoubleEmbedd}). 
Let, for each $i\leq m$, the group $_fG^i_{Db}$ be the amalgamated free product $G^1_{Db}*_{f(G^1)}G^i$. 

Then $_fG_{Db}:={_fG^m_{Db}}$ admits a structure of a tower over $\F$ 
witnessed by $_fG^m_{Db}>{_fG^{m-1}_{Db}}>\ldots>{_fG^1_{Db}}>\F$ and splittings 
$\mathcal{G}({_fG^{i+1}_{Db}},{_fG^i_{Db}})$, as in Lemma \ref{ToweroverDouble}.
\end{remark}

Before moving to the definition of a twin tower we record some easy lemmata that will help us 
prove that our construction of a twin tower is indeed a tower. 

\begin{lemma}\label{MaxAb1}
Suppose $G$ has the structure of a tower over a limit group $L$. Let $E$ be a maximal abelian subgroup of $G$ and 
suppose $E\cap L$ is not trivial. Let $B$ be the maximal abelian group in $L$ that contains $E\cap L$. 
Then either $E$ is $B$, or $E$ is  
the free abelian group $B\oplus \Z^n$ that corresponds to an abelian flat of some floor of the tower glued along $B$ to $L$. 
\end{lemma}

The following lemma is an easy exercise in normal forms.

\begin{lemma}\label{MaxAb2}
Let $G:=A*_CB$ be limit group and $E$ be a maximal abelian group in $A$. Suppose that no 
non trivial element of $E$ commutes with a non trivial element of $C$, then $E$ is maximal abelian in $G$.
\end{lemma}

We define the notion of a twin tower, first in a case which is free of some technical complexity, in the following proposition.

\begin{proposition}[Twin tower - non abelian case]
Suppose $G$ has the structure of a tower $\mathcal{T}(G,\F)$ over $\F$. Assume that the first floor 
$\mathcal{G}(G^1,G^0)$ is not an abelian floor. Then the amalgamated free product $G*_{\F}G$ 
admits a natural tower structure over $\F$ which we call the twin tower of $G$ with respect to $\mathcal{T}(G,\F)$.
\end{proposition}
\begin{proof}
Let $((\mathcal{G}(G^1,G^0),r_1),\ldots,(\mathcal{G}(G^m,G^{m-1}),r_m))$ be the sequence witnessing that $G$ 
is a tower over $\F$. Let $G^{m+i}:=G^i*_{\F}G$ be the amalgamated free product of $G^i$ with $G$ over $\F$. We claim 
that there exists a sequence 
$$((\mathcal{G}(G^1,G^0),r_1),\ldots,(\mathcal{G}(G^m,G^{m-1}),r_m), (\mathcal{G}(G^{m+1},G^m),r_{m+1}),\ldots, (\mathcal{G}(G^{2m},G^{2m-1}),r_{2m}))$$
where the splitting $\mathcal{G}(G^{m+i+1},G^{m+i})$ has the same data as the splitting 
$\mathcal{G}(G^{i+1},G^i)$ apart from replacing the vertex group $G_u$ that contains $\F$ with $G_u*_{\F}G$ and moreover it witnesses that 
$G*_{\F}G$ has the structure of a tower over $\F$. We proceed by induction:
\\
{\bf Base step.} We show that $G^{m+1}$ is a free product or has a surface flat structure over or has an abelian flat structure over $G^m=G$, according to whether $G^1$ 
 is a free product or has a surface flat structure $\F$. In addition, we show that it respects the requirements 
 of being a floor of a tower together with the already given sequence of floors. We take cases:
       \begin{itemize}
      \item assume that $G^1=\F*\F_n$, then $G^{m+1}=(\F*\F_n)*_{\F}G=G*\F_n$. Thus, $G^{m+1}$ has a free product structure over $G$;
      
      \item assume that $G^1$ has a surface flat structure over $\F$. We consider the graph of groups with the same data as in $\mathcal{G}(G^1,\F)$ apart from 
      replacing the vertex group $\F$ by the amalgamated free product $G*_{\F}\F$. Then the fundamental group of this graph of groups is isomorphic to $G^{m+1}$ 
      and together with the retraction $r_{m+1}:G^{m+1}\rightarrow G^m$ that agrees with $r_1$ on $G^1$ and stays the identity on $G$ they witness that 
      $G^{m+1}$ has a surface flat structure over $G^m$;
      
      \end{itemize}
{\bf Inductive step.} Assume that the result holds for all $G^{m+1},\ldots, G^{m+i}$. We will show that it holds for $G^{m+i+1}$. 
 We take cases according to whether $G^{i+1}$ is a free product or has a surface flat structure or an abelian flat structure over $G^i$:
      
      \begin{itemize}
      \item assume that $G^{i+1}=G^i*\F_n$, then $G^{m+i+1}=G^{i+1}*_{\F}G=G^i*_{\F}G*\F_n=G^{m+i}*\F_n$. Thus, $G^{m+i+1}$ 
      is a free product of $G^m$ with $\F_n$ and satisfies the conditions of being part of a tower with the already given sequence of floors;
      
      \item assume that $G^{i+1}$ has a surface flat structure $\mathcal{G}(G^{i+1},G^i)$ over $G^i$. Consider the graph of groups decomposition 
      with the same data as in $\mathcal{G}(G^{i+1},G^i)$ apart from replacing the vertex group $G_u$ that contains $\F$ with the amalgamated free 
      product $G_u*_{\F}G$. The fundamental group of this graph of groups is isomorphic to $G^{m+i+1}$ and together with the retraction $r_{m+i+1}:G^{m+i+1}\rightarrow G^{m+i}$ 
      that agrees with $r_{i+1}$ on $G^{i+1}$ and stays the identity on $G$ they witness that $G^{m+i+1}$ has a surface flat structure over $G^{m+i}$. 
      
      \item assume that $G^{i+1}$ has an abelian flat structure $G^i*_A(A\oplus\Z^n)$ over $G^i$. Consider the amalgamated free product $(G^i*_{\F}G)*_A(A\oplus\Z^n)$. 
      This is a splitting of $G^{m+i+1}$ and moreover, by Lemma \ref{MaxAb2}, $A$ is maximal abelian in $G^i*_{\F}G$. A maximal abelian group 
      of $G^{m+i}$ that contains a peg of 
      a previous abelian flat must either live in $G^i$ or in $G$. Now it is enough to observe that $A$ cannot be conjugated to a maximal abelian group of $G$ 
      and if a conjugate $A^{g}$ of $A$ intersects non-trivially another maximal abelian group of $G^i$, then $g\in G^i$. 
      In addition, we define the map $r_{m+i+1}:G^{m+i+1}\twoheadrightarrow G^{m+i}$ to agree with $r_{i+1}$ on $G^{i+1}$ and stay the identity on $G$.
      \end{itemize}

\end{proof}

 \begin{figure}[ht!]
\centering
\includegraphics[width=.8\textwidth]{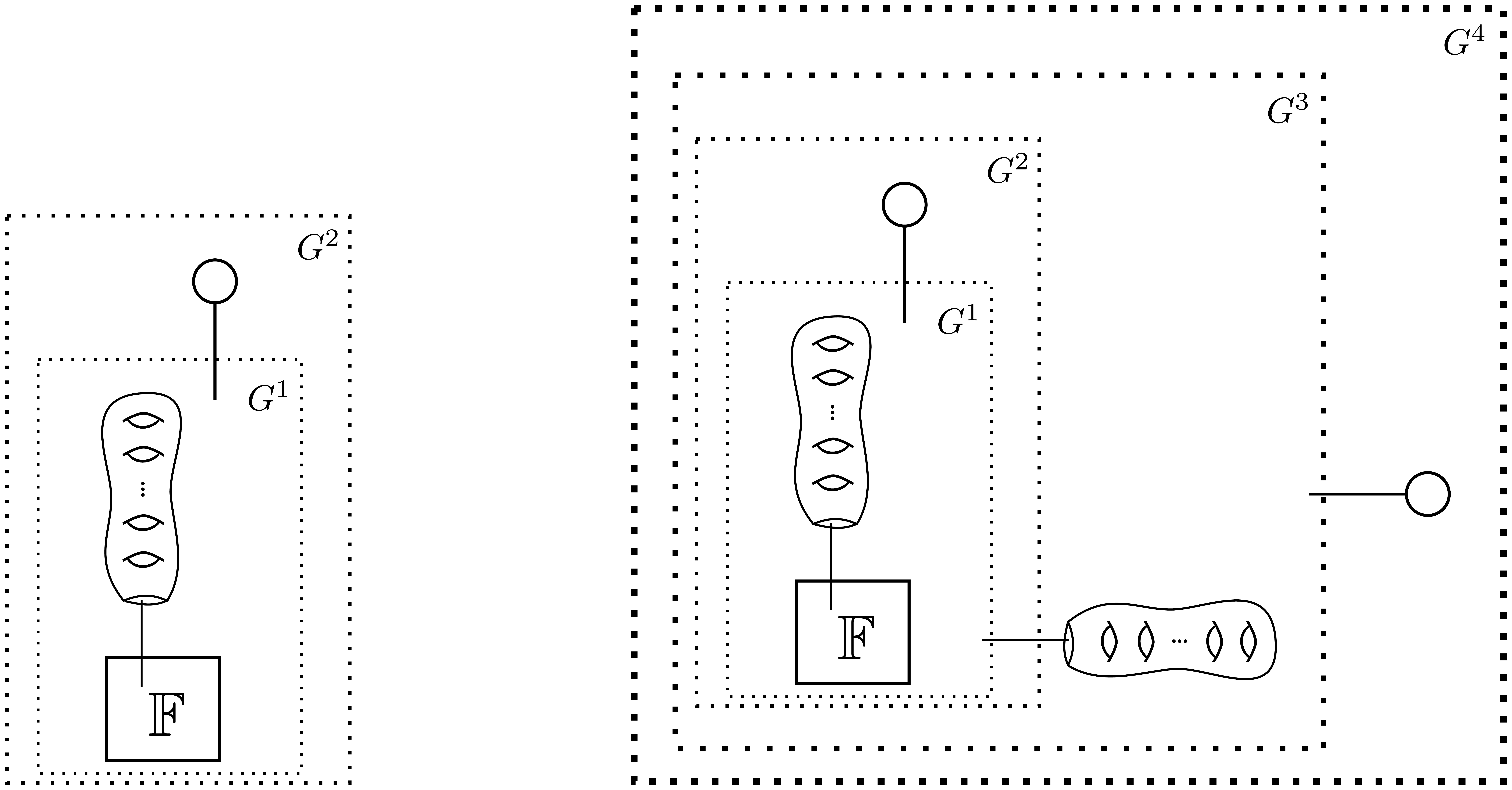}
\caption{A tower and its corresponding twin tower - non abelian case.}
\end{figure}

\begin{example}
Let $G:=\langle x_1,x_2,e_1,e_2,z_1,z_2 \ | \ [x_1,x_2]\cdot[e_1,e_2], [x_1^3x_2^4,z_1], [x_1^3x_2^4,z_1], [z_1,z_2] \rangle$. We can give $G$ 
a tower structure over $\F_2:=\langle e_1,e_2\rangle$ as follows. The tower consists of two floors:
\begin{itemize}
\item the first floor is just a surface flat  
that it is obtained by gluing $\Sigma_{1,1}$, whose fundamental group $\pi_1(\Sigma_{1,1})$ is $\langle x_1,x_2, s \ | \ s^{-1}\cdot[x_1,x_2]\rangle$, 
along its boundary to the subgroup $\langle [e_1,e_2]\rangle$ of the free group $\F_2$. In group theoretic terms the first floor 
is the amalgamated free product $G^1:=\F_2*_{[e_1,e_2]=s}\pi_1(\Sigma_{1,1})$. The retraction $r_1$ sends $x_i$ to $e_i$ and stays the identity on $e_i$;
\item the second floor is just an abelian flat that it is obtained by gluing $\Z^2$ along the (maximal) abelian subgroup $\langle x_1^3x_2^4\rangle$ of $G^1$. 
In group theoretic terms the second floor is the amalgamated free product $G^2:=G^1*_{x_1^3x_2^4=z}(\langle z\rangle\oplus\Z^2)$. The retraction $r_2$ 
sends $z_1$ and $z_2$ to $z$ and stays the identity on $G^1$.
\end{itemize}

The group $G*_{\F_2}G$ can be given a twin tower structure as follows. This tower has four floors, which we describe:
\begin{itemize}
 \item the first two floors are identical to the floors of the tower structure of $G$;
 \item the third floor is just a surface flat that it is obtained by gluing $\Sigma_{1,1}$, whose fundamental group 
 $\pi_1(\Sigma_{1,1})$ is $\langle y_1,y_2, s \ | \ s^{-1}\cdot[y_1,y_2]\rangle$, 
 along its boundary to the subgroup $\langle [e_1,e_2]\rangle$ of the free group $\F_2$. In group theoretic terms the third floor 
 is the amalgamated free product $G^3:=G^2*_{[e_1,e_2]=s}\pi_1(\Sigma_{1,1})$. The retraction $r_3$ sends $y_i$ to $e_i$ and stays the identity on $G^2$; 
 \item the fourth floor is just an abelian flat that it is obtained by gluing $\Z^2$ along the (maximal) abelian subgroup $\langle y_1^3y_2^4\rangle$ of $G^3$. 
 In group theoretic terms the fourth floor is the amalgamated free product $G^4:=G^3*_{y_1^3y_2^4=z}(\langle z\rangle\oplus\Z^2)$. The retraction $r_4$ 
 sends $z'_1$ and $z'_2$ to $z$ and stays the identity on $G^3$. 
\end{itemize}

\end{example}

\begin{proposition}[Twin tower - abelian case]\label{TwinAbel}
Suppose $G$ has the structure of a tower $\mathcal{T}(G,\F)$ (of height $m$) over $\F$. 
Assume that the first floor $\mathcal{G}(G^1,G^0)$ is an abelian floor. 

Let $G^1_{Db}$ be the double of $G^1$ with respect to $\mathcal{G}(G^1,G^0)$ and $f:G^1\hookrightarrow G^1_{Db}$ be the non-identity 
natural embedding of $G^1$ into $G^1_{Db}$. 

Let $G_{Db}$ be the double of $G$ with respect to $\mathcal{G}(G^1,G^0)$,    
and $_fG_{Db}$ be the amalgamated free product $G^1_{Db}*_{G^1}G$, where $f_{\bar{e}}:G^1\rightarrow G^1_{Db}$ is $f$ and $f_e:G^1\rightarrow G$ is the identity map. 
Then the amalgamated free product $G_{Db}*_{G^1_{Db}}(_fG_{Db})$ 
admits a natural tower structure over $\F$. 
\end{proposition}
\begin{proof} 
Suppose that $((\mathcal{G}(G^1,G^0),r_1),(\mathcal{G}(G^2,G^1),r_2),\ldots,(\mathcal{G}(G^m,G^{m-1}),r_m))$ 
witnesses that $G$ has the structure of a tower over $\F$.

Let 
$$\mathcal{T}(G_{Db},\F):=((\mathcal{G}(G^1_{Db},G^0),r^{Db}_1),(\mathcal{G}(G^2_{Db},G^1_{Db}),r^{Db}_2),\ldots,(\mathcal{G}(G^m_{Db},G^{m-1}_{Db}),r^{Db}_m))$$ 
be the natural tower structure (see Lemma \ref{ToweroverDouble}) of $G_{Db}$ over $\F$. And  

$$\mathcal{T}(_fG_{Db},G^1_{Db}):=((\mathcal{G}(_fG^2_{Db},G^1_{Db}),r^{f}_2),\ldots,(\mathcal{G}(_fG^m_{Db}, _fG^{m-1}_{Db}),r^{f}_m))$$
be the natural tower structure (see Remark \ref{ToweroverfDouble}) of $_fG_{Db}$ over $G^1_{Db}$. 

For each $i< m$, let $G^{m+i}_{Db}=G_{Db}*_{G^1_{Db}}(_fG^{i+1}_{Db})$. We claim 
that there exists a natural sequence of floors 
$((\mathcal{G}(G^1_{Db},G^0),r^{Db}_1),(\mathcal{G}(G^2_{Db},G^1_{Db}),r^{Db}_2),\ldots,(\mathcal{G}(G^m_{Db},G^{m-1}_{Db}),r^{Db}_m), 
(\mathcal{G}(G^{m+1}_{Db},G^m_{Db}),$ \ \\ $r^{Db}_{m+1}),\ldots,
(\mathcal{G}(G^{2m},G^{2m-1}),r^{Db}_{2m}))$
that witnesses that $G_{Db}*_{G^1_{Db}}(_fG_{Db})$ has a tower structure over $\F$. We construct this sequence starting with the $m$ floors of the 
tower $\mathcal{T}(G_{Db},\F)$ and we proceed recursively as follows:  
\\
{\bf Base step.} We show that $G_{Db}^{m+1}$ is a free product or has a surface flat structure or an abelian flat structure over $G_{Db}^m$  
 according to whether $_fG^2_{Db}$ is a free product or has a surface flat structure or an abelian flat structure over $G^1_{Db}$: 
      \begin{itemize}
      \item assume that $_fG^2_{Db}=G^1_{Db}*\F_n$, then $G^{m+1}_{Db}=G_{Db}*_{G^1_{Db}}(G^1_{Db}*\F_n)=G_{Db}*\F_n$;
      
      \item assume that $_fG^2_{Db}$ has a surface flat structure $(\mathcal{G}(_fG^2_{Db},G^1_{Db}),r^{f}_2)$ over $G^1_{Db}$. 
      We consider the graph of groups with the same data as $\mathcal{G}(_fG^2_{Db},G^1_{Db})$ apart from replacing the 
      vertex group $G^1_{Db}$ with the group $G_{Db}$. The fundamental group of this latter graph of groups is $G^{m+1}_{Db}$ 
      and together with the retraction $r_{m+1}^{Db}:G^{m+1}_{Db}\rightarrow G_{Db}$ that agrees with $r_2^f$ on $_fG^2_{Db}$ 
      and stays the identity on $G_{Db}$ they witness that $G^{m+1}_{Db}$ has a surface flat structure over $G_{Db}$;
      
      \item assume that $_fG^2_{Db}$ has an abelian flat structure $G^1_{Db}*_A (A\oplus\Z^n)$ over $G^1_{Db}$. Consider 
      the amalgamated free product $(G_{Db}*_{G^1_{Db}}G^1_{Db})*_A(A\oplus\Z^n)$: this is a splitting of $G^{m+1}_{Db}$. 
      By definition $A$ cannot be conjugated to any other peg of some abelian flat of $\mathcal{T}(G_{Db},\F)$, 
      thus it is maximal abelian in $G_{Db}$ and satisfies the properties that make $G_{Db}*_A(A\oplus\Z^n)$ the $m+1$-th 
      floor of our tower;  
      \end{itemize}
{\bf Recursive step.} Suppose we have constructed the $m+i$-th floor of the tower. We show that $G_{Db}^{m+i+1}$ is a free product 
or has a surface flat structure or an abelian flat structure over $G_{Db}^{m+i}$  
according to whether $_fG^{i+2}_{Db}$ is a free product or has a surface flat structure or an abelian flat structure over $_fG^{i+1}_{Db}$:  
       \begin{itemize}
       \item assume that $_fG^{i+2}_{Db}=\ _fG^{i+1}_{Db}*\F_n$, then $G^{m+1}_{Db}=G_{Db}*_{G^1_{Db}}(G^{i+1}_{Db}*\F_n)=G^{m+i}_{Db}*\F_n$;
       
       \item assume that $_fG^{i+2}_{Db}$ has a surface flat structure $(\mathcal{G}(_fG^{i+2}_{Db},\ _fG^{i+1}_{Db}),r^{f}_{i+2})$ over $_fG^{i+1}_{Db}$. 
      We consider the graph of groups with the same data as $\mathcal{G}(_fG^{i+2}_{Db},G^{i+1}_{Db})$ apart from replacing the 
      vertex group $G_u$ that contains $G^1_{Db}$ with the group $G_{Db}*_{G^1_{Db}}G_u$. The fundamental group of this latter graph of groups is $G^{m+i+1}_{Db}$ 
      and together with the retraction $r_{m+i+1}^{Db}:G^{m+i+1}_{Db}\rightarrow G^{m+i}_{Db}$ that agrees with $r_{i+2}^f$ on $_fG^{i+2}_{Db}$ 
      and stays the identity on $G_{Db}$ they witness that $G^{m+i+1}_{Db}$ has a surface flat structure over $G^{m+i}_{Db}$;
      
      \item assume that $_fG^{i+2}_{Db}$ has an abelian flat structure $_fG^{i+1}_{Db}*_A(A\oplus\Z^n)$ over $_fG^{i+1}_{Db}$. 
      We consider the amalgamated free product $(G_{Db}*_{G^1_{Db}}(_fG^{i+1}_{Db}))*_A(A\oplus\Z^n)$: this is a splitting 
      of $G^{m+i+1}_{Db}$. Since $A$ cannot be conjugated to any other previous peg of the tower $\mathcal{T}(G_{Db},\F)$ 
      and the tower $\mathcal{T}(_fG_{Db},G^1_{Db})$ we see that $A$ is maximal abelian in 
      $G_{Db}*_{G^1_{Db}}(_fG^{i+1}_{Db})$ and it satisfies the properties that make $G^{m+i}_{Db}*_A(A\oplus\Z^n)$ the $m+i+1$-th 
      floor of our tower;
       \end{itemize}

\end{proof}

\begin{example}
Let $G:=\langle e_1,e_2,z_1,z_2,x_1,x_2 \ | \ [z_1,z_2], [z_1,e_1^2e_2^2],[z_2,e_1^2e_2^2], [x_1,x_2]\cdot[z_1,e_1]\rangle$. 
We can give $G$ a tower structure over $\F_2:=\langle e_1,e_2\rangle$ as follows. The tower $\mathcal{T}(G,\F_2)$ consists of two floors:
\begin{itemize}
 \item the first floor is just an abelian flat obtained by gluing $\Z^2$ 
 along the maximal abelian group $\langle e_1^2e_2^2 \rangle$ of $\F_2$. In 
 group theoretic terms, the group corresponding to the first floor is the 
 amalgamated free product $G^1:=\F_2*_{e_1^2e_2^2=z}(\langle z\rangle\oplus\Z^2)$. The 
 retraction $r_1:G^1\twoheadrightarrow\F_2$ sends $z_1$ and $z_2$ to $z$ and stays the identity on $\F_2$.
 
 \item the second floor is just a surface flat obtained by gluing $\Sigma_{1,1}$, whose fundamental 
 group is $\langle x_1,x_2,s \ | \ s^{-1}[x_1,x_2]\rangle$, along its boundary to 
 the subgroup $\langle [z_1,e_1]\rangle$ of $G^1$. In group theoretic terms the group 
 corresponding to the second floor is the amalgamated free product $G^2:=G^1*_{[z_1,e_1]=s}\pi_1(\Sigma_{1,1})$. 
 The retraction $r_2:G^2\twoheadrightarrow G^1$ sends $x_1$ to $z_1$, $x_2$ to $e_1$ and stays the identity on $G^1$.
\end{itemize}

We now consider the double of $G^1$ with respect to the splitting $\mathcal{G}(G^1,\F)$ of the first point above. 
As a group $G^1_{Db}$ has the following presentation $\langle e_1,e_2,z_1,z_2,y_1,y_2 \ | 
\ [z_i,y_j] \ i,j\leq 2, [z_1,e_1^2e_2^2],[z_2,e_1^2e_2^2], [y_1,e_1^2e_2^2],[y_2,e_1^2e_2^2] \rangle$. 
It can be seen as the amalgamated free product $G^1_{Db}:=\F_2*_{e_1^2e_2^2=z}(\langle z\rangle\oplus\Z^4)$. 

The twin tower that corresponds to $\mathcal{T}(G,\F_2)$ consists of three floors as follows: 
\begin{itemize}
 \item the first floor is the floor double of $\mathcal{G}(G^1,\F_2)$ and the group corresponding to this floor is $G^1_{Db}$. 
 The retraction $r_1^{Db}$ sends each $z_1,z_2,y_1,y_2$ to $z$ and stays the identity on $\F_2$. 
 \item the second floor is just a surface flat obtained by gluing $\Sigma_{1,1}$, whose fundamental 
 group is $\langle x_1,x_2,s \ | \ s^{-1}[x_1,x_2]\rangle$, along its boundary to 
 the subgroup $\langle [z_1,e_1]\rangle$ of $G^1_{Db}$. In group theoretic terms the group 
 corresponding to the second floor is the amalgamated free product $G^2_{Db}:=G^1_{Db}*_{[z_1,e_1]=s}\pi_1(\Sigma_{1,1})$. 
 The retraction $r_2^{Db}:G^2_{Db}\twoheadrightarrow G^1_{Db}$ sends $x_1$ to $z_1$, $x_2$ to $e_1$ and stays the identity on $G^1_{Db}$.
 \item the third floor is again a surface flat obtained by gluing $\Sigma_{1,1}$, whose fundamental 
 group is $\langle p_1,p_2,s \ | \ s^{-1}[p_1,p_2]\rangle$, along its boundary to 
 the subgroup $\langle [y_1,e_1]\rangle$ of $G^2_{Db}$. In group theoretic terms the group 
 corresponding to the second floor is the amalgamated free product $G^3_{Db}:=G^2_{Db}*_{[y_1,e_1]=s}\pi_1(\Sigma_{1,1})$. 
 The retraction $r_3^{Db}:G^3_{Db}\twoheadrightarrow G^2_{Db}$ sends $p_1$ to $y_1$, $p_2$ to $e_1$ and stays the identity on $G^2_{Db}$.
\end{itemize}
The group corresponding to the twin tower has presentation $\langle e_1,e_2,z_1,z_2,y_1,y_2,x_1,x_2,p_1,p_2 \ | 
\ $ $[z_i,y_j] \ i,j\leq 2, [z_1,e_1^2e_2^2],[z_2,e_1^2e_2^2], [y_1,e_1^2e_2^2],[y_2,e_1^2e_2^2], [x_1,x_2]\cdot[z_1,e_1],[p_1,p_2]\cdot[y_1,e_1] \rangle$
\end{example}

 \begin{figure}[ht!]
\centering
\includegraphics[width=.8\textwidth]{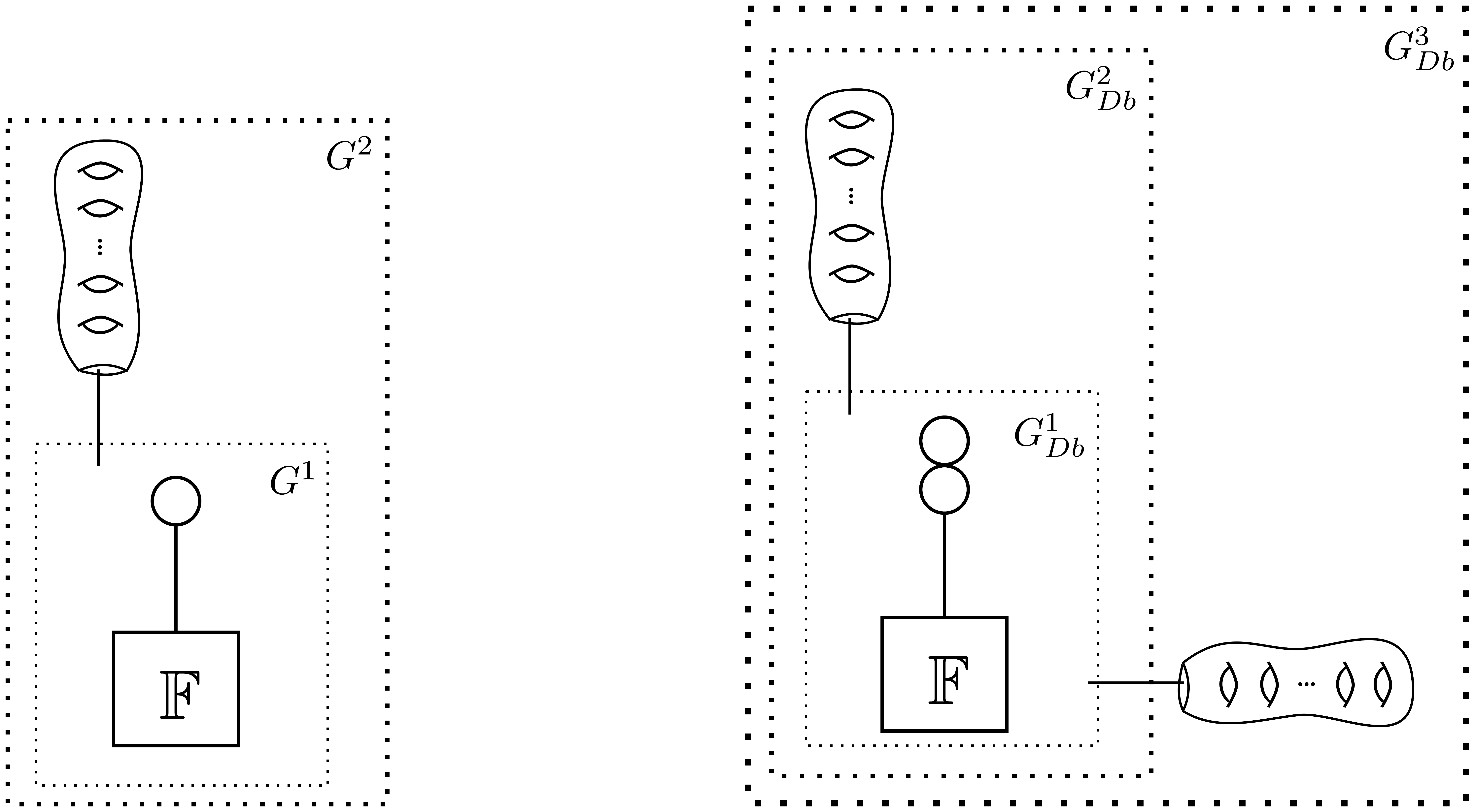}
\caption{A tower and its corresponding twin tower - abelian case.}
\end{figure}

\subsection{Closures of towers}
We pass to the notion of a tower closure. We first define the notion of an abelian floor closure
\begin{definition}[Abelian floor closure]
Suppose $G$ has the structure of an abelian floor over a limit group $L$ and $\mathcal{G}(G,L)$ is the splitting witnessing it. 
Let $\{\Z^{m_i}\}_{i\in I}$ be the collection of the free abelian groups 
that we glue along the corresponding pegs $\{E_i\}_{i\in I}$ in forming the abelian flats of the floor. 

Let $\{A^{m_i}\}_{i\in I}$ be free abelian groups and $f_i:E_i\oplus\Z^{m_i}\rightarrow E_i\oplus A^{m_i}$ be 
an embedding with $f_i\upharpoonright E_i=Id$ and such that $f_i(E_i\oplus\Z^{m_i})$ is a finite index subgroup of $E_i\oplus A^{m_i}$,  
for every $i\in I$. We call $\{f_i\}_{i\in I}$ a family of closure embeddings. 

We denote by $cl(G)$ the group that is the fundamental group of a graph of groups in which 
all the data is as in $\mathcal{G}(G,L)$ apart from replacing $\{\Z^{m_i}\}_{i\in I}$ 
by $\{\Z^{m_i}\oplus A^{m_i}\}_{i\in I}$ and add to the vertex groups $\{E_i\oplus\Z^{m_i}\oplus A^{m_i}\}_{i\in I}$ the 
relations corresponding to the 
family of closure embeddings, i.e. $z=f_i(z)$ for every $z\in E\oplus\Z^{m_i}$. Moreover, we call the 
latter graph of groups, $\mathcal{G}(cl(G),L)$, the floor closure of $\mathcal{G}(G,L)$ with respect to $\{f_i\}_{i\in I}$.
\end{definition}

\begin{remark}\label{MorphismExtensions}
Let $A^{m+1}:=\langle c, a_1,\ldots, a_m\rangle$, $\Z^{m+1}:=\langle c, z_1,\ldots, z_m\rangle$ be free abelian groups of rank $m+1$. 
Let $f:\Z^{m+1}\rightarrow A^{m+1}$ be an embedding for which $f(c)=c$ and $f(\Z^{m+1})$ has finite index in $A^{m+1}$. Then, 
to such an embedding we can assign a system of ($m$) linear equations (with $2m$ unknowns) over the integers as follows: 
for each $i\leq m$, if $f(z_i)=c^{k_{i_1}}a_1^{k_{i_2}}\cdots a_m^{k_{i_{m+1}}}$, then $x_i=k_{i+1}+k_{i_2}y_1+\ldots+k_{i_{m+1}}y_m$. 
We denote this system of equations $\Sigma_{f}(x_1,\ldots,x_m,y_1,\ldots,y_m)$.

Furthermore, for any such system of equations $\Sigma(x_1,\ldots,x_m,y_1,\ldots,y_m)$ there exists a 
finite index subgroup $U$ of the group $\mathbb{A}^m:=\langle(1,0,\ldots,0),\ldots,(0,0,\ldots,1)\rangle$ 
and a tuple of integers $(k_1,\ldots,k_m)$ so that for each tuple of integers $(p_1,\ldots,p_m)$ 
the system of linear equations over the integers with $m$ unknowns $\Sigma(p_1,\ldots,p_m,y_1,\ldots,y_m)$ 
has an integer solution if and only if $(p_1,\ldots,p_m)\in (k_1,\ldots,k_m)+U$.

On the other hand, whenever we have a finite index subgroup $U$ of $\mathbb{A}^m$, a coset $(k_1,\ldots,k_m)+U$ 
defines a system of $m$ linear equations with $2m$ unknowns, $\Sigma(x_1,\ldots,x_m,y_1,\ldots,y_m)$, such that 
the system of equations $\Sigma(p_1,\ldots,p_m,y_1,\ldots,y_m)$ has an integer solution if and only if $(p_1,\ldots,p_m)\in (k_1,\ldots,k_m)+U$. 
In turn, this system of equations defines an embedding $f:\Z^{m+1}\rightarrow A^{m+1}$ for which $f(c)=c$ and 
$f(\Z^{m+1})$ has finite index in $A^{m+1}$.
\end{remark}

The above remark characterizes the morphisms $h:G\rightarrow\F$ from a group $G$, that has a 
structure of a floor over a limit group $L$, which can be extended to morphisms $h':cl(G)\rightarrow\F$ 
from the closure of $G$ with respect to some closure embeddings $\{f_i\}_{i\in I}$.

\begin{example}
We consider the group $G:=\langle e_1,e_2, z \ | \ [z,e_1]\rangle$. The group $G$ admits 
the structure of an abelian flat over $\F_2:=\langle e_1,e_2\rangle$. This can be seen 
as the amalgamated free product $\F_2*_{e_1=c}(\langle c\rangle\oplus\langle z\rangle)$. 

Let $f:\langle c\rangle\oplus\langle z\rangle \rightarrow \langle c\rangle\oplus\langle a\rangle$ be the morphism 
that stays the identity on $c$ and sends $z$ to $c^2a^3$. Since $f$ is injective and $\langle c, c^2a^3\rangle$ is 
an index $3$ subgroup of $\langle c,a\rangle$ the conditions of Remark \ref{MorphismExtensions} are met. Moreover the 
corresponding system of equations is $\Sigma(x,y):= x = 2 +3y$ and $\Sigma(p,y)$ has integer solutions if and only if 
$p\in 2+3\mathbb\Z$.

In addition, the closure of $G$ with respect to $f$ can be seen as the amalgamated free product 
$cl(G):=\F_2*_{e_1=c}\langle c,z,a \ | \ [c,z], [c,a], [z,a], z=c^2a^3\rangle$. 
By our previous argument, a morphism $h:G\rightarrow \F_2$ (that stays the identity on $\F_2$) extends to a morphism $h':cl(G)\rightarrow\F_2$ 
if and only if it sends $z$ to $e_1^{2+3k}$, for some $k\in\Z$.
\end{example}

\begin{remark}
Suppose $G$ has the structure of a floor over a limit group $L$. Suppose $\{\Z^{m_i}\}_{i\in I}$ is the collection of the free abelian groups 
that we glue along the corresponding pegs $\{E_i\}_{i\in I}$ in forming the abelian flats of the floor. 
Let $(\mathcal{G}(cl(G),L), \{f_i\}_i\in I)$ be the closure of $G$ with respect to 
some family of closure embeddings. Let $G=G^1*\ldots *G^m$ be a free splitting of $G$. Then for each $i\in I$ 
we have that $\gamma_i E_i\oplus \Z^{m_i}\gamma_i^{-1}$ for some $\gamma_i\in G$, is a subgroup of some $G_j$, for $j\leq m$.

Thus, there exists a free splitting, $H^1*\ldots *H^m$ of $cl(G)$ such that each $H^i$ is the group 
obtained by gluing, $\gamma_iA^{m_i}\gamma_i^{-1}$ in $G^i$, along each maximal abelian group of the form 
$\gamma_i E_i\oplus \Z^{m_i}\gamma_i^{-1}$ that is contained in $G^i$ and adding the relations in this 
new vertex group according to the family of the closure embeddings, i.e. 
$\gamma_i z\gamma_i^{-1}=\gamma_if_i(z)\gamma_i^{-1}$ for every $z\in E_i\oplus\Z^{m_i}$.
\end{remark}

\begin{definition}[Tower closure]
Let $\mathcal{T}(G,\F)$ be a tower of height $m$. Let, for $i<m$, $\mathcal{G}(G^{i+1},G^i)$ be the $i$-th floor, and 
$\{f^i_j\}_{j\in J}$ a family of closure embeddings (in the case where the $i$-th floor is a free product or a hyperbolic floor we take the family to be empty). 
Then the tower closure, $cl(\mathcal{T}(G,\F))$, of the tower 
$\mathcal{T}(G,\F)$ with respect to the previous families of closure embeddings for each floor of the tower,
is defined recursively as follows:
\\
{\bf Base step.} The first floor consists of the floor closure $\mathcal{G}(cl(G^1),\F)$ with respect to $\{f^1_j\}_{j\in J}$;\\
{\bf Recursive step.} Let $G^i_{cl}$ be the group that corresponds to the $i$-th floor of the tower closure.
Then $G^{i+1}_{cl}$ is the fundamental group of the graph of groups in which:
    \begin{itemize}
     \item the underlying graph is the same as the underlying graph of $\mathcal{G}(G^{i+1},G^{i})$;
     \item the vertex groups, $G^i_1,\ldots,G^i_m$ whose free product 
     is $G^i$ are replaced by the corresponding groups in $G^i_{cl}$;
     \item the abelian flats $\Z^{m_j}$ are replaced by $A^{m_j}\oplus\Z^{m_j}$; and
     \item in the new abelian vertex groups $E_j\oplus A^{m_j}\oplus\Z^{m_j}$ we add relations 
     according to the family of closure embeddings, i.e. $z=f^{i+1}_j(z)$ for every $z\in E_j\oplus\Z^{m_j}$. 
    \end{itemize}
\end{definition}
It is not hard to see that.
\begin{lemma}
Let $\mathcal{T}(G,\F)$ be a tower over $\F$. Then $cl(\mathcal{T}(G,\F))$ with respect to any families of closure embeddings is a tower over $\F$.
\end{lemma}

\subsection{Symmetrizing closures of twin towers} 
When moving to the closure of a twin tower it could be the case that the ``twin'' abelian flats that appear 
in the floors of the twin tower embed in different ways in the ambient free abelian free groups. For example, 
if $E\oplus\langle z\rangle$ is an abelian flat and $\hat{E}\oplus\langle \hat{z}\rangle$ is its twin, 
then we could have the closure embeddings: $f:E\oplus\langle z\rangle\rightarrow E\oplus\langle a\rangle$ with $f(z)=a^2$, 
and $\hat{f}:\hat{E}\oplus\langle \hat{z}\rangle\rightarrow\hat{E}\oplus\langle b\rangle$ with $\hat{f}(\hat{z})=b^3$. 
We would like to ``symmetrize'' the situation by moving to a closure of our given closure demanding that the 
closure embeddings will be: $f_1:E\oplus\langle a\rangle\rightarrow E\oplus\langle s\rangle$ with $f_1(a)=s^3$, and 
$\hat{f}_1:\hat{E}\oplus\langle b\rangle\rightarrow\hat{E}\oplus\langle q\rangle$ with $\hat{f}_1(b)=q^2$. 
The main point is that the image of $E\oplus\langle z\rangle$ in $E\oplus\langle s\rangle$ under $f_1\circ f$ can be 
identified with the image of $\hat{E}\oplus\langle \hat{z}\rangle$ in $\hat{E}\oplus\langle q\rangle$ under $\hat{f}_1\circ f$. 

We define the symmetric closure of a closure of a twin tower as follows. 

\begin{definition}
Suppose $G$ has the structure of a tower (of height $m$) $\mathcal{T}(G,\F)$ 
and that the first floor $\mathcal{G}(G^1,G^0)$ is an abelian floor. Let $\mathcal{T}\#\mathcal{T}(G,\F)$ 
be the twin tower of $\mathcal{T}(G,\F)$. For each $1< i\leq m$, we call the the $m+i-1$ the twin floor of the $i$-th floor.

Let $cl(\mathcal{T}\#\mathcal{T}(G,\F))$ be a closure with respect to some closure embeddings $\{(f_i,\hat{f}_i)_{i\in I}\}$ 
where $\hat{f}$ is the embedding defined on the abelian flat that corresponds to the twin of the abelian flat where $f$ is defined on. 
For each abelian flat $E\oplus\Z^m$ and its twin $\hat{E}\oplus\hat{Z}^m$ we consider the cosets $p+U$ and $\hat{p}+\hat{U}$, 
where $U,\hat{U}$ are finite index subgroups of $\langle(1,0,\ldots,0),\ldots, (0,0,\ldots,1)\rangle$, that correspond to the 
closure embeddings $f,\hat{f}$. For each such couple $p+U, \hat{p}+\hat{U}$, we consider the cosets $p+U\cap\hat{U}, \hat{p}+U\cap\hat{U}$ 
and the induced corresponding embeddings defined on $E\oplus\Z^m, \hat{E}\oplus\hat{\Z}^m$.

We call the closure of $\mathcal{T}\#\mathcal{T}(G,\F)$ obtained by the above closure embeddings the symmetric closure 
of $cl(\mathcal{T}\#\mathcal{T}(G,\F))$ with respect to $\{(f_i,\hat{f}_i)_{i\in I}\}$. 
\end{definition}

\begin{lemma}
The symmetric closure of the closure of a twin tower is a closure of its closure. Moreover, 
for each abelian flat $E\oplus\Z^m$ and its twin $\hat{E}\oplus\hat{\Z}^m$ the corresponding 
subgroups $U,\hat{U}<\langle (1,0,\ldots,0),\ldots, (0,0,\ldots,1)\rangle$, obtained by the  
closure embeddings $f,\hat{f}$, are the same, i.e. $U=\hat{U}$.
\end{lemma}

\section{Solid limit groups and strictly solid morphisms}\label{Completions}
In this section we record the definitions of a strictly solid morphism and a family of such morphism as given by Sela. 
A strictly solid morphism is a morphism from a solid limit group to a free group, that satisfy certain conditions. These 
morphisms are of fundamental importance in the work of Sela to answering Tarski's question: first because of a boundedness result (see \cite[Theorem 2.9]{Sel3}), and second because 
in contrast to solid morphisms they are first order definable.

The definition of a strictly solid morphism requires a technical construction, 
called the {\em completion} of a strict map. The next subsection explains this construction. 

In subsection \ref{SldMorph} we define the above special class of morphisms and their families.    

\subsection{Completions}

We start by modifying a $GAD$ for a limit group $G$ in order to simplify the conditions 
for a map $\eta:G\twoheadrightarrow L$ to be strict with respect to it. The goal is to transform the $GAD$ in a way that 
the rigid vertex groups will be enlarged to their envelopes, every edge group connecting two rigid vertex groups would be 
maximal abelian in both vertex groups  
of the one edged splitting induced by its edge (and after replacing all abelian vertex groups with their peripheral subgroups), 
and abelian vertex groups will be leaves connected through a rigid vertex 
to the rest of the graph.

The following lemma of Sela will be helpful. 

\begin{lemma}[Sela]
Let $L$ be a limit group and $M$ be a non-cyclic maximal abelian subgroup. Then:
\begin{itemize}
 \item if $L$ admits an amalgamated free product splitting with abelian edge group, then $M$ can be conjugated into one of the factors;
 \item if $L$ admits an $HNN$ extension splitting $A*_C$ where $C$ is abelian, then either $M$ can be conjugated into $A$ or 
 $M$ can be conjugated in $\langle C,t\rangle$ and $L=A*_C\langle C,t\rangle$.
 
\end{itemize}

\end{lemma}

\begin{lemma}\label{ModificationGAD}
Let $G$ be a limit group and $\Delta:=(\mathcal{G}(V,E),(V_S,V_A,V_R))$ be a $GAD$ for $G$, where the image of each edge group 
is maximal abelian in at least one vertex group of the one edged splitting induced by the edge. Assume moreover that $V_A$ is empty. 
Then there exists a $GAD$ $\hat{\Delta}:=(\hat{\mathcal{G}}(\hat{V},\hat{E}),(\hat{V}_S, \hat{V}_A, \hat{V}_R))$ for $G$ satisfying the following properties:
\begin{enumerate}
 \item the underlying graph $\hat{\mathcal{G}}$ is the same as $\mathcal{G}$ up to some sliding of edges;
 \item the set of rigid vertices and the set of surface type vertices are the same; 
 \item every rigid vertex group of the graph of groups $\hat{\mathcal{G}}$ coincides with the envelope 
 of the corresponding rigid vertex group in $\mathcal{G}$;
 \item the image of every edge group connecting two rigid vertices of $\hat{V}_R$ is maximal abelian in both vertex groups 
 in the splitting induced by the edge. 
\end{enumerate}
\end{lemma}

\begin{remark} Suppose $\eta:G\rightarrow L$ be a strict map with respect to the $GAD$ $\Delta$ for $G$. We consider the following modification of 
 $\Delta$:
\begin{itemize}
 \item[Step 1] we replace every vertex group $G_u$ with $u\in V_A$ by its peripheral subgroup and we place $u$ in the set $V_R$ (i.e. we consider it rigid), we call 
 this $GAD$ $\Delta_0$;
 \item[Step 2] we modify $\Delta_0$ according to Lemma \ref{ModificationGAD} in order to obtain $\hat{\Delta}_0$; 
 \item[Step 3] to every rigid vertex in $\hat{\Delta}_0$ whose vertex group was a peripheral subgroup in $\Delta_0$ we attach 
 an edge whose edge group is the peripheral subgroup itself and the vertex group on its other end is the abelian group that contained 
 the peripheral subgroup in $\Delta$, these new vertex groups will be abelian type vertex groups. We denote by $\hat{\Delta}$ this $GAD$ for $G$.   
\end{itemize}
We will either explicitly or implicitly use the above modification for the rest of the paper.  
\end{remark}

\begin{definition}\label{Completion}
Let $G$ be a group and $\mathcal{G}(G)$ be a $GAD$ for $G$ with $m$ edges and at least one rigid vertex. 
Let $\eta:G\twoheadrightarrow H$ and $(\mathcal{G}(G),\eta)$ be strict. 

Let $\mathcal{G}_0\subseteq\mathcal{G}_1\subseteq\ldots\subseteq\mathcal{G}_m:=\mathcal{G}(G)$ be a sequence of sub-graph of groups such that $\mathcal{G}_i$ 
has $i$ edges. We define the group $Comp_i$ together with its splitting $\mathcal{G}(Comp_i)$ by the following recursion:\\
\noindent
{\bf Base step.} The subgraph of groups $\mathcal{G}_0$ consists of a single vertex $V$ which we may assume rigid, 
then $Comp_0$ is the group $H$ and $\mathcal{G}(Comp_0)$ is the trivial splitting;
\\ \\
\noindent
{\bf Recursion step.} Let $e_{i+1}$ be the edge in $\mathcal{G}_{i+1}\setminus\mathcal{G}_i$. We take cases:
     \begin{itemize}
     \item[Case 1:] suppose $e_{i+1}$ connects two rigid vertices. We further take cases:
            \begin{itemize}
            \item[1A.] Assume that the centralizer of $\eta(G_{e_{i+1}})$ in $Comp_i$, cannot be conjugated  
            neither to the centralizer of $\eta(G_{e})$ for some edge $e$ in $\mathcal{G}_i$ that connects two rigid vertex groups 
            nor to the centralizer of the image of the peripheral subgroup of some abelian vertex group in $\mathcal{G}_i$. 
            Then $Comp_{i+1}$ is the fundamental group of the graph of groups 
            obtained by gluing to $\mathcal{G}(Comp_i)$ a free abelian flat of rank $1$ along the centralizer of $\eta(G_{e_{i+1}})$ in $Comp_i$. The 
            latter graph of groups is $\mathcal{G}(Comp_{i+1})$;
            \item[1B.] Assume that the centralizer of $\eta(G_{e_{i+1}})$ in $Comp_i$, can be conjugated to 
            the centralizer $C$ (in $Comp_i$) either of $\eta(G_{e})$ for some edge $e$ in $\mathcal{G}_i$ that connects two rigid vertex groups, or 
            of the image of the peripheral group of some abelian vertex group (i.e. a vertex group whose vertex belongs to $V_A$) in $\mathcal{G}_i$. Then 
            $Comp_i$ is the fundamental group of the graph of groups obtained by gluing to $\mathcal{G}(Comp_i)$ a free abelian 
            flat of rank $1$ along $C$. 
            \end{itemize}
     \item[Case 2:] suppose $e_{i+1}$ connects a rigid vertex with a free abelian vertex group of rank $n$ and let $P$ be its peripheral subgroup. 
     We may assume that the free abelian vertex group is not in $\mathcal{G}_i$ and we further take cases:
             \begin{itemize} 
             \item[2A.] assume that $\eta(P)$ cannot be conjugated neither 
             to the centralizer of $\eta(G_{e})$ for some edge $e$ in $\mathcal{G}_i$ that connects two rigid vertex groups 
             nor to the centralizer of the image of the peripheral subgroup of some abelian vertex group in $\mathcal{G}_i$. 
             Then $Comp_{i+1}$ is the fundamental group of the graph of groups 
             obtained by gluing to $\mathcal{G}(Comp_i)$ a free abelian flat of rank $n$ along the centralizer of $\eta(P)$ in $Comp_i$ and moreover 
             in the new abelian vertex group we add the relations identifying the peripheral subgroup in its centralizer. The latter 
             graph of groups is $\mathcal{G}(Comp_{i+1})$;
             \item[2B.] assume that $\eta(P)$ can be conjugated to the centralizer $C$ (in $Comp_i$) either of $\eta(G_{e})$ for some edge $e$ in $\mathcal{G}_i$ 
             that connects two rigid vertex groups, or 
             of the image of the peripheral group of some abelian vertex group (i.e. a vertex group whose vertex belongs to $V_A$) in $\mathcal{G}_i$. Then 
             $Comp_{i+1}$ is the group corresponding to the graph of groups obtained by gluing to $\mathcal{G}(Comp_i)$ a free abelian 
             flat of rank $n$ along $C$. 
             \end{itemize}
     \item[Case 3:] suppose $e_{i+1}:=(u,v)$ connects a surface type vertex with a rigid vertex (i.e. a vertex that belongs to $V_R$). We further take cases 
     according to whether the surface type vertex belongs to $\mathcal{G}_i$ or not: 
           \begin{itemize}
           \item[3A.] assume that the surface vertex group $G_v$ does not belong to $\mathcal{G}_i$. Then $Comp_{i+1}$ 
           is the amalgamated free product $Comp_i*_{G_{e_{i+1}}}\tilde{G}_v$ where $\tilde{G}_v$ is an isomorphic copy of $G_v$ 
           witnessed by the isomorphism $f:G_v\rightarrow \tilde{G}_v$, and the edge group embeddings $\tilde{f}_{e_{i+1}}, \tilde{f}_{\bar{e}_{i+1}}$ 
           are defined as follows: $\tilde{f}_{e_{i+1}}=f\circ f_{e_{i+1}}$ and $\tilde{f}_{\bar{e}_{i+1}}=\eta\circ f_{\bar{e}_{i+1}}$ 
           where $f_{e_i+1}, f_{\bar{e}_{i+1}}$ are the injective morphisms that correspond to the edge group of the splitting 
           $\pi_1(\mathcal{G}_i)*_{G_{e_{i+1}}}G_v$;
           \item[3B.] assume that the surface vertex group $G_u$ belongs to $\mathcal{G}_i$. Then by our recursive hypothesis there 
           exists an isomorphic copy of $G_u$, say $f:G_u\rightarrow \tilde{G}_u$, in $Comp_i$. We define $Comp_{i+1}$ to be the   
           $HNN$ extension $Comp_i*_{G_{e_{i+1}}}$, where the edge group embeddings $\tilde{f}_{e_{i+1}}, \tilde{f}_{\bar{e}_{i+1}}$ 
           are defined as follows: $\tilde{f}_{e_{i+1}}=\eta\circ f_{e_{i+1}}$ and $\tilde{f}_{\bar{e}_{i+1}}=f\circ f_{\bar{e}_{i+1}}$, 
           where $f_{e_{i+1}},f_{\bar{e}_{i+1}}$ are the injective morphisms that correspond to the edge group of the splitting 
           $\pi_1(\mathcal{G}_i)*_{G_{e_{i+1}}}G_v$ (in the case $v$ is not in $\mathcal{G}_i$) or of the splitting 
           $\pi_1(\mathcal{G}_i)*_{G_{e_{i+1}}}$ (in the case $v$ is in $\mathcal{G}_i$).
           \end{itemize}
     \end{itemize}

Finally the group $Comp(\mathcal{G}(G),\eta):=Comp_m$ is called the completion of $G$ with respect to $\mathcal{G}(G)$ the sequence $\mathcal{G}_0\subset
\ldots\subset\mathcal{G}_m$ and $\eta$.
\end{definition}

The completion of a group $G$ with respect to a strict map $\eta:G\rightarrow L$ has a natural structure of a floor over $L$. Moreover,    
Sela has proved  \cite[Lemma 1.13]{Sel2} that $G$ admits a natural embedding into its completion.

\begin{lemma}\label{CompletionEmbedding}
Let $G$ be a group and $(\mathcal{G}(G),(V_S,V_A,V_R))$ be a $GAD$ for $G$. Let $L$ be a limit group 
and $\eta:G\twoheadrightarrow L$ be such that $(\mathcal{G}(G),\eta)$ is strict. 
Let $Comp(\mathcal{G}(G),\eta)$ be the completion of $G$ with respect to $\mathcal{G}(G)$ and $\eta$. 
Then $G$ admits a natural embedding to $Comp(\mathcal{G}(G),\eta)$.  
\end{lemma}
\begin{proof}
Let $\mathcal{G}_0\subset \mathcal{G}_1\subset\ldots\subset\mathcal{G}_m:=\mathcal{G}(G)$ 
be the sequence of subgraphs of groups that covers the graph of groups $\mathcal{G}(G)$ and with respect to which we 
have constructed the completion $Comp(\mathcal{G}(G),\eta)$. We will prove by induction that for each $i\leq m$, there exists an 
injective map $f_i:\pi_1(\mathcal{G}_i)\rightarrow Comp_i$ such that $f_{i+1}\supset f_i$ 
and $\bigcup f_i:=f:G\rightarrow Comp(\mathcal{G}(G),\eta)$ agrees with $\eta$ up to conjugation, by an element that is either trivial or does not live in $L$, in 
the vertex groups whose vertices belong to $V_R$ in the $GAD$ for $G$. 
\\ \\ \noindent
{\bf Base step.} We take $f_0$ to be $\eta\upharpoonright \pi_1(\mathcal{G}_0)$. Since we have assumed that 
 the unique vertex in $\mathcal{G}_0$ belongs $V_R$ and $Comp_0$ is $L$, one sees that $f_0:\pi_1(\mathcal{G}_0)\rightarrow Comp_0$ 
 is injective and respects our hypothesis on vertex groups whose vertices belong to $V_R$.
\\ \\ \noindent
{\bf Induction step.} Let $f_i:\pi_1(\mathcal{G}_i)\rightarrow Comp_i$ be the morphism that satisfies our induction hypothesis. 
 We find an injective morphism $f_{i+1}:\pi_1(\mathcal{G}_{i+1})\rightarrow Comp_{i+1}$ that extends $f_i$ and satisfies the 
 hypothesis on vertex groups whose vertices belong to $V_R$. We take cases according to the initial and terminal vertices of 
 the edge $e:=(u,v)\in\mathcal{G}_{i+1}\setminus\mathcal{G}_i$. 
     \begin{itemize}
      \item Suppose we are in Case 1A of Definition \ref{Completion}. Then 
      $\pi_1(\mathcal{G}_{i+1})$ is either the amalgamated free product $\pi_1(\mathcal{G}_i)*_{G_e}G_v$ (if $v\notin\mathcal{G}_i$) or 
      the $HNN$ extension $\pi_1(\mathcal{G}_i)*_{G_e}$ (if $v\in\mathcal{G}_i$) and 
      $Comp_{i+1}$ is the amalgamated free product $Comp_i*_C(C\oplus\langle z\rangle)$ where $C$ is the 
      centralizer of $\eta(G_e)$ in $Comp_i$. 
      
      {\bf Suppose that $\pi_1(\mathcal{G}_{i+1})$ is an amalgamated free product}. By the induction hypothesis 
      $f_i\upharpoonright G_u=Conj(\gamma_u)\circ\eta\upharpoonright G_u$. 
      We define $f_{i+1}$ to agree with $f_i$ on $\pi_1(\mathcal{G}_i)$ and $f_{i+1}(g)=\gamma_u z \eta(g)z^{-1}\gamma_u^{-1}$ for $g\in G_v$. Note that 
      obviously $\gamma_uz$ does not belong to $Comp_i$, thus it does not belong to $L$. The map 
      $f_{i+1}$ is indeed a morphism since for any $g\in G_e$ we have that $f_{\bar{e}}(g)$ is an element that lives in $G_u$, 
      thus $f_{i+1}(f_{\bar{e}}(g))=f_i(f_{\bar{e}}(g))=\gamma_u\eta(f_{\bar{e}}(g))\gamma_u^{-1}$, on the other hand 
      $f_{i+1}(f_{e}(g))=\gamma_uz\eta(f_{e}(g))z^{-1}\gamma_u^{-1}$ and since $\eta(f_e(g))$ is in $C$ we get that 
      $f_{i+1}(f_{e}(g))=\gamma_u\eta(f_e(g))\gamma_u^{-1}$. Therefore, since $\eta(f_{\bar{e}}(g))=\eta(f_e(g))$, we see that 
      $f_{i+1}(f_e(g))=f_{i+1}(f_{\bar{e}}(g))$. We continue by proving that $f_{i+1}$ is injective. Let $g:=a_1b_1\ldots a_nb_n$ 
      be an element of $\pi_1(\mathcal{G}_i)*_{G_e}G_v$ in reduced form. Then $f_{i+1}(g)=f_i(a_1)\gamma_u z \eta(b_1)z^{-1}\gamma_u^{-1}\ldots f_i(a_n)
      \gamma_u z \eta(b_n)z^{-1}\gamma_u^{-1}$, we show that this form is reduced with respect to $Comp_i*_C(C\oplus\langle z\rangle)$. Suppose not, 
      then either $\gamma_u^{-1}f_i(a_j)\gamma_u$, for some $2\leq j\leq n$, or $\eta(b_j)$ for some $j\leq n$ is in $C$. 
      In the first case 
      this means that $\gamma_u^{-1}f_i(a_j)\gamma_u$ commutes with some (any) non-trivial element, say $\gamma$, of $\eta(G_e)$, thus $f_i(a_j)$ 
      commutes with $\gamma_u\gamma\gamma_u^{-1}$, but $\gamma_u\gamma\gamma_u^{-1}$ is the image of an element in $G_e$ under $f_i$ 
      and since $f_i$ is injective we have that $a_j$ commutes with an element of $G_{e}$, by the maximality condition of $G_{e}$ 
      this shows that $a_j$ belongs to $G_e$, a contradiction. In the second case, this means that $\eta(b_j)$ commutes with some (any) non-trivial element, 
      say $\gamma$, of $\eta(G_e)$, thus since $\eta$ is injective on $G_v$, we see that $b_j$ commutes with an element of $G_e$. By the maximality condition 
      of $G_e$, $b_j$ must belong to $G_e$, a contradiction.

      {\bf Suppose that $\pi_1(\mathcal{G}_{i+1})$ is an $HNN$ extension}. By the induction hypothesis  
      $f_i\upharpoonright G_u=Conj(\gamma_u)\circ\eta\upharpoonright G_u$, 
      $f_i\upharpoonright G_v=Conj(\gamma_v)\circ\eta\upharpoonright G_v$. We define $f_{i+1}$ to agree with 
      $f_i$ on $\pi_1(\mathcal{G}_i)$ and $f_{i+1}(t)=\gamma_v\eta(t)z\gamma_u^{-1}$, where $t$ is the stable letter of 
      the $HNN$ extension. The map 
      $f_{i+1}$ is indeed a morphism since for any $g\in G_e$ we have that $f_{e}(g)$ is an element that lives in $G_v$, 
      thus $f_{i+1}(f_{e}(g))=f_i(f_e(g))=\gamma_v\eta(f_e(g))\gamma_v^{-1}$, on the other hand 
      $f_{\bar{e}}(g)$ is an element that lives in $G_u$, 
      thus $f_{i+1}(tf_{\bar{e}}(g)t^{-1})= \gamma_v\eta(t)z\gamma_u^{-1}\gamma_u\eta(f_{\bar{e}}(g))\gamma_u^{-1}\gamma_u z^{-1}\eta(t)^{-1}
      \gamma_v^{-1}=\gamma_v\eta(tf_{\bar{e}}(g)t^{-1})\gamma_v^{-1}$, and it follows that $f_{i+1}$ is a morphism. We continue by proving that 
      $f_{i+1}$ is injective. Let $g:=g_0t^{\epsilon_1}g_1t^{\epsilon_2}\ldots t^{\epsilon_n}g_n$ with $\epsilon_i\in\{-1,1\}$ 
      be an element in reduced form with respect to the $HNN$ extension, 
      then $f_{i+1}(g)=f_i(g_0)(\gamma_v\eta(t)z\gamma_u^{-1})^{\epsilon_1}f_i(g_1)(\gamma_v$ $\eta(t)z\gamma_u^{-1})^{\epsilon_2}\ldots$ $
      (\gamma_v\eta(t)z\gamma_u^{-1})^{\epsilon_n}f_i(g_n)$. We will show by induction that for every $n\geq 1$, if $g$ is an element 
      of $\pi_1(\mathcal{G}_{i+1})$ of length $n$ (in reduced form) with respect to the $HNN$ extension that $\pi_1(\mathcal{G}_{i+1})$ admits, then 
      $f_{i+1}(g)$ can be put in reduced form of length at least one with respect to the amalgamated free product $Comp_i*_C(C\oplus\langle z\rangle)$ 
      that $Comp_{i+1}$ admits, moreover $f_{i+1}(g)$ ends with either $z\gamma_u^{-1}f_i(g_n)$ or $z^{-1}\eta(t)^{-1}\gamma_v^{-1}f_i(g_n)$ depending on 
      whether $\epsilon_n$ is positive or negative. For the base step ($n=1$), the result is obvious. Suppose it is true for every $k<n$, we show it is true 
      for $n$. We take cases with respect to whether $\epsilon_{n-1},\epsilon_{n}$ are positive or negative. Since the cases when both are negative or both are 
      positive are symmetric we assume that both are positive and we leave the symmetric case as an exercise. Thus, 
      $f_{i+1}(g)=f_{i+1}(g_0t^{\epsilon_1}g_1t^{\epsilon_2}\ldots tg_{n-1})\cdot f_{i+1}(t g_n)$ and by the induction hypothesis 
      $f_{i+1}(g_0t^{\epsilon_1}g_1t^{\epsilon_2}\ldots t^{\epsilon_{n-1}}g_{n-1})$ can be put in reduced form of length at least one that ends with $z\gamma_u^{-1}f_i(g_{n-1})$. 
      Therefore $f_{i+1}(g)=\gamma z\gamma_u^{-1}f_i(g_{n-1})\gamma_v\eta(t)z\gamma_u^{-1}f_i(g_n)$. This latter element has the desired properties since 
      if $\gamma_u^{-1}f_i(g_{n-1})\gamma_v\eta(t)$ belongs to $C$, then we consider $z\gamma_u^{-1}f_i(g_{n-1})\gamma_v\eta(t)z$ 
      as an element of $C\oplus\langle z\rangle\setminus C$ and if not then already the element is in reduced form ending with $z\gamma_u^{-1}f_i(g_n)$. 
      We now treat the case where $\epsilon_{n-1}=1$ and $\epsilon_n=-1$. In this case 
      $f_{i+1}(g)=f_{i+1}(g_0t^{\epsilon_1}g_1t^{\epsilon_2}\ldots tg_{n-1})\cdot f_{i+1}(t^{-1} g_n)$ and by the induction hypothesis 
      $f_{i+1}(g_0t^{\epsilon_1}g_1t^{\epsilon_2}\ldots tg_{n-1})$ can be put in reduced form of length at least one that ends with 
      $z\gamma_u^{-1}f_i(g_{n-1})$. Thus, $f_{i+1}(g)=\gamma z\gamma_u^{-1}f_i(g_{n-1})\gamma_uz^{-1}\eta(t)^{-1}\gamma_v^{-1}f_i(g_n)$. 
      Is enough to show that $\gamma_u^{-1}f_i(g_{n-1})\gamma_u$ does not belong to the centralizer $C$ of $\eta(G_e)$. Suppose, for 
      a contradiction, that it does, then $\gamma_u^{-1}f_i(g_{n-1})\gamma_u$ commutes with some (any) element of $\eta(G_e)$, such an 
      element can be written as $\gamma_u^{-1}f_i(\beta)\gamma_u$ for some $\beta\in G_u$. Therefore, $f_i(g_{n-1})$ commutes with $f_i(b)$ 
      and since $f_i$ is injective, we see that $g_{n-1}$ commutes with $b$. We can now use the maximality of $f_{\bar{e}}(G_e)$ to 
      conclude that $g_{n-1}$ belongs to it, contradicting the reduced form for $g$. The case where $\epsilon_{n-1}=-1$ and $\epsilon_n=1$ 
      is symmetric to the previous case and we leave it to the reader.
      
      \item Suppose we are in case 1B of Definition \ref{Completion}. Suppose the centralizer of $\eta(G_{e})$ (in $Comp_i$) can 
      be conjugated, by the element $\gamma$, into $C$, where $C$ satisfies the hypothesis of case 1B. Then 
      $\pi_1(\mathcal{G}_{i+1})$ is either the amalgamated free product $\pi_1(\mathcal{G}_i)*_{G_e}G_v$ (if $v\notin\mathcal{G}_i$) or 
      the $HNN$ extension $\pi_1(\mathcal{G}_i)*_{G_e}$ (if $v\in\mathcal{G}_i$) and 
      $Comp_{i+1}$ is the amalgamated free product $Comp_i*_C(C\oplus\langle z\rangle)$.
      
      {\bf Suppose that $\pi_1(\mathcal{G}_{i+1})$ is an amalgamated free product}. By the induction hypothesis 
      $f_i\upharpoonright G_u=Conj(\gamma_u)\circ\eta\upharpoonright G_u$. 
      We define $f_{i+1}$ to agree with $f_i$ on $\pi_1(\mathcal{G}_i)$ and $f_{i+1}(g)=\gamma_u\gamma^{-1} z \gamma \eta(g)\gamma^{-1} z^{-1}\gamma \gamma_u^{-1}$ 
      for $g\in G_v$. 
      It is not hard to check that $f_{i+1}$ is a morphism, the reason being that for any element $g$ of $G_e$, since $\gamma \eta(f_{e}(g)) \gamma^{-1}$ 
      belongs to $C$, it commutes with $z$, thus $f_{i+1}(f_e(g))=\gamma_u\eta(f_e(g))\gamma_u^{-1}$ and this is enough. We next prove that $f_{i+1}$ 
      is injective. Let $g:=a_1b_1\ldots a_nb_n$ 
      be an element of $\pi_1(\mathcal{G}_i)*_{G_e}G_v$ in reduced form. Then $f_{i+1}(g)=f_i(a_1)\gamma_u\gamma^{-1} z\gamma \eta(b_1)\gamma^{-1}z^{-1}\gamma\gamma_u^{-1}\ldots 
      f_i(a_n)
      \gamma_u\gamma^{-1} z\gamma $ $\eta(b_n)\gamma^{-1}z^{-1}\gamma\gamma_u^{-1}$, we show that this form is reduced with respect to $Comp_i*_C(C\oplus\langle z\rangle)$. 
      Suppose not, 
      then either $\gamma\gamma_u^{-1}f_i(a_j)\gamma_u\gamma^{-1}$, for some $2\leq j\leq n$, or $\gamma\eta(b_j)\gamma^{-1}$ for some $j\leq n$ is in $C$. 
      In the first case 
      this means that $\gamma\gamma_u^{-1}f_i(a_j)\gamma_u\gamma^{-1}$ commutes with some (any) non-trivial element, say $\beta$, of $\gamma\eta(G_e)\gamma^{-1}$, 
      thus $\gamma_u^{-1}f_i(a_j)\gamma_u$ 
      commutes with $\gamma^{-1}\beta\gamma$, which is a non trivial element of $\eta(G_e)$. Thus $f_i(a_j)$ commutes with 
      $\gamma_u\gamma^{-1}\beta\gamma\gamma_u^{-1}$, and the latter can be taken to be the image of some element of $G_e$ under $f_i$. 
      The injectivity of $f_i$ tells us, that $a_j$ commutes with a non trivial element of $G_e$ and the maximality of $G_e$ 
      gives us that $a_j$ belongs to $G_e$, a contradiction. The argumentation for the case where $\gamma\eta(b_j)\gamma^{-1}$, for some $j\leq n$, is in $C$ 
      is similar. 
      
      {\bf Suppose that $\pi_1(\mathcal{G}_{i+1})$ is an $HNN$ extension}. By the induction hypothesis  
      $f_i\upharpoonright G_u=Conj(\gamma_u)\circ\eta\upharpoonright G_u$, 
      $f_i\upharpoonright G_v=Conj(\gamma_v)\circ\eta\upharpoonright G_v$. We define $f_{i+1}$ to agree with 
      $f_i$ on $\pi_1(\mathcal{G}_i)$ and $f_{i+1}(t)=\gamma_v\eta(t)\gamma^{-1}z\gamma\gamma_u^{-1}$, where $t$ is the stable letter of 
      the $HNN$ extension. One can check in a similar way to the corresponding case 1A, that $f_{i+1}$ is an injective morphism.
      
      \item Suppose we are in case 2A of Definition \ref{Completion}. Then $\pi_1(\mathcal{G}_{i+1})$ is the 
      amalgamated free product $\pi_1(\mathcal{G}_i)*_{G_e}G_v$, where $G_v$ is a free abelian group of rank $n$, 
      and $Comp_{i+1}$ is the amalgamated free product $Comp_i*_CA$ where $C$ is isomorphic to the centralizer of the image of the 
      peripheral subgroup $P(G_v)$ by $\eta$ (which is actually the same as $G_e$) and $A$ is the free abelian group $C\oplus \Z^n$ 
      with the relations identifying the peripheral subgroup as a subgroup of $C$ and as a subgroup of $\Z^n$. 
      
      By the induction hypothesis $f_i\upharpoonright G_u=Conj(\gamma_u)\circ\eta\upharpoonright G_u$. We define $f_{i+1}$ 
      to agree with $f_i$ on $\pi_1(\mathcal{G}_i)$ and $f_{i+1}(g)=\gamma_u f(g)\gamma_u^{-1}$ for $g\in G_v$, where 
      $f:G_v\rightarrow \Z^n$ is an isomorphism between $G_v$ and its copy in $Comp_{i+1}$. It is not hard to check that $f_{i+1}$ 
      is a morphism: let $g\in G_e$, then $f_{i+1}(f_e(g))= \gamma_uf(f_e(g))\gamma_u^{-1}$ and $f_{i+1}(f_{\bar{e}}(g))=\gamma_u \eta(f_{\bar{e}}(g))\gamma_u^{-1}$, 
      since $f(f_e(g))=\eta(f_{\bar{e}}(g))$ in $Comp_{i+1}$, we have what we wanted. We next prove that $f_{i+1}$ 
      is injective. Let $g:=a_1b_1\ldots a_nb_n$ be an element of $\pi_1(\mathcal{G}_i)*_{G_e}G_v$ in reduced form. Then 
      we prove that $f_{i+1}(g)=f_i(a_1)\gamma_u f(b_1)\gamma_u^{-1}\ldots f_i(a_n)\gamma_u f(b_n)\gamma_u^{-1}$ is in reduced 
      form with respect to $Comp_i*_CA$. It is trivial to check that $f(b_j)$ does not live in $\tilde{f}_e(C)$ (since $b_j\in G_v\setminus G_e$), thus 
      we only need to check that $\gamma_u^{-1} f_i(a_j)\gamma_u$ is not in $\tilde{f}_{\bar{e}}(C)$ for $2\leq j\leq n$. Suppose not, then 
      $\gamma_u^{-1}f_i(a_j)\gamma_u$ commutes with some (any) element of $\eta(f_{\bar{e}}(G_e))$. Therefore, $f_i(a_j)$ 
      commutes with $\gamma_u \eta(f_{\bar{e}}(g))\gamma_u^{-1}$, for some $g\in G_e$, but that is the image of $g$ under $f_i$, and since 
      $f_i$ is injective we see that $a_j$ commutes with $f_{\bar{e}}(g)$, by the maximality of $f_{\bar{e}}(G_e)$ 
      we have that $a_j$ belongs to it, a contradiction.
      
      
      \item Suppose we are in case 3A of Definition \ref{Completion}. Then 
      $\pi_1(\mathcal{G}_{i+1})$ is the amalgamated free product $\pi_1(\mathcal{G}_i)*_{G_e}G_v$ and 
      $Comp_{i+1}$ is the amalgamated free product $Comp_i*_{G_e}\tilde{G}_v$, where $\tilde{G}_v$ is 
      an isomorphic copy of $G_v$ witnessed by the isomorphism $f:G_v\rightarrow \tilde{G}_v$.
      
      By the induction hypothesis 
      $f_i\upharpoonright G_u=Conj(\gamma_u)\circ\eta\upharpoonright G_u$. 
      We define $f_{i+1}$ to agree with $f_i$ on $\pi_1(\mathcal{G}_i)$ and $f_{i+1}(g)=\gamma_u f(g) \gamma_u^{-1}$ for $g\in G_v$. 
      It is not hard to check that $f_{i+1}$ is a morphism: let $g\in G_e$, then $f_{i+1}(f_e(g))=\gamma_u f(f_e(g)) \gamma_u^{-1}$ 
      and $f_{i+1}(f_{\bar{e}}(g))=\gamma_u \eta(f_{\bar{e}}(g))\gamma_u^{-1}$, since $\eta(f_{\bar{e}}(g))=f(f_e(g))$ in $Comp_i*_{G_e}\tilde{G}_v$ 
      we have what we wanted. We next prove that $f_{i+1}$ is injective. Let $g:=a_1b_1\ldots a_nb_n$ be an element of $\pi_1(\mathcal{G}_i)*_{G_e}G_v$
      in reduced form, then we prove that $f_{i+1}(g)=f_i(a_1)\gamma_u f(b_1) \gamma_u^{-1}\ldots f_i(a_n)\gamma_u f(b_n) \gamma_u^{-1}$ is in reduced form 
      with respect to $Comp_i*_{G_e}\tilde{G}_v$. We only need to check that $\gamma_u^{-1}f_i(a_j)\gamma_u$ does not live in $\eta(G_e)$, for any $2\leq j\leq n$. 
      Suppose, for a contradiction, not, then $f_i(a_j)$ is $\gamma_u\eta(\gamma)\gamma_u^{-1}$, for some $\gamma\in G_e$, but the latter is the 
      image of $\gamma$ under $f_i$, and since $f_i$ is injective we have that $a_j$ is $\gamma$, a contradiction.
      
      \item Suppose we are in case 3B of Definition \ref{Completion}. Then 
      $\pi_1(\mathcal{G}_{i+1})$ is the amalgamated free product $\pi_1(\mathcal{G}_i)*_{G_e}G_v$ or the $HNN$ extension 
      $\pi_1(\mathcal{G}_i)*_{G_e}$ and $Comp_{i+1}$ is the $HNN$ extension $Comp_i*_{G_e}$.
      
      {\bf Suppose that $\pi_1(\mathcal{G}_{i+1})$ is an amalgamated free product}. By the induction hypothesis $f_i\upharpoonright G_u=Conj(\gamma_u)\circ f$, 
      where $f:G_u\rightarrow \tilde{G}_u$ is an isomorphism, and $\tilde{G}_u$ is the isomorphic copy of $G_u$ in $Comp_i$. We define $f_{i+1}$ to agree 
      with $f_i$ on $\pi_1(\mathcal{G}_i)$ and $f_{i+1}(g)=\gamma_u t^{-1}\eta(g) t\gamma_u^{-1}$ for $g\in G_v$, where $t$ is the stable letter of 
      the $HNN$ extension $Comp_i*_{G_e}$. It is not hard to check that $f_{i+1}$ is a morphism: let $g\in G_e$, then 
      $f_{i+1}(f_{e}(g))=\gamma_u t^{-1}\eta(f_e(g))t\gamma_u^{-1}$ and $f_{i+1}(f_{\bar{e}}(g))=\gamma_u f(f_{\bar{e}}(g))\gamma_u^{-1}$, since 
      $\eta(f_e(g))=t f(f_{\bar{e}}(g)) t^{-1}$ in $Comp_i*_{G_e}$, we have what we wanted. We next prove that $f_{i+1}$ is injective. 
      Let $g:=a_1b_1\ldots a_nb_n$ be an element of $\pi_1(\mathcal{G}_i)*_{G_e}G_v$
      in reduced form, then we prove that $f_{i+1}(g)= f_i(a_1) \gamma_u t^{-1} \eta(b_1) t \gamma_u^{-1} \ldots f_i(a_n)\gamma_u t^{-1} \eta(b_n) t\gamma_u^{-1}$ is 
      in reduced form with respect to $Comp_i*_{G_e}$. We need to check that $\eta(b_j)$ is not in $\tilde{f}(G_e)$ (which is actually $\eta(f_e(G_e))$) for any $j\leq n$
      and $\gamma_u^{-1}f_i(a_j)\gamma_u$ is not in $\tilde{f}_{\bar{e}}(G_e)$ (which is actually $f(f_{\bar{e}}(G_e))$) for any $2\leq j\leq n$. But both follow easily 
      by the fact that $b_j$ is in $G_v\setminus f_e(G_e)$ and $a_j\in\pi_1(\mathcal{G}_i)\setminus f_{\bar{e}}(G_e)$.
      
      {\bf Suppose that $\pi_1(\mathcal{G}_{i+1})$ is an $HNN$ extension}. By the induction hypothesis $f_i\upharpoonright G_u=Conj(\gamma_u)\circ f$, 
      where $f:G_u\rightarrow \tilde{G}_u$ is an isomorphism, and $\tilde{G}_u$ is the isomorphic copy of $G_u$ in $Comp_i$, and $f_i\upharpoonright G_v= 
      Conj(\gamma_v)\circ \eta$. We define $f_{i+1}$ to agree with $f_i$ on $\pi_1(\mathcal{G}_i)$ and $f_{i+1}(t)=\gamma_v \tilde{t}\gamma_u^{-1}$ where $\tilde{t}$ 
      is the Bass-Serre element of the $HNN$ extension $Comp_i*_{G_e}$. 
      It is not hard to check that $f_{i+1}$ is a morphism: let $g\in G_e$, then $f_{i+1}(f_e(g))=\gamma_v\eta(f_e(g))\gamma_v^{-1}$ and 
      $f_{i+1}(tf_{\bar{e}}(g))t^{-1})= \gamma_v\tilde{t}f(f_{\bar{e}})\tilde{t}^{-1}\gamma_v^{-1}$, since $\eta(f_e(g))= \tilde{t}f(f_{\bar{e}})\tilde{t}^{-1}$ 
      in $Comp_{i+1}$, we have what we wanted. We next prove that $f_{i+1}$ is injective. Let $g:= g_0t^{\epsilon_1}g_1t^{\epsilon_2}\ldots t^{\epsilon_n}g_n$,  
      with $\epsilon_j\in\{1,-1\}$ be an element in reduced form with respect to the $HNN$ extension $\pi_1(\mathcal{G}_i)$, then we prove that 
      $f_{i+1}(g)=f_i(g_0)(\gamma_v\tilde{t}\gamma_u^{-1})^{\epsilon_1}f_i(g_1)\ldots(\gamma_v\tilde{t}\gamma_u^{-1})^{\epsilon_n}f_i(g_n)$ is in 
      reduced form. We need to check that when $\epsilon_j=1$ and $\epsilon_{j+1}=-1$, then $\gamma_u^{-1}f_i(g_j)\gamma_u$ does not live in $\tilde{f}_{\bar{e}}(G_e)$ 
      (which is actually $f(f_{\bar{e}}(G_e))$), and when $\epsilon_j=-1$ and $\epsilon_{j+1}=1$, then $\gamma_v^{-1}f_i(g_j)\gamma_v$ does not live in 
      $\tilde{f}_e(G_e)$ (which is actually $\eta(f_e(G_e))$). We show just the former, since the latter case is symmetric. Suppose, for a contradiction, 
      that $\gamma_u^{-1}f_i(g_j)\gamma_u$ is in $f(f_{\bar{e}}(G_e))$, then there exists an element in $\gamma\in f_{\bar{e}}(G_e)$, such that 
      $f_i(g_j)=\gamma_u f(\gamma)\gamma_u^{-1}$, but then $f_i(g_j)=f_i(\gamma)$ and since $f_i$ is injective we see that $g_j=\gamma$, a contradiction. 
     
     \end{itemize}

\end{proof}

It is not hard to deduce from the construction of the completion that: 

\begin{lemma}\label{SolidInjections}
Let $Sld$ be a solid group with respect to a finitely generated subgroup $H$. 
Let $JSJ_{H}(Sld)$ be the abelian $JSJ$ 
decomposition of $Sld$ with respect to $H$ and $Comp(Sld, Id)$ the completion of $Sld$ 
with respect to $Id:Sld\rightarrow Sld$ and the modification $\hat{JSJ}_{H}(Sld))$ of its relative $JSJ$ decomposition. 
Let $\hat{H}$ be the vertex group in $\hat{JSJ}_H(Sld)$ that contains $H$ and 
$i_t,i_b$ be the natural embeddings from $Sld$ to $Comp(Sld, Id)$. Then $i_t(Sld)\cap i_b(Sld)=\hat{H}$.
\end{lemma}

For the notion of the abelian $JSJ$ decomposition of a solid limit group we refer the reader to \cite[Theorem 9.2]{Sel1}.

\subsection{Strictly solid morphisms and families}\label{SldMorph}

We start by defining the notion of a degenerate map with respect to a tower over a solid limit group. 

\begin{definition}
Let $\mathcal{T}(G,Sld)$ be a tower of height $m$ over a limit group $L$, 
$$\{G,(\mathcal{G}(G^m,G^{m-1}),r_m),(\mathcal{G}(G^{m-1},G^{m-2}),r_{m-1}),\ldots, (\mathcal{G}(G^1,G^0),r_1), L\}$$ 
and $s:L\rightarrow \F$ be a morphism. 
We say that a morphism $h:G\rightarrow\F$ factors through the 
tower $\mathcal{T}(G,L)$ based on $s$ if $h=s\circ r_1\circ \alpha_1\circ\ldots\circ r_m\circ \alpha_m$ where 
$\alpha_i\in Mod(\mathcal{G}(G^i,G^{i-1}))$. 
\end{definition}

\begin{definition}
Let $\mathcal{T}(G,Sld)$ be a tower of height $m$ over a solid limit group $Sld$,  
$$\{G,(\mathcal{G}(G^m,G^{m-1}),r_m),(\mathcal{G}(G^{m-1}, G^{m-2}),r_{m-1}),\ldots, (\mathcal{G}(G^1,G^0),r_1), Sld\}$$  
and let $s:Sld\rightarrow \F$ be a morphism. 
We say that $s$ is degenerate if for all morphisms $h:G\rightarrow\F$ that factor through the 
tower $\mathcal{T}(G,Sld)$ based on $s$ and some $i<m$ one of the following holds:
\begin{itemize}
 \item an edge group of $(\mathcal{G}(G^{i+1},G^i),(V_S,V_A,V_R))$ is always mapped to the trivial element;
 \item a vertex group, $G_u$, of $(\mathcal{G}(G^{i+1},G^i),(V_S,V_A,V_R))$ with $u\not\in V_A$ is always 
 mapped to a cyclic subgroup of $\F$.
\end{itemize}

\end{definition}

As a matter of fact being degenerate is a condition definable by a system of equations. 

\begin{lemma}[Sela]\label{DefDeg}
Let $\F:=\F(\bar{a})$ be a non abelian free group. Let $\mathcal{T}(G,Sld)$ be a graded tower over a solid limit 
group $Sld:=\langle \bar{x},\bar{a} \ | \ \Sigma(\bar{x},\bar{a})\rangle$. 
Then there exists a finite system of equations $\Psi(\bar{x},\bar{a})=1$ such that a morphism $s:Sld\rightarrow\F$ is 
degenerate with respect to $\mathcal{T}(G,Sld)$ if and only if $\Psi(s(\bar{x}),\bar{a})=1$.
\end{lemma}

\begin{definition}[Strictly solid morphism]
Let $Sld$ be a solid group with respect to a finitely generated subgroup $H$. Let $JSJ_{H}(Sld)$ be the abelian $JSJ$ 
decomposition of $Sld$ with respect to $H$ and $Comp(Sld, Id)$ the completion of $Sld$ 
with respect to $Id:Sld\rightarrow Sld$ and the modification $\hat{JSJ}_{H}(Sld))$ of its relative $JSJ$ decomposition. 
Let $i_t,i_b$ be the natural embeddings 
from $Sld$ to $Comp(Sld, Id)$. The morphisms $h_1,h_2:Sld\to \F$ are tower equivalent with respect 
to $Comp(Sld, Id)$ if there is a morphism $H:Comp(Sld)\to\F$ such that $H\circ i_t=h_1$ and $H\circ i_b=h_2$. 

Moreover, a morphism $h:Sld\rightarrow\F$ is called strictly solid if it is non degenerate and it is 
not tower equivalent with respect to $Comp(Sld,Id)$ to a flexible morphism;
\end{definition}

The relation of tower equivalence between two morphisms is an equivalence relation and we call the 
set of non degenerate morphisms that belong to the class of a strictly solid morphism a {\em strictly solid family}.

We prove the following:

\begin{lemma}\label{StrictlySolidRel}
Let $Sld$ be a solid limit group with respect to a finitely generated subgroup $H$. Let $h_1,h_2:Sld\rightarrow\F$ be two morphisms. Then $h_1, h_2$ 
are tower equivalent with respect to $Comp(Sld, Id)$ if and only if they agree up to conjugation 
in the rigid vertices of the relative $JSJ$ decomposition of $Sld$ and they are 
identical on the vertex group that contains $H$.
\end{lemma}

\begin{proof}
Let $\mathcal{G}_0\subseteq\mathcal{G}_1\subseteq\ldots\subseteq\mathcal{G}_m:=JSJ(Sld)$ be the sequence of the subgraphs of groups 
used in the construction of the completion $Comp(Sld, Id)$. We recall that $\mathcal{G}_0$ is the vertex group that contains $H$ and 
it is rigid by definition. Let $f_b:Sld\rightarrow Comp(Sld)$ be the identity map that 
sends $Sld$ onto the $Comp_0$, and 
let, for each $i\leq m$, $g_i:\pi_1(\mathcal{G}_i)\rightarrow Comp_i$ 
be the natural injective map as in the proof of Lemma \ref{CompletionEmbedding}. We will define recursively 
a sequence of morphisms $(F_i)_{i\leq m}:Comp_i\rightarrow \F$, so that for each $i\leq m$ the morphism $F_i$:
\begin{enumerate}
 \item extends the map $F_{i-1}$; 
 \item satisfies $F_i\circ f_b = h_1$; and 
 \item it satisfies $F_i\circ g_i = h_2\upharpoonright \pi_1(\mathcal{G}_i)$.
\end{enumerate}
\noindent
{\bf Base step.} We define $F_0: Comp_0\rightarrow \F$ to be essentially the map $h_1$, i.e. a morphism 
such that $F_0\circ f_b=h_1$. Clearly, since $g_0(\pi_1(\mathcal{G}_0))$ belongs to $Comp_0$ and 
$h_1,h_2$ agree on $\pi_1(\mathcal{G}_0)$, the morphism $F_0$ satisfies the required properties. 
\\ \\
{\em Before moving to the recursive step we note that property (2) is automatically satisfied for $i>1$ once property (1) is satisfied.}
\\ \\
{\bf Recursive step.} Suppose there exists a morphism $F_i$ that satisfies properties (1)-(3), we define 
$F_{i+1}$ according to the cases for the edge $e:=(u,v)\in\mathcal{G}_{i+1}\setminus\mathcal{G}_i$.
\begin{itemize}
 \item Suppose we are in case 1A of Definition \ref{Completion}. Then $\pi_1(\mathcal{G}_{i+1})$ 
 is either the amalgamated free product $\pi_1(\mathcal{G}_i)*_{G_e}G_v$ (if $v\notin\mathcal{G}_i$) or 
 the $HNN$ extension $\pi_1(\mathcal{G}_i)*_{G_e}$ (if $v\in \mathcal{G}_i$) and $Comp_{i+1}$ 
 is the amalgamated free product $Comp_i*_C(C\oplus\langle z\rangle)$ where $C$ is the centralizer of $G_e$ 
 in $Comp_i$. By the hypothesis 
 we have that $h_1(g)=r_v h_2(g)r_v^{-1}$ for every $g\in G_v$ and $h_1(g)=r_uh_2(g)r_u^{-1}$ for every $g\in G_u$. 
 \\ 
 {\bf Suppose that $\pi_1(\mathcal{G}_{i+1})$ is an amalgamated free product.} The fact that $h_1$ and $h_2$ are morphisms 
 induces a relation between $r_v$ and $r_u$. For every $g\in G_e$ we have 
 that $h_2(f_{\bar{e}}(g))=h_2(f_e(g))$, thus $r_u^{-1}h_1(f_{\bar{e}}(g))r_u=r_v^{-1}h_1(f_e(g))r_v$. The last equation 
 implies that $r_vr_u^{-1}$ commutes with $h_1(f_e(g))$, thus $r_v=c r_u$ for some element $c$ in the centralizer of 
 $h_1(f_{e}(g))$.  
 
 By the inductive hypothesis and the 
 definition of the map $g_i:\pi_1(\mathcal{G}_i)\rightarrow Comp_i$ we see that, for every element $g\in G_u$, 
 $F_i\circ g_i(g)=F_i(\gamma_uf_b(g)\gamma_u^{-1})=h_2(g)=r_u^{-1}h_1(g)r_u$ for some $\gamma_u\in Comp_i$. This 
 implies that $r_uF_i(\gamma_u)$ commutes with $h_1(g)$ for every $g\in G_u$, therefore $r_uF_i(\gamma_u)=d$ for some element $d$ in the 
 intersection of all centralizers $C(h_1(g))$ for $g\in G_u$. 
 
 Finally, we define $F_{i+1}$ to 
 agree with $F_i$ in $Comp_i$ and $F_{i+1}(z)=d^{-1}r_ur_v^{-1}$. This definition makes indeed $F_{i+1}$ a morphism 
 as $r_ur_v^{-1}$ and $d$ both commute with the image of $C$ under $F_i$. Moreover, $F_{i+1}\circ g_{i+1}(g)=F_{i+1}(\gamma_uzf_b(g)z^{-1}\gamma_u^{-1})=
 r_v^{-1}h_1(g)r_v=h_2(g)$ for every $g\in G_v$, as we wanted.\\
 {\bf Suppose that $\pi_1(\mathcal{G}_{i+1})$ is an $HNN$ extension.} By the fact that $h_1$ and $h_2$ are morphisms 
 we get, for every $g\in G_e$, that  $h_1(t)h_1(f_{\bar{e}}(g))h_1(t)^{-1}=h_1(f_e(g))$, and 
 $h_2(t)h_2(f_{\bar{e}}(g))h_2(t)^{-1}=h_2(f_e(g))$. 
 
 By the inductive hypothesis and the 
 definition of the map $g_i:\pi_1(\mathcal{G}_i)\rightarrow Comp_i$ we see that, for every element $g\in G_u$, 
 $F_i\circ g_i(g)=F_i(\gamma_uf_b(g)\gamma_u^{-1})=h_2(g)$. 
 The same line of thought as above gives us that $F_i\circ g_i(g)=F_i(\gamma_vf_b(g)\gamma_v^{-1})=h_2(g)$  for $g\in G_v$. 
 
 Finally, we define $F_{i+1}$ to agree with $F_i$ in $Comp_i$ and $F_{i+1}(z)=h_1(t)^{-1}F_i(\gamma_v)^{-1}h_2(t)F_i(\gamma_u)$. We can easily check 
 that $F_{i+1}\circ g_{i+1}(t)=F_{i+1}(\gamma_v f_b(t)z\gamma_u^{-1})=h_2(t)$ as well as that $F_{i+1}$ is a morphism. For the latter we only need to 
 check that $F_{i+1}(z)$ commutes with $F_{i+1}(f_{\bar{e}}(g))$ for some non trivial $g\in G_e$. Indeed, $F_{i+1}(z)^{-1}F_{i+1}(f_{\bar{e}}(g))F_{i+1}(z)= 
 F_i(\gamma_u)^{-1}h_2(t)^{-1}F_i(\gamma_v)h_1(t)F_{i+1}(f_{\bar{e}}(g))h_1(t)^{-1}F_i(\gamma_v)^{-1}h_2(t)F_i(\gamma_u)$. We can replace 
 $h_1(t)$ $F_{i+1}(f_{\bar{e}}(g))h_1(t)^{-1}$ with $F_{i+1}(f_e(g))$, so $F_{i+1}(z)^{-1}F_{i+1}(f_{\bar{e}}(g))F_{i+1}(z)= F_i(\gamma_u)^{-1}h_2(t)^{-1}$ $F_i(\gamma_v)
 F_{i+1}(f_e(g))F_i(\gamma_v)^{-1}h_2(t)F_i(\gamma_u)$. Again we can replace, by the induction hypothesis, 
 $F_i(\gamma_v)F_{i+1}(f_e(g))F_i(\gamma_v)^{-1}$ with $h_2(f_e(g))$, so $F_{i+1}(z)^{-1}F_{i+1}(f_{\bar{e}}(g))F_{i+1}(z)=
 F_i(\gamma_u)^{-1}$ $h_2(t)^{-1}h_2(g)h_2(t)F_i(\gamma_u)$. We can continue by replacing $h_2(t)^{-1}h_2(f_e(g))h_2(t)$ with $F_i\circ g_i(f_{\bar{e}}(g))$, so 
 $F_{i+1}(z)^{-1}F_{i+1}(f_{\bar{e}}(g))F_{i+1}(z)=F_i(\gamma_u)^{-1}F_i\circ g_i(f_{\bar{e}}(g))F_i(\gamma_u)$. And finally, 
 $F_i(\gamma_u)^{-1}F_i\circ g_i(f_{\bar{e}}(g))F_i(\gamma_u)=F_i(f_{\bar{e}}(g))$ as we wanted.
 
 \item Suppose we are in case 2A of Definition \ref{Completion}. Then $\pi_1(\mathcal{G}_{i+1})$ is the amalgamated free product 
 $\pi_1(\mathcal{G}_i)*_{G_e}G_v$, where $G_v$ is a free abelian group of rank $n$, and $Comp_{i+1}$ is the amalgamated free product 
 $Comp_i*_CA$ where $C$ is isomorphic to the centralizer (in $Comp_i$) of the peripheral subgroup $P(G_v)$ of $G_v$ and $A$ is the free abelian group $C\oplus \Z^n$, 
 where $\Z^n$ is an isomorphic copy of $G_v$ witnessed by $\phi:G_v\rightarrow \Z^n$, together 
 with the relations identifying the peripheral subgroup as a subgroup of $C$ and as a subgroup of $\Z^n$. 
 Moreover, $g_{i+1}(g)=\gamma_u\phi(g)\gamma_u^{-1}$ for every $g\in G_u$, where $G_u$ is by our modification the peripheral subgroup of $G_v$. 
 
 By the hypothesis we have that $h_1(g)=r_u h_2(g)r_u^{-1}$ for every $g\in G_u$. By the inductive hypothesis we know that 
 $F_i\circ g_i(g)=F_i(\gamma_uf_b(g)\gamma_u^{-1})=h_2(g)$ for every $g\in G_u$, so we have that $F_i(\gamma_u)h_1(g) F_i(\gamma_u)^{-1}=r_u^{-1}h_1(g)r_u$ $(*)$. 
 We define $F_{i+1}$ to agree with $F_i$ in $Comp_i$ 
 and $F_{i+1}(\phi(g))=F_i(\gamma_u)^{-1}h_2(g)F_i(\gamma_u)$ for every $g\in G_v$. It is immediate to check 
 that $F_{i+1}\circ g_{i+1}(g)=h_2(g)$ for every $g\in G_v$, thus we only check that $F_{i+1}$ is a morphism. 
 Let $g\in G_u$, we show that $F_i(g)=F_{i+1}(\phi(g))$ 
 as these are the relation in the group $A$. Indeed, $F_i(g)=h_1(g)$ and $F_{i+1}(\phi(g))=F_i(\gamma_u)^{-1}h_2(g)F_i(\gamma_u)$ 
 and replacing $h_2(g)$ by $r_u^{-1}h_1(g)r_u$ we get $F_{i+1}(\phi(g))=F_i(\gamma_u)^{-1}r_u^{-1}h_1(g)r_uF_i(\gamma_u)$ which 
 is by $(*)$ equal to $h_1(g)$ as we wanted. 
 
 \item Suppose we are in case 3A of Definition \ref{Completion}. Then $\pi_1(\mathcal{G}_{i+1})$ is the amalgamated free product 
 $\pi_1(\mathcal{G}_i)*_{G_e}G_v$, where $G_v$ is the fundamental group of a surface, and $Comp_{i+1}$ is the amalgamated free product 
 $Comp_i*_{G_e}\tilde{G}_v$, where $\tilde{G}_v$ is an isomorphic copy of $G_v$ witnessed by $\phi:G_v\rightarrow \tilde{G}_v$ and the edge group 
 embeddings are $\tilde{f}_e=\phi\circ f_e$ and $\tilde{f}_{\bar{e}}=f_b\circ f_{\bar{e}}$. Moreover, 
 $g_{i+1}(g)=\gamma_u\phi(g)\gamma_u^{-1}$ for every $g\in G_u$, where $G_u$ is a rigid vertex. 
 
 By the hypothesis we have that $h_1(g)=r_u h_2(g)r_u^{-1}$ for every $g\in G_u$. By the inductive hypothesis we know that 
 $F_i\circ g_i(g)=F_i(\gamma_uf_b(g)\gamma_u^{-1})=h_2(g)$ for every $g\in G_u$, so we have that $F_i(\gamma_u)h_1(g) F_i(\gamma_u)^{-1}=r_u^{-1}h_1(g)r_u$ $(*)$. 
 We define $F_{i+1}$ to agree with $F_i$ in $Comp_i$ 
 and $F_{i+1}(\phi(g))=F_i(\gamma_u)^{-1}h_2(g)F_i(\gamma_u)$ for every $g\in G_v$. It is easy to see 
 that $F_{i+1}\circ g_{i+1}(g)=h_2(g)$ for every $g\in G_v$, thus we only check that $F_{i+1}$ is a morphism. 
 Indeed, $F_{i+1}(\tilde{f}_{\bar{e}}(g))=h_1(f_{\bar{e}}(g))$ and $F_{i+1}(\phi(\tilde{f}_e(g))=F_i(\gamma_u)^{-1}h_2(f_e(g))F_i(\gamma_u)$ 
 and replacing $h_2(f_e(g))$ by $h_2(f_{\bar{e}}(g))$ and then by $r_u^{-1}h_1(f_{\bar{e}}(g))r_u$ which is by $(*)$ equal to 
 $h_1(f_{\bar{e}}(g))$ we get what we wanted.
 
 \item Suppose we are in case 3B of Definition \ref{Completion}. Then $\pi_1(\mathcal{G}_{i+1})$ is the amalgamated free product 
 $\pi_1(\mathcal{G}_i)*_{G_e}G_v$ or the $HNN$ extension $\pi_1(\mathcal{G}_i)*_{G_e}$ and $Comp_{i+1}$ is the $HNN$ extension 
 $Comp_i*_{G_e}$. In this case there exist an isomorphic copy $\phi:G_u\rightarrow \tilde{G}_u$ of the surface group $G_u$ in $Comp_i$ and 
 the edge group embeddings are $\tilde{f}_e = f_b\circ f_e$, $\tilde{f}_{\bar{e}}=\phi\circ f_{\bar{e}}$.\\ 
 {\bf Suppose that $\pi_1(\mathcal{G}_{i+1})$ is an amalgamated free product.} In this case $g_i(g)=\gamma_u\phi(g)\gamma_u^{-1}$ for every $g\in G_u$, 
 and $g_{i+1}(g)=\gamma_u t^{-1}f_b(g)t\gamma_u^{-1}$ for every $g\in G_v$, 
 where $t$ is the Bass-Serre element of the $HNN$ extension $Comp_i*_{G_e}$. And by the hypothesis we have that $h_1(g)=r_vh_2(g)r_v^{-1}$ for every $g\in G_v$. 
 
 We define $F_{i+1}$ to agree with $F_i$ in $Comp_i$ and $F_{i+1}(t)=r_vF_{i+1}(\gamma_u)$. It is easy to see that $F_{i+1}\circ g_{i+1}(g)=h_2(g)$ for every $g\in G_v$, 
 thus we only need to check that $F_{i+1}$ is a morphism. Indeed, for every $g\in G_e$, $F_{i+1}(\tilde{f}_e(g))=h_1(f_e(g))$ and $F_{i+1}(t\tilde{f}_{\bar{e}}(g)t^{-1})= 
 F_{i+1}(t)F_{i+1}(\phi(f_{\bar{e}}(g)))F_{i+1}(t)^{-1}$. If, in the latter equation, we replace $F_{i+1}(t)$ with $r_vF_{i+1}(\gamma_u)$, we get 
 $F_{i+1}(t\tilde{f}_{\bar{e}}(g)t^{-1})=r_vF_{i+1}(\gamma_u)F_{i+1}(\phi(f_{\bar{e}}(g)))F_{i+1}(\gamma_u)^{-1}$ $r_v^{-1}$ and replacing 
 $\gamma_u\phi(f_{\bar{e}}(g))\gamma_u^{-1}$ with $g_{i+1}(f_{\bar{e}}(g))$ we get $F_{i+1}(t\tilde{f}_{\bar{e}}(g)t^{-1})=r_vF_{i+1}\circ g_{i+1}(f_{\bar{e}}(g))r_v^{-1}$ 
 which is $r_vh_2(f_{\bar{e}}(g))r_v^{-1}$ and finally it is $h_1(f_{\bar{e}}(g))$ as we wanted. \\ 
 {\bf Suppose that $\pi_1(\mathcal{G}_{i+1})$ is an $HNN$ extension.} In this case $g_i(g)=\gamma_u\phi(g)\gamma_u^{-1}$ for every $g\in G_u$ and 
 $g_i(g)=\gamma_vf_b(g)\gamma_v^{-1}$ for every $g\in G_v$. Moreover $g_{i+1}(t)=\gamma_v\tilde{t}\gamma_u^{-1}$ where $t$ is the Bass-Serre element 
 of the $HNN$ extension $\pi_1(\mathcal{G}_i)*_{G_e}$ and $\tilde{t}$ is the Bass-Serre element of the $HNN$ extension $Comp_i*_{G_e}$. 
 By the hypothesis we have that $h_1(g)=r_vh_2(g)r_v^{-1}$ for every $g\in G_v$. 
 
 We define $F_{i+1}$ to agree with $F_i$ in $Comp_i$ and $F_{i+1}(\tilde{t})=F_{i+1}(\gamma_v)^{-1}h_2(t)F_{i+1}(\gamma_u)$. It is easy to see that 
 $F_{i+1}\circ g_{i+1}(t)=h_2(t)$, thus we only need to check that $F_{i+1}$ is a morphism. Indeed, $F_{i+1}(\tilde{f}_e(g))=h_1(f_{e}(g))$ and 
 $F_{i+1}(\tilde{t}\tilde{f}_{\bar{e}}(g)\tilde{t}^{-1})=F_{i+1}(\gamma_v)^{-1}h_2(t)$ $F_{i+1}(\gamma_u)F_{i+1}(\tilde{f}_{\bar{e}}(g))
 F_{i+1}(\gamma_u)^{-1}h_2(t)^{-1}F_{i+1}(\gamma_v)$. If, in the latter equation, we replace $\gamma_u\tilde{f}_{\bar{e}}(g)\gamma_u^{-1}$ with 
 $g_{i+1}(f_{\bar{e}}(g))$, we get $F_{i+1}(\tilde{t}\tilde{f}_{\bar{e}}(g)\tilde{t}^{-1})=F_{i+1}(\gamma_v)^{-1}h_2(t)h_2(f_{\bar{e}}(g))h_2(t)^{-1}$ 
 $F_{i+1}(\gamma_v)$ and replacing $tf_{\bar{e}}(g)t^{-1}$ with $f_e(g)$ we get $F_{i+1}(\tilde{t}\tilde{f}_{\bar{e}}(g)\tilde{t}^{-1})= 
 F_{i+1}(\gamma_v)^{-1}h_2(f_e(g))$ $F_{i+1}(\gamma_v)$ which in turn is equal to $F_{i+1}(\gamma_v)^{-1}r_v^{-1}h_1(f_e(g))r_vF_{i+1}(\gamma_v)$ 
 that finally equals $h_1(f_e(g))$ as we wanted.
\end{itemize}

\end{proof}




\section{Test sequences, diophantine envelopes and applications}\label{Envelopes}
In this section we record a notion that we will use extensively throughout the rest of the paper: the notion of a {\em test sequence} over a tower. 
A test sequence is a sequence of morphisms from a group that has the structure of a tower to a free group that, roughly speaking, 
witnesses the tower structure of the group in the limit action. We give more details in subsection \ref{ts}.    

As we have already noted in the introduction of our paper, it is hard to decide when a subset of some cartesian product of a 
non abelian free group $\F$ is definable in $\F$. 
Our main idea is that one can deduce certain properties of a definable set through ``generic'' points in its {\em envelope}. 
The envelope of a definable set consists of a finite set of diophantine sets which moreover have a geometric structure. 
The union of the diophantine sets that take part in the envelope ``cover'' the definable set and in addition ``generic'' 
elements (with respect to the geometric structure of each diophantine set) live in the definable set. In subsection 
\ref{GrdDio} we give all the details.

In the final subsection we explain the connections between the geometric tools developed and model theory. We prove 
some variations of Merzlyakov-type theorems that make apparent the usefulness of passing from an arbitrary definable set, 
to the diophantine sets in its envelope.  

\subsection{test sequences}\label{ts}

We begin by giving some examples of groups that have the structure of a tower and we define the notion of a test sequence for them. 
The simplest cases of groups that admit a structure of a tower are finitely generated free groups and free abelian groups. 

For the rest of this section we fix a non abelian free group $\F$ and a basis of $\F$ with respect to which we will measure 
the length of elements of $\F$.

\begin{definition}\label{FreeTest}
Let $\langle x_1,\ldots, x_k\rangle$ be a free group of rank $k$. Then a sequence of morphisms $(h_n)_{n<\omega}:\langle\bar{x}\rangle\rightarrow\F$ is a 
test sequence with respect 
to $\langle\bar{x}\rangle$, if $h_n(\bar{x})$ satisfies the small cancellation property $C'(1/n)$ for each $n<\omega$. 
\end{definition}

\begin{remark}
Without loss of generality we will assume that all the $x_i$'s have similar growth under $(h_n)_{n<\omega}$, i.e. for each $i,j<k$ there 
are $c_{i,j},c'_{i,j}\in \R^+$ such that $c_{i,j}<\frac{\abs{h_n(x_i)}_{\F}}{\abs{h_n(x_j)}_{\F}}<c'_{i,j}$.
\end{remark}

\begin{definition}\label{FreeAbelianTest}
Let $\Z^k:=\langle x_1,\ldots, x_k \ | \ [x_i,x_j]\rangle$ be a free abelian group of rank $k$. Let $x_{i_1}>x_{i_2}>\ldots>x_{i_k}$ be some order on the $x_i$'s. 
Then a sequence of morphisms $(h_n)_{n<\omega}:\Z^k\rightarrow\F$ is a test sequence with respect 
to the free abelian group $\Z^k$ (and the given order), 
if $h_n(x_{i_1})=b_n^{m_{i_1}(n)}, \ldots, h_n(x_{i_k})=b_n^{m_{i_k}(n)}$ where $b_n$ satisfies the small cancellation property $C'(1/n)$ for each $n<\omega$ 
and $\frac{m_{i_{j+1}}(n)}{m_{i_j}(n)}$ goes to $0$ as $n\to\infty$ for every $j<k$. 
\end{definition}

We continue by defining a test sequence for a tower that consists of a single abelian flat over the parameter free group.

\begin{definition}\label{SingleAbelianFlatTest}
Let $\Z^k:=\langle x_1,\ldots, x_k \ | \ [x_i,x_j]\rangle$ be a free abelian group of rank $k$. Let $G$ be the amalgamated free product $\F*_CC\oplus \Z^k$, 
where $C:=\langle c\rangle$ is infinite cyclic and $f_{\bar{e}}(c)=a$ for 
some $a\in\F$ such that $\langle a\rangle$ is maximal abelian in $\F$ and $f_e(c)=c$. Let $x_{i_1}>x_{i_2}>\ldots>x_{i_k}$ be some order on the fixed 
basis of $\Z^k$. 

Then a sequence of morphisms $(h_n)_{n<\omega}:G\rightarrow\F$ is a test sequence with respect 
to (the tower structure of) $G$ (and the given order), 
if $h_n\upharpoonright\F=Id$ for every $n<\omega$ and $h_n(x_{i_1})=a^{m_{i_1}(n)}, \ldots, h_n(x_{i_k})=a^{m_{i_k}(n)}$ where 
$m_{i_k}\to\infty$ and $\frac{m_{i_{j+1}}(n)}{m_{i_j}(n)}$ goes to $0$ as $n\to\infty$ for every $j<k$. 
\end{definition}

\begin{remark}
In particular when in the above definition $k=1$, any infinite sequence $(h_n)_{n<\omega}:G\rightarrow\F$ with $h_n\upharpoonright\F=Id$ is 
a test sequence with respect to (the tower structure of) $G$.
\end{remark}

\begin{definition} 
If $(h_n)_{n<\omega}:G\to H$ is a sequence of morphisms from $G$ to a finitely generated group $H$ (with a fixed generating set) 
and $g_1,g_2$ are in $G$, then we say that {\em the growth of $g_1$ dominates the growth of $g_2$ (under $(h_n)_{n<\omega}$)} if 
$\frac{\abs{h_n(g_2)}_{H}}{\abs{h_n(g_1)}_{H}}\to 0$ as $n\to \infty$, where $\abs{g}_H$ is the word length of $g$ with respect to the fixed 
generating set for $H$. 
\end{definition}

When a tower consists only of abelian floors and free products we define a test sequence as follows: 


\begin{definition}
Let $G$ be a group that has the structure of a tower $\mathcal{T}(G,\F)$ over $\F$. Suppose $\mathcal{T}(G,\F)$ only contains abelian floors 
and free products. For each abelian floor of the tower we choose an order for the abelian flats or equivalently we assume that each abelian floor 
consists of a single abelian flat. For each abelian flat $\Z^k:=\langle x_1,\ldots, x_k \ | \ [x_i,x_j]\rangle$ we choose an 
order $x_{i_1}>x_{i_2}>\ldots> x_{i_k}$ for the elements of the fixed basis $x_1,\ldots,x_k$.

Then a sequence of morphisms $(h_n)_{n<\omega}:G\to\F$ is called a test sequence for $\mathcal{T}(G,\F)$ (with respect to the 
given order of abelian flats and the given order of their generating sets) if the following conditions hold: 
\begin{itemize}
 \item $h_n\upharpoonright \F=Id$ for every $n<\omega$;
 \item we define the conditions of the restriction of $(h_n)_{n<\omega}$ to the $i+1$-th flat by taking cases according to 
 whether $G^{i+1}$ has a structure of a free product or a free abelian flat over $G^i$:
    \begin{enumerate}
     \item Suppose $G^{i+1}$ is the free product of $G^i$ with $\F_l$, then $h_n\upharpoonright \F_l$ satisfies 
     the requirements of Definition \ref{FreeTest}. Moreover, the growth of any non trivial element in $\F_l$ (under $h_n$) dominates 
     the growth of every element in $G^i$ (under $h_n$).  
     \item Suppose $G^{i+1}=G^i*_E(E\oplus\Z^k)$, is obtained from $G^i$ by gluing a free abelian flat along $E$ (where $E$ is maximal abelian in $G^i$). 
     Let $\gamma_n$ be the generator of the cyclic group (in $\F$) that $E$ is mapped into by $h_n\upharpoonright G^i$. Then we 
     define $h_n(x_{i_1})=\gamma_n^{m_{i_1}(n)}, \ldots, h_n(x_{i_k})=\gamma_n^{m_{i_k}(n)}$, where 
     $m_{i_k}(n)\to\infty$ and $\frac{m_{i_{j+1}}(n)}{m_{i_j}(n)}$ goes to $0$ as $n\to\infty$ for every $j<k$. Moreover the growth of $x_{i_k}$ (under $h_n$) dominates the 
     growth of every element in $G^i$ (under $h_n$).
    \end{enumerate}

\end{itemize}

\end{definition}

More generally, in order to define a test sequence for a group $G$ that has the structure of a tower $\mathcal{T}(G,\F)$ over $\F$ we first need to 
order the abelian flats, the surface flats and the free factors that appear in the floors of the tower and in addition we need to order the base elements of each 
abelian flat. 

\begin{definition}
Let $\mathcal{T}(G,\F):=((\mathcal{G}(G^1,G^0),$ $r_1),(\mathcal{G}(G^2,G^1),$ $r_2),\ldots,(\mathcal{G}(G^m,G^{m-1}),r_m))$ be a tower 
of height $m$ over $\F$. Assume that each floor is either a single flat (abelian or surface) or it is a free product. For each $i<m$, 
let $B_i$ be one of the following: 
\begin{itemize}
 \item if $\mathcal{G}(G^{i+1},G^i)$ is a surface flat that is obtained by gluing a surface $\Sigma_{g,n}$ along its boundary onto 
 some subgroups of $G_i$, then $B_i$ is the subgroup of $G^{i+1}$ generated by the fundamental group of the 
 surface together with the Bass-Serre elements $\langle \pi_1(\Sigma_{g,n}), t_1,\ldots, t_n\rangle$;
 \item if $\mathcal{G}(G^{i+1},G^i)$ is an abelian flat obtained from $G^i$ by gluing a free abelian group $\Z^k$ along the 
 maximal abelian subgroup $E$, then $B_i$ is $\Z^k$;
 \item if $\mathcal{G}(G^{i+1},G^i)$ is the free product $G^{i+1}=G^i*\F_l$, then $B_i$ is $F_l$. 
\end{itemize}

We say that $B_{i_0}<B_{i_1}<\ldots<B_{i_m}$ is a legitimate ordering if the following conditions hold: 
\begin{itemize}
 \item $r_1(B_{i_0})\leq \F$;
 \item for each $0<j<m$, the image of $B_{i_j}$ under the retraction $r_{i_j+1}$ is a subgroup of the following group $\langle \F, B_{i_0},\ldots, B_{i_j-1}\rangle$.
\end{itemize}

Moreover, the tower $\mathcal{T}(G,\F)$ together with a legitimate ordering 
and an order for the basis of each abelian flat is called an ordered tower. We will denote an ordered tower (for some ordering) 
by $(\mathcal{T}(G,\F), <)$
\end{definition}

A tower can always be ordered by choosing an arbitrary order on the surface flats and abelian flats of each floor, then placing the flats of the 
$i$-th floor before those of the $i+1$-th floor and choosing an order 
for the basis elements of each abelian flat. On the other hand, one could have more complicated legitimate orderings in the sense that 
a flat that is part of a higher floor than some other flat can be ordered before this latter flat. In any case, one can obtain from 
a legitimate ordering a tower structure by reshuffling the floors according to the legitimate order and change the retractions accordingly. We give some examples.

\begin{example}
We consider the tower over $\F$ with two floors defined as follows:
\begin{itemize}
 \item the first floor $\mathcal{G}(G^1,\F)$ is a surface flat, that is obtained by gluing the surface $\Sigma_{1,1}$ (whose fundamental 
 group is $\langle x_1,x_2\rangle$)
 along its boundary onto the subgroup of $\F$ which is generated by the commutator of two non commuting elements $a_1,a_2$. Thus, 
 $G^1:=\langle \F,x_1,x_2 \ | \ [x_1,x_2]=[a_1,a_2]\rangle$ and $r_1:G^1\twoheadrightarrow\F$ is the morphism staying the 
 identity on $\F$ and sending $x_i$ to $a_i$ for $i\leq 2$. Note that $B_0$ is $\langle x_1,x_2\rangle$;
 \item the second floor $\mathcal{G}(G^2,G^1)$ is a surface flat, that is obtained by gluing the surface $\Sigma_{1,1}$ (whose fundamental group 
 is $\langle y_1,y_2\rangle$) 
 along its boundary onto the subgroup of $\F$ which is generated by the commutator of two non commuting elements $b_1,b_2$. 
 Thus, $G^2:=\langle G^1, y_1,y_2 \ | \ [y_1,y_2]=[b_1,b_2]\rangle$ and $r_2:G^2\twoheadrightarrow G^1$ is the morphism staying  
 the identity on $G^1$ and sending $y_i$ to $b_i$ for $i\leq 2$. Note that $B_1$ is $\langle y_1,y_2\rangle$.  
\end{itemize}
This tower admits two legitimate orderings: the natural one $B_0<B_1$, but also the following $B_1<B_0$ since $r_2(B_1)\leq \F$ and 
$r_1(B_0)\leq \langle \F, B_1\rangle$. 

As noted above one can give $G^2$ the following tower structure:
\begin{itemize}
 \item the first floor $\mathcal{G}(\hat{G}^1,\F)$ is a surface flat, that is obtained by gluing the surface $\Sigma_{1,1}$ (whose fundamental 
 group is $\langle y_1,y_2\rangle$)
 along its boundary onto the subgroup of $\F$ which is generated by the commutator of $b_1,b_2$. Thus, 
 $\hat{G}^1:=\langle \F,y_1,y_2 \ | \ [y_1,y_2]=[b_1,b_2]\rangle$ and $\hat{r}_1:\hat{G}^1\twoheadrightarrow\F$ is $r_2$ 
 restricted on $\langle \F, y_1, y_2\rangle$;
 \item the second floor $\mathcal{G}(G^2,\hat{G}^1)$ is a surface flat, that is obtained by gluing the surface $\Sigma_{1,1}$ (whose fundamental group 
 is $\langle x_1,x_2\rangle$) 
 along its boundary onto the subgroup of $\F$ which is generated by the commutator of $a_1,a_2$. 
 Thus, $G^2:=\langle \hat{G}^1, x_1,x_2 \ | \ [x_1,x_2]=[a_1,a_2]\rangle$ and $\hat{r}_2:G^2\twoheadrightarrow \hat{G}^1$ is the morphism agreeing 
 with $r_1$ on $\langle \F, x_1,x_2\rangle$ and stays the identity on $\langle y_1,y_2\rangle$.
\end{itemize}

\end{example}

\begin{example}
We consider the tower over $\F$ with two floors defined as follows:
\begin{itemize}
 \item the first floor $\mathcal{G}(G^1,\F)$ is a surface flat, that is obtained by gluing the surface $\Sigma_{1,1}$ (whose fundamental 
 group is $\langle x_1,x_2\rangle$)
 along its boundary onto the subgroup of $\F$ which is generated by the commutator of two non commuting elements $a_1,a_2$. Thus, 
 $G^1:=\langle \F,x_1,x_2 \ | \ [x_1,x_2]=[a_1,a_2]\rangle$ and $r_1:G^1\twoheadrightarrow\F$ is the morphism staying the 
 identity on $\F$ and sending $x_i$ to $a_i$ for $i\leq 2$. Note that $B_0$ is $\langle x_1,x_2\rangle$;
 \item the second floor $\mathcal{G}(G^2,G^1)$ is a surface flat, that is obtained by gluing the surface $\Sigma_{1,1}$ (whose fundamental group 
 is $\langle y_1,y_2\rangle$) 
 along its boundary onto the subgroup which is generated by the commutator of $x_1$ and $b$ for some non trivial $b\in\F$. 
 Thus, $G^2:=\langle G^1, y_1,y_2 \ | \ [y_1,y_2]=[x_1,b]\rangle$ and $r_2:G^2\twoheadrightarrow G^1$ is the morphism staying  
 the identity on $G^1$ and sending $y_1$ to $x_1$ and $y_2$ to $b$. Note that $B_1$ is $\langle y_1,y_2\rangle$.  
\end{itemize}
This tower admits only one legitimate ordering: the natural one $B_0<B_1$. 

One can easily check that $B_1<B_0$ is not a legitimate ordering since $r_2(B_1)$ 
is not a subgroup of $\F$. 
\end{example}

\begin{remark} It is not hard to check that:
\begin{itemize}
 \item a twin tower admits two natural legitimate orderings;
 \item a tower closure inherits an ordering from the corresponding tower.
\end{itemize}
\end{remark}

Suppose $G$ has the structure of a tower $\mathcal{T}(G,\F)$ over $\F$ and let $(\mathcal{T}(G,\F), <)$ be some ordering on it. Then  
a sequence of morphisms $(h_n)_{n<\omega}:G\to\F$ is a {\em test sequence} 
for this (ordered) tower if it satisfies the combinatorial conditions $(i)-(xiv)$ in \cite[p.222]{Sel2}. The existence of a test sequence 
for a group that has the structure of a tower (for any ordering of the tower) has been proved in \cite[Lemma 1.21]{Sel2}.  

\begin{proposition}
Suppose $G$ has the structure of a tower $\mathcal{T}(G,\F)$ over $\F$. Let $(\mathcal{T}(G,\F), <)$ be some ordering. 
Then a test sequence for $(\mathcal{T}(G,\F), <)$ exists.
\end{proposition}

In this paper we will not use the full strength of the results connected with a test sequence. So, for our purposes the following facts, 
used extensively in \cite{Sel2} (cf. Theorem 1.3, Proposition 1.8, Theorem 1.18), about test sequences assigned to groups that have 
the structure of a tower will be enough.

\begin{fact}[Free product limit action]\label{GroundFloorFact}
Let $\mathcal{T}(G,\F)$ be a tower and $(h_n)_{n<\omega}:G\to\F$ 
be a test sequence for $\mathcal{T}(G,\F)$. 

Suppose $G^{i+1}$ is the free product of $G^i$ with a group $\F_l$ and $(h_n\upharpoonright G^{i+1})_{n<\omega}$ 
is the restriction of $(h_n)_{n<\omega}$ to $G^{i+1}$. Then, any subsequence of $(h_n\upharpoonright G^{i+1})_{n<\omega}$ 
that converges, as in Lemma \ref{LimitAction}, induces a faithful action of $G^{i+1}$ on a based real tree $(Y,*)$, with the following properties:  
  \begin{enumerate}
   \item the action $G^{i+1}\curvearrowright Y$ decomposes as 
   a graph of actions in the following way $(G^{i+1}\curvearrowright T, \{Y_u\}_{u\in V(T)}$ $,\{ p_e\}_{e\in E(T)})$;
   \item the Bass-Serre presentation for $G^{i+1}\curvearrowright T$, $(T_1=T_0,T_0)$, is a segment $(u,v)$;
   \item $Stab_G(u):=\F_l\curvearrowright Y_u$ is a simplicial type action, its Bass-Serre presentation, 
   $(Y_u^1,Y_u^0,t_1,$ $\ldots,t_l)$ consists of a ``star graph'' $Y_u^1:=\{(x,b_1),\ldots,(x,b_l)\}$ with all of its edges 
   trivially stabilized, a point 
   $Y_u^0=x$ which is trivially stabilized and Bass-Serre elements $t_i=e_i$, for $i\leq l$;
   \item $Y_v$ is a point and $Stab_G(v)$ is $G^i$;
   \item the edge $(u,v)$ is trivially stabilized.
   \end{enumerate}
\end{fact}

\begin{fact}[Surface flat limit action]\label{SurfaceFact}
Let $\mathcal{T}(G,\F)$ be a tower and $(h_n)_{n<\omega}:G\to\F$ 
be a test sequence with respect to $\mathcal{T}(G,\F)$.
  
Suppose $G^{i+1}$ is a surface flat over $G^i:=G^i_1*\ldots * G^i_m$, witnessed by $\mathcal{A}(G^{i+1},G^i)$,  
and $(h_n\upharpoonright G^{i+1})_{n<\omega}$ is the restriction of $(h_n)_{n<\omega}$ to the $i+1$-flat of 
$\mathcal{T}(G,\F)$. 

Then, any subsequence of $(h_n\upharpoonright G^{i+1})_{n<\omega}$ that converges, as in Lemma \ref{LimitAction},  
induces a faithful action of $G^{i+1}$ on a based real tree $(Y,*)$, with the following properties:
 \begin{enumerate}
  \item $G^{i+1}\curvearrowright Y$ decomposes as a graph of actions $(G^{i+1}\curvearrowright T, \{Y_u\}_{u\in V(T)}, \{ p_e\}_{e\in E(T)})$, 
  with the action $G^{i+1}\curvearrowright T$ being identical to $\mathcal{A}(G^{i+1},G^i)$; 
  \item if $v$ is not a surface type vertex then $Y_v$ is a point stabilized by the corresponding $G^i_j$ for some $j\leq m$;
  \item if $u$ is the surface type vertex, then $Stab_G(u)=\pi_1(\Sigma_{g,l})$ and 
  the action $Stab_G(u)\curvearrowright Y_u$ is a surface type action coming from $\pi_1(\Sigma_{g,l})$;
 \end{enumerate}
\end{fact}

\begin{fact}[Abelian flat limit action]\label{AbelianFact}
Let $\mathcal{T}(G,\F)$ be a limit tower and $(h_n)_{n<\omega}:G\to\F$ 
be a test sequence with respect to $\mathcal{T}(G,\F)$. 

Suppose $G^{i+1}=G^i*_A(A\oplus\Z)$ is obtained from $G^i$ by gluing a free abelian flat along $A$ (where $A$ is a maximal abelian subgroup of $G^i$) 
and $(h_n\upharpoonright G^{i+1})_{n<\omega}$ is the restriction of $(h_n)_{n<\omega}$ to the $i+1$-flat of $\mathcal{T}(G,\F)$. 

Then any subsequence of $(h_n\upharpoonright G^{i+1})_{n<\omega}$ that converges, as in Lemma \ref{LimitAction}, 
induces a faithful action of $G^{i+1}$ on a based real tree $(Y,*)$, with the following properties:
 \begin{enumerate}
  \item the action of $G^{i+1}$ on $Y$, $G^{i+1}\curvearrowright Y$, decomposes as 
  a graph of actions $(G^{i+1}\curvearrowright T, \{Y_u\}_{u\in V(T)},\{ p_e\}_{e\in E(T)})$;
   \item the Bass-Serre presentation for $G^{i+1}\curvearrowright T$, $(T_1=T_0,T_0)$, is a segment $(u,v)$;
   \item $Stab_G(u):=A\oplus\Z\curvearrowright Y_u$ is a simplicial type action, its Bass-Serre presentation, 
   $(Y_u^1,Y_u^0,t_e)$ consists of a segment $Y_u^1:=(a,b)$ whose stabilizer is $A$, a point 
   $Y_u^0=a$ whose stabilizer is $A$ and a Bass-Serre element $t_e$ which is $z$;
   \item $Y_v$ is a point and $Stab_G(v)$ is $G_i$;
   \item the edge $(u,v)$ is stabilized by $A$.
 \end{enumerate}
\end{fact}

\subsection{Graded towers and test sequences}

We begin by defining the notion of a graded tower.

\begin{definition}\label{GradedTower}
Let $Sld$ be a solid limit group with respect to the finitely generated subgroup $H$ and $Comp(Sld):=Comp(JSJ_H(Sld), Id)$ be  
its completion with respect to the relative $JSJ$ decomposition and the identity map $Id:Sld\rightarrow Sld$. Assume that $JSJ_H(Sld)$ is not 
trivial and let $i_b,i_t:Sld\rightarrow Comp(Sld)$ be the natural injective maps (where $i_b$ is the identity).
Let $\mathcal{T}(G,Sld)$ be a tower over $Sld$. 

Then the group corresponding to the graded tower $\mathcal{GT}(G,Sld)$ is the amalgamated free product $G*_{Sld}Comp(Sld)$, where $f_e:Sld\rightarrow Comp(Sld)$ is the map $i_t$,  
and $f_{\bar{e}}:Sld\rightarrow G$ maps $Sld$ isomorphically onto the ground floor of $\mathcal{T}(G,Sld)$. Moreover, $\mathcal{GT}(G,Sld)$ is the tower 
over $i_b(Sld)$ starting with the floors of $Comp(Sld)$ and continuing with the floors of $\mathcal{T}(G,Sld)$ over the $i_t(Sld)$.
\end{definition}

\begin{remark}
Concerning Definition \ref{GradedTower} we note the following: 
\ \begin{itemize}
   \item in the case where the relative $JSJ$ decomposition of $Sld$ is trivial, we take the graded tower to be the tower itself. In the arguments 
   that will follow this will always be a degenerate case, so we will always assume that the relative $JSJ$ of $Sld$ is not trivial;
   \item the name ``graded tower'' might be misleading since it is not obvious that $\mathcal{GT}(G,Sld)$ has the structure of a tower. Nevertheless, 
    this amalgamated free product can indeed be seen as a tower over the solid limit group that sits on the ground floor of $Comp(Sld)$ after a 
    few modifications as the reader can easily check. These would just involve the ``incrementation'' of some abelian flats 
    of the first floor of $Comp(Sld)$ since now some peg of some floor of $\mathcal{T}(G,Sld)$ can be conjugated to an abelian flat in this first floor, that did not 
    exist before. 
  \end{itemize}

\end{remark}

A {\em graded test sequence} for a graded tower $\mathcal{GT}(G,Sld)$ (with respect to some ordering) is a sequence of morphisms 
that restricts to a fixed non degenerate (with respect to $\mathcal{GT}(G,Sld)$) strictly solid morphism on $i_b(Sld)$ and 
moreover it satisfies the same properties 
as in Facts \ref{GroundFloorFact}, \ref{SurfaceFact}, \ref{AbelianFact} for the floors of $\mathcal{GT}(G,Sld)$ 
with the exception that the group acting on the limit action is quotiened by the 
kernel of the fixed strictly solid morphism. 

\begin{fact}\label{GroundSolidFloorFact}
Let $\mathcal{GT}(G,Sld)$ be a graded tower over $Sld$. Let $s:Sld\rightarrow\F$ be a non degenerate (with respect to $\mathcal{GT}(G,Sld)$) 
strictly solid morphism and let $(h_n)_{n<\omega}:G*_{Sld}Comp(Sld)\to\F$ 
be a graded test sequence for $\mathcal{GT}(G,\F)$ (for some ordering) based on $s$. Let $K$ be the kernel of $s$.

Suppose $G^{i+1}$ is the free product of $G^i$ with a free group $\F_l:=\langle e_1,\ldots, e_l\rangle$ and $(h_n\upharpoonright G^{i+1})_{n<\omega}$ 
is the restriction of $(h_n)_{n<\omega}$ to the $i+1$-flat of 
$\mathcal{GT}(G,Sld)$. Then, any subsequence of $(h_n\upharpoonright G^{i+1})_{n<\omega}$ 
that converges, as in Lemma \ref{LimitAction}, induces a faithful action of $G^{i+1}/K$ on a based real tree $(Y,*)$, with the following properties:  
  \begin{enumerate}
   \item the action $G^{i+1}/K\curvearrowright Y$ decomposes as 
   a graph of actions in the following way $(G^{i+1}/K\curvearrowright T, \{Y_u\}_{u\in V(T)}$ $,\{ p_e\}_{e\in E(T)})$;
   \item the Bass-Serre presentation for $G^{i+1}/K\curvearrowright T$, $(T_1=T_0,T_0)$, is a segment $(u,v)$;
   \item $Stab_G(u):=\F_l\curvearrowright Y_u$ is a simplicial type action, its Bass-Serre presentation, 
   $(Y_u^1,Y_u^0,t_1,$ $\ldots,t_l)$ consists of a ``star graph'' $Y_u^1:=\{(x,b_1),\ldots,(x,b_l)\}$ with all of its edges 
   trivially stabilized, a point 
   $Y_u^0=x$ which is trivially stabilized and Bass-Serre elements $t_i=e_i$, for $i\leq l$;
   \item $Y_v$ is a point and $Stab_G(v)$ is $G^i/K$;
   \item the edge $(u,v)$ is trivially stabilized.
   \end{enumerate}
\end{fact}

\begin{fact}\label{SurfaceSolidFact}
Let $\mathcal{GT}(G,Sld)$ be a graded tower over $Sld$. Let $s:Sld\rightarrow\F$ be a non degenerate (with respect to $\mathcal{GT}(G,Sld)$) 
strictly solid morphism and let $(h_n)_{n<\omega}:G*_{Sld}Comp(Sld)\to\F$ 
be a graded test sequence for $\mathcal{GT}(G,\F)$ (for some ordering) based on $s$. Let $K$ be the kernel of $s$.
  
Suppose $G^{i+1}$ is a surface flat over $G^i:=G^i_1*\ldots G^i_m$, witnessed by $\mathcal{A}(G^{i+1},G^i)$,  
and $(h_n\upharpoonright G^{i+1})_{n<\omega}$ is the restriction of $(h_n)_{n<\omega}$ to the $i+1$-flat of 
$\mathcal{GT}(G,Sld)$. 

Then, any subsequence of $(h_n\upharpoonright G^{i+1})_{n<\omega}$ that converges, as in Lemma \ref{LimitAction},  
induces a faithful action of $G^{i+1}/K$ on a based real tree $(Y,*)$, with the following properties:
 \begin{enumerate}
  \item $G^{i+1}/K\curvearrowright Y$ decomposes as a graph of actions $(G^{i+1}/K\curvearrowright T, \{Y_u\}_{u\in V(T)}, \{ p_e\}_{e\in E(T)})$, 
  with the action $G^{i+1}/K\curvearrowright T$ having the same data as $\mathcal{A}(G^{i+1},G^i)$, apart from replacing the vertex stabilizer that contains 
  $f_{\bar{e}}(Sld)$ with its quotient $f_{\bar{e}}(Sld)/K$;
  \item if $v$ is not a surface type vertex then $Y_v$ is a point stabilized by the corresponding $G^i_j/K$ for some $j\leq m$;
  \item if $u$ is the surface type vertex, then $Stab_G(u)=\pi_1(\Sigma_{g,l})$ and 
  the action $Stab_G(u)\curvearrowright Y_u$ is a surface type action coming from $\pi_1(\Sigma_{g,l})$;
 \end{enumerate}
\end{fact}

\begin{fact}\label{AbelianSolidFact}
Let $\mathcal{GT}(G,Sld)$ be a graded tower over $Sld$. Let $s:Sld\rightarrow\F$ be a non degenerate (with respect to $\mathcal{GT}(G,Sld)$) 
strictly solid morphism and let $(h_n)_{n<\omega}:G*_{Sld}Comp(Sld)\to\F$ 
be a graded test sequence for $\mathcal{GT}(G,\F)$ (for some ordering) based on $s$. Let $K$ be the kernel of $s$. 

Suppose $G^{i+1}=G^i*_A(A\oplus\Z)$ is obtained from $G^i$ by gluing a free abelian flat along $A$ (where $A$ is a maximal abelian subgroup of $G^i$) 
and $(h_n\upharpoonright G^{i+1})_{n<\omega}$ is the restriction of $(h_n)_{n<\omega}$ to the $i+1$-flat of $\mathcal{GT}(G,Sld)$. 

Then any subsequence of $(h_n\upharpoonright G^{i+1})_{n<\omega}$ that converges, as in Lemma \ref{LimitAction}, 
induces a faithful action of $G^{i+1}/K$ on a based real tree $(Y,*)$, with the following properties:
 \begin{enumerate}
  \item the action of $G^{i+1}/K$ on $Y$, $G^{i+1}/K\curvearrowright Y$, decomposes as 
  a graph of actions $(G^{i+1}/K\curvearrowright T, \{Y_u\}_{u\in V(T)},\{ p_e\}_{e\in E(T)})$;
   \item the Bass-Serre presentation for $G^{i+1}/K\curvearrowright T$, $(T_1=T_0,T_0)$, is a segment $(u,v)$;
   \item $Stab_G(u):=A/K\oplus\Z\curvearrowright Y_u$ is a simplicial type action, its Bass-Serre presentation, 
   $(Y_u^1,Y_u^0,t_e)$ consists of a segment $Y_u^1:=(a,b)$ whose stabilizer is $A/K$, a point 
   $Y_u^0=a$ whose stabilizer is $A/K$ and a Bass-Serre element $t_e$ which is $z$;
   \item $Y_v$ is a point and $Stab_G(v)$ is $G_i/K$;
   \item the edge $(u,v)$ is stabilized by $A/K$.
 \end{enumerate}
\end{fact}

Using the above properties we can prove. 

\begin{lemma}\label{RigidVertex}
Let $Sld$ be a solid limit group with respect to the finitely generated subgroup $H$. Let $\hat{JSJ}_H(Sld)$ be the modification 
of the relative $JSJ$ decomposition of $Sld$ and $\hat{H}$ be the vertex group where $H$ is contained. Let $Comp(Sld):=Comp(\hat{JSJ}_H(Sld), Id)$ be  
its completion with respect to the modification of the relative $JSJ$ decomposition and the identity map $Id:Sld\rightarrow Sld$. 

Let $\mathcal{T}(G,Sld)$ be a tower over $Sld$ and $(s_n)_{n<\omega}:Sld\rightarrow\F$ be a convergent sequence of non degenerate 
(with respect to $\mathcal{T}(G,Sld)$) strictly solid morphisms with trivial stable kernel. 

Let $g\in G$ and assume that for each $n$ there exists a graded test sequence $(h^n_m)_{m<\omega}:G*_{Sld}Comp(Sld)\rightarrow\F$ 
based on $s_n$ such that $\abs{\{h^n_m(g) \ | \ m<\omega\}}<\infty$. Then $g\in \hat{H}$.
\end{lemma}
\begin{proof}
The proof is by induction on the height of the floor in the tower $\mathcal{T}(G,Sld)$ that $g$ belongs to. Suppose that 
$g$ belongs to the ground floor, i.e. $g\in Sld$. Then $g$ belongs to some floor of the $Comp(Sld)$. If $g$ belongs to the ground floor i.e. in $i_b(Sld)$, 
then by Lemma \ref{SolidInjections} we see that $g$ belongs to $\hat{H}$ as 
we wanted. On the other hand, if $g$ belongs to $Comp(Sld)^{i+1}\setminus Comp(Sld)^i$ (i.e. to some higher floor of the completion), 
we can choose a morphism $s_j:Sld\rightarrow\F$ to base the 
test sequence of the completion that gives finitely many values to $g$, so that $g/Ker(s_j)$ belongs to $Comp(Sld)^{i+1}/Ker(s_j)$ but not to 
$Comp(Sld)^i/Ker(s_j)$. By the properties of test sequences listed above this gives us a contradiction, since in this case $g$ must 
take infinitely many values. 

We continue by assuming that $g$ belongs to some higher level $G^{i+1}\setminus G^i$ in the tower $\mathcal{T}(G,Sld)$. In this case, we can choose 
a morphism $s_j:Sld\rightarrow \F$, so that $g/Ker(s_j)$ belongs to $G^{i+1}/Ker(s_j)$ but not to $G^i/Ker(s_j)$. By the properties of test sequences 
listed above $g$ must take 
infinitely many values under $(h_m^j)_{m<\omega}$ and this gives a contradiction.
\end{proof}

\subsection{Diophantine envelopes}\label{GrdDio}

In this subsection we start by collecting some theorems of Sela that give an understanding of 
the ``rough'' structure of definable sets or parametric families of definable sets 
in non abelian free groups. We will use the machinery developed in the previous subsections, namely towers and test sequences on them. 

To give the rough idea before moving to the detailed statements: we would like to have an object (e.g. a definable set equipped with some geometric structure) 
which we can more easily handle than a ``arbitrarily complicated'' definable set, but as close as possible (in terms of the solution sets) to the definable set.

\begin{theorem}[Sela - Graded diophantine envelope]\label{GradedEnvelope}
Let $\phi(\bar{x},\bar{y},\bar{a})$ be a parametric family (with respect to $\bar{y}$) of first order formulas over $\F$. Then there exist finitely many 
graded towers, $\{(\mathcal{GT}(G_i,Sld_i))_{i\leq k}\}$, where for each $i\leq k$, $Sld_i:=\langle\bar{v}_i,\bar{y},\bar{a}\rangle$ is a solid limit group with 
respect to the subgroup generated by $\langle\bar{y},\bar{a}\rangle$, for which the following hold:
\begin{itemize}
 \item[(i)] for each $i\leq k$, there exists a convergent sequence of non degenerate (with respect to $\mathcal{T}(G_i,Sld_i)$) strictly 
 solid morphisms $s_n:Sld_i\rightarrow\F$ with trivial stable kernel and for each $n$, there exists a graded test sequence 
 $(h^n_m)_{n<\omega}:G_i*_{Sld_i}Comp(Sld_i)\rightarrow\F$ based on $s_n$, with $\F\models \phi(h_m(\bar{x}),h_m(\bar{y}),\bar{a})$; 
 \item[(ii)] if $\F\models \phi(\bar{b}_0,\bar{c}_0)$. Then there exist $i\leq k$ and:
  \begin{itemize}
  \item[(1)] a non degenerate (with respect to $\mathcal{T}(G_i,Sld_i)$) strictly solid morphism $s:Sld_i\rightarrow \F$ with $s(\bar{y})=\bar{c}_0$;
  \item[(2)] a morphism $h:G_i\rightarrow \F$ that extends $s$, factors through the tower $\mathcal{T}(G_i,Sld_i)$, and such that $h(\bar{x})=\bar{b}_0$;
  \item[(3)] a graded test sequence $(h_m)_{m<\omega}:G^i*_{Sld_i}Comp(Sld_i)\rightarrow\F$ based on $s$, 
  such that $\F\models\phi(h_m(\bar{x}),h_m(\bar{y}),\bar{a})$.
  \end{itemize}
\end{itemize}

\end{theorem}

We also record the following easy corollary of the above theorem.

\begin{corollary}[Sela - Diophantine envelope]\label{Envelope}
Let $\phi(\bar{x},\bar{a})$ be a first order formula over $\F$. Then there exist finitely many 
towers over $\F$, $\{(\mathcal{T}(G_i,\F))_{i\leq k}\}$, where $G_i:=\langle \bar{u}_i,\bar{x},\bar{a}\ | \ \Sigma_i\rangle$, 
such that:
\begin{itemize}
 \item[(i)] $\F\models\phi(\bar{x},\bar{a})\rightarrow \exists \bar{u}_1,\ldots,\bar{u}_k(\bigvee_{i=1}^{k}\Sigma_i(\bar{u}_i,\bar{x},\bar{a})=1)$; 
 \item[(ii)] for each $i\leq k$, there exists a test sequence, $(h_n)_{n<\omega}:G_i\rightarrow\F$ for $\mathcal{T}(G_i,\F)$  
 such that $\F\models \phi(h_n(\bar{x}),\bar{a})$.
\end{itemize}

\end{corollary}

We use Theorem \ref{GradedEnvelope} and the definition of a graded test sequence in order to prove.

\begin{theorem}\label{RigidGradedTower}
Let $\{\phi(\bar{x},\bar{y},\bar{a})\}$ be a parametric family (with respect to $\bar{y}$) of first order formulas over $\F$. 
Suppose that $\F\models\forall\bar{y}\exists^{<\infty}\bar{x}\phi(\bar{x},\bar{y},\bar{a})$. 
Let $\{(\mathcal{GT}(G_i,Sld_i))_{i\leq k}\}$ be 
a graded envelope for $\phi(\bar{x},\bar{y},\bar{a})$.

Let, for each $i\leq k$, $H_i:=\langle\bar{y},\bar{a}\rangle_{Sld_i}$. Suppose $\hat{H}_i$ is the vertex group that contains $H_i$ 
in the modification $\hat{JSJ}_{H_i}(Sld_i)$ of the relative $JSJ$ decomposition of $Sld_i$.

Then the tuple of elements $\bar{x}$ of each $G_i$ belongs to $\hat{H}_i$.  
\end{theorem}
\begin{proof}
The first property of the graded envelope (Theorem \ref{GradedEnvelope}(i)) brings us to the situation of Lemma \ref{RigidVertex} for 
each graded tower $\mathcal{GT}(G_i,Sld_i)$. Thus the result follows.
\end{proof}

The following theorem is a consequence of the quantifier elimination procedure (see \cite{Sel5bis}). It has been used in the proof of the 
stability of the first order theory of non abelian free groups as well as in showing that this theory is not equational, but also 
in the (weak) elimination of imaginaries.

\begin{theorem}[Sela]\label{FactSela}
Let $\mathcal{T}(G,\F)$ where $G:=\langle \bar{u},\bar{x},\bar{a}\ | \ \Sigma\rangle$ be a tower over $\F$. Let 
$\phi(\bar{x},\bar{a})$ be a first order formula over $\F$. 
Suppose there exists a test sequence, $(h_n)_{n<\omega}:G\rightarrow\F$ for $\mathcal{T}(G,\F)$, such that $\F\models\phi(h_n(\bar{x}),\bar{a})$. 

Then there exist: 

\begin{itemize}
 \item a closure, $\mathcal{R}:=cl(\mathcal{T}(G,\F))$, of $\mathcal{T}(G,\F)$;
 \item finitely many closures, $\mathcal{R}_1:=cl_1(\mathcal{R}), \ldots, \mathcal{R}_k:=cl_k(\mathcal{R})$, of $\mathcal{R}$;
 \item for each $i\leq k$, finitely many closures, $cl_1(\mathcal{R}_i),\ldots, cl_{m_i}(\mathcal{R}_i)$. 
\end{itemize}
So that a subsequence of $(h_n)_{n<\omega}$ extends to a test sequence of $\mathcal{R}$ and either: 
   \begin{itemize}
   \item it cannot be extended to a test sequence for any of the closures $\mathcal{R}_i$; or 
   \item for each $i\leq k$, that can be extended to a test sequence for $\mathcal{R}_i$, then 
   there exists $1\leq j\leq m_i$ so that it extends to a test sequence for $cl_j(\mathcal{R}_i)$, of $\mathcal{R}_i$.  
   \end{itemize}
Finally, for any test sequence, $(h'_n)_{n<\omega}:G^{cl}\rightarrow\F$, for $\mathcal{R}$ that one of the above conditions hold, 
there exists $n_0<\omega$ such that $\F\models\phi(h'_n(\bar{x}),\bar{a})$ for all $n>n_0$.
\end{theorem}

The above fact has some strong consequences regarding definability. We record some corollaries that resolve some long standing questions. 

\begin{corollary}
A first order formula $\phi(x)$ over $\F$ is generic if and only if it contains a test sequence for the tower 
$\mathcal{T}(\langle x\rangle*\F,\F)$.
\end{corollary}

We note that the following theorem has been proved by C. Perin using different methods.

\begin{corollary}\label{PureCentra}
Let $c\in\F\setminus\{1\}$. Then the induced structure on $(C_{F}(c),\cdot)$ (be seen as a subgroup of $(\F,\cdot)$) is the structure of a pure group, i.e. every definable set 
in the induced structure can be defined by multiplication alone. 
\end{corollary}
\begin{proof}
By \cite{ThesisSklinos} and \cite{FCP} it is enough to prove that every infinite definable subset of $C_{F}(a)$ is generic. Let $X$ be an infinite definable 
subset of $C_{\F}(c)=\langle\gamma\rangle$. Then we can extract from $X$ a test sequence $(b_n,\bar{a})_{n<\omega}$ for the tower, 
$\mathcal{T}(\langle x,z\ | \ [x,z]\rangle*_{z=\gamma}\F,\F)$, obtained by gluing 
an infinite cyclic group along the centralizer of $a$ in $\F$.

According to the definition of a test sequence and Theorem \ref{FactSela} there exist test sequences $(\gamma^{kn+l})_{n<\omega}$ and $(\gamma^{-kn+l})_{n<\omega}$  
for some natural numbers $k>0$ and $l\geq 0$ such that all but finitely many elements of them belong to the definable set $X$. Thus, 
$X$ is generic in $C_{\F}(a)$.
\end{proof}

\subsection{Merzlyakov-type theorems}\label{Merzlyakov}
We fix a non abelian free group $\F:=\F(\bar{a})$. 
In this section we record some generalizations of Merzlyakov's theorem. Merzlyakov's original theorem stated:

\begin{theorem}[Merzlyakov]
Let $\Sigma(\bar{x},\bar{y},\bar{a})\subset_{\textrm{finite}}\langle \bar{x},\bar{y}\rangle*\F$ be a finite set of words. Suppose 
$\F\models \forall\bar{x}\exists\bar{y}(\Sigma(\bar{x},\bar{y},\bar{a})=1)$. Then there exists 
a ``formal solution'' $\bar{w}(\bar{x},\bar{a})\subset\langle\bar{x}\rangle*\F$ such that $\Sigma(\bar{x},\bar{w}(\bar{x},\bar{a}),\bar{a})$ 
is trivial in $\langle\bar{x}\rangle*\F$.
\end{theorem}

The above theorem is the conceptual basis of the positive solution to Tarski's problem. The first step 
towards the solution consists of generalizations of Merzlyakov's theorem to varieties corresponding to limit groups that have the structure 
of a tower. 

\begin{theorem}
Let $\Sigma(\bar{x},\bar{y},\bar{a})\subset_{\textrm{finite}}\langle \bar{x},\bar{y}\rangle*\F$ be a finite set of words. 
Let $\mathcal{T}(G,\F)$ with $G:=\langle \bar{x},\bar{a} \ | \ T(\bar{x},\bar{a})\rangle$ be a tower over $\F$. 
Let $\F\models\forall\bar{x}(T(\bar{x},\bar{a})=1\to\exists\bar{y}(\Sigma(\bar{x},\bar{y},\bar{a})=1))$. 

Then there exist finitely many closures, $cl_1(\mathcal{T}(G,\F)),\ldots, cl_k(\mathcal{T}(G,\F))$ with 
$G^{cl_i}:=\langle \bar{u}_i,\bar{x},\bar{a}\rangle$, and for each $i\leq k$ a ``formal solution'' $\bar{w}_i(\bar{u}_i,\bar{x},\bar{a})$ 
such that $\Sigma(\bar{x},\bar{w}_i(\bar{u}_i,\bar{x},\bar{a}),\bar{a})$ 
is trivial in $G^{cl_i}$. 

Moreover, for any morphism $h:G\rightarrow\F$ that factors through $\mathcal{T}(G,\F)$, there exists some $i\leq k$ 
such that $h$ extends to a morphism from $G^{cl_i}$ to $\F$ that factors through $cl_i(\mathcal{T}(G,\F))$. 

\end{theorem}

One can generalize Merzlyakov's theorem after strengthening the assumptions in the following way.

\begin{theorem}\label{ExtMerzlyakovFree}
Let $\F\models \forall\bar{x}\exists^{<\infty}\bar{y}\phi(\bar{x},\bar{y},\bar{a})$ and assume there 
exists a test sequence $(h_n)_{n<\omega}:G\rightarrow\F$ with respect to the hyperbolic tower $\mathcal{T}(G,\F)$ (for  
some ordering $(\mathcal{T}(G,\F),<)$) 
and a sequence of tuples $(\bar{c}_n)_{n<\omega}$ in $\F$ such that $\F\models\phi(h_n(\bar{x}),\bar{c}_n,\bar{a})$. Then 
there exists a tuple of words $\bar{w}(\bar{x},\bar{a})\subset G$  
such that for any test sequence $(h'_n)_{n<\omega}:G\rightarrow\F$ for the tower $\mathcal{T}(G,\F)$, 
there exists $n_0$ (that depends on the test sequence) with $\F\models\phi(h'_n(\bar{x}),h'_n(\bar{w}(\bar{x},\bar{a})),\bar{a})$  
for all $n>n_0$.
\end{theorem}

We prove:

\begin{theorem}\label{PrepareExtFormalSolutions}
Let $\mathcal{T}(G,\F)$ be a tower over $\F:=\F(\bar{a})$ with $G:=\langle\bar{u},\bar{y},\bar{a}\rangle$ and let 
$Sld:=\langle\bar{v},\bar{y},\bar{a} \ | \ \Sigma(\bar{v},\bar{y},\bar{a})\rangle$ be a solid limit group 
with respect to $H:=\langle\bar{y},\bar{a}\rangle$. 

Assume that for some ordering of the tower $(\mathcal{T}(G,\F),<)$ there exists a test sequence $(h_n)_{n<\omega}:G\rightarrow\F$ 
that extends to a sequence $(H_n)_{n<\omega}:G*_HSld\rightarrow\F$ so that, for each $n<\omega$,   
$(H_n\upharpoonright Sld)_{n<\omega}$ is a strictly solid morphism. 

Then: 
\begin{itemize}
 \item(existence for a single test sequence) there exist a closure $cl(\mathcal{T}(G,\F))$ of $\mathcal{T}(G,\F)$ with $G^{cl}:=\langle \bar{w},\bar{u},\bar{y},\bar{a}\rangle$ 
and a morphism $r:Sld\rightarrow G^{cl}$ with $r(\bar{y},\bar{a})=(\bar{y},\bar{a})$ such that a subsequence of $(h_n)_{n<\omega}$ (still denoted 
$(h_n)_{n<\omega}$) extends to a test sequence $(g_n)_{n<\omega}:G^{cl}\rightarrow\F$ of $cl(\mathcal{T}(G,\F))$ for the inherited ordering and 
for which, for each $n$, $g_n\circ r$ is in the same strictly solid family as $H_n\upharpoonright Sld$; 
 \item(universal property for all test sequences) there exist finitely many closures $cl_1(\mathcal{T}(G,\F)),$ $\ldots, 
 cl_k(\mathcal{T}(G,\F))$ with 
 $G^{cl_i}:=\langle \bar{w}_i,\bar{u},\bar{y},\bar{a}\rangle$ and for each $i\leq k$ there is a morphism $r_i:Sld\rightarrow G^{cl_i}$ 
 with $r_i(\bar{y},\bar{a})=(\bar{y},\bar{a})$ such that for every test sequence with respect to $(\mathcal{T}(G,\F)<)$, $(h_n)_{n<\omega}:G\rightarrow\F$, 
 that extends to a sequence $(H_n):G*_HSld\rightarrow\F$ so that, for each $n<\omega$, $H_n\upharpoonright Sld$ is 
 a strictly solid morphism, there exists $i\leq k$ such that a subsequence of $(h_n)_{n<\omega}$ extends to a test sequence $(g_n)_{n<\omega}:G^{cl_i}\rightarrow\F$ 
 for the inherited ordering and, for each $n<\omega$, $g_n\circ r_i$ is in the same strictly solid family as $H_n\upharpoonright Sld$. 
\end{itemize}

\end{theorem}

We first record a result of Sela (see \cite[Proposition 1.9]{Sel3}) which is essential in proving Theorem \ref{PrepareExtFormalSolutions}. 

\begin{proposition}[Sela]\label{ShortEncore}
Let $Sld$ be a solid limit group with respect to a finitely generated subgroup $H$. Let $(s_n)_{n<\omega}:Sld\rightarrow\F$ 
be a converging sequence of strictly solid morphisms. Let $q:Sld\twoheadrightarrow Q:=Sld/\Ker s_n$ and $\Delta$ be a $GAD$ 
of $Q$. 

Then for each edge group and rigid (non abelian) vertex group of $JSJ_H(Sld)$, its image by $q$ can either:
\begin{itemize}
 \item be conjugated into an edge group or rigid (non abelian) vertex group of $\Delta$; or 
 \item if it is conjugated into an abelian vertex group, then it is conjugated it the peripheral subgroup 
 of this group.
\end{itemize}

\end{proposition}
\noindent
{\em Proof( of Theorem \ref{PrepareExtFormalSolutions}).} 
The proof follows the arguments in the proof of Theorem 1.18 in \cite{Sel2}. Thus we only point out the parts that differ. 

We note that by Lemma \ref{DefDeg} there exists 
a system of equations $\Psi=1$ over $\F$, that collects all morphisms $h:Sld\rightarrow\F$ that are degenerate. 

We start with 
a test sequence $(h_n)_{n<\omega}:G\rightarrow\F$ that by the hypothesis we can extend it to a sequence of 
morphisms $(H_n)_{n<\omega}: G*_HSld\rightarrow\F$ such that each morphism restricts to a strictly solid morphism on $Sld$. 
For each $n$ we choose the shortest possible morphism (with respect to 
a fixed basis of $\F$) that belongs to the same strictly solid family as $H_n\upharpoonright Sld$. The problem in applying directly 
the arguments of proof Theorem 1.18 in \cite{Sel2} would be that the shortening argument could shorten our morphisms in a way that 
they do not belong in the same strictly solid family. According to Proposition \ref{ShortEncore} and Lemma \ref{StrictlySolidRel} 
this is not possible. Thus, the quotient group $G*_HSld/Ker(H_n)$ under the stable kernel of $(H_n)_{n<\omega}$ would have the structure 
of a closure of $\mathcal{T}(G,\F)$ as we wanted. 

The universal property follows exactly as in Theorem 1.18 in \cite{Sel2}. \qed
\ \\ \\
Theorem \ref{ExtMerzlyakovFree} is an easy corollary of our next result. 

\begin{theorem}\label{ExtFormalSolutions}
Let $\mathcal{T}(G,\F)$ with $G:=\langle\bar{u},\bar{y},\bar{a} \ | \ T(\bar{u},\bar{y},\bar{a})\rangle$ be a tower over $\F$. 
Let $\{\phi(\bar{x},\bar{y},\bar{a})\}$ be a parametric family of first order formulas with respect to $\bar{y}$, 
such that $\F\models\forall \bar{y}\exists^{<\infty}\bar{x}\phi(\bar{x},\bar{y},\bar{a})$. 

Suppose there exists a test sequence $(h_n)_{n<\omega}:G\rightarrow\F$ for $\mathcal{T}(G,\F)$ and a sequence $\{(\bar{b}_n)_{n<\omega}\}$ of 
tuples in $\F$ such that $\F\models \phi(\bar{b}_n,h_n(\bar{y}),\bar{a})$. 

Then there exist finitely many closures $cl_1(\mathcal{T}(G,\F)),\ldots, cl_k(\mathcal{T}(G,\F))$ with $G^{cl_i}:=\langle \bar{v}_i,\bar{u},\bar{y},$ 
$\bar{a}\ | \ cl_i(T)(\bar{v}_i,\bar{u},\bar{y},\bar{a})\rangle$, 
and for each $i\leq k$ a tuple of words 
$\bar{x}_i:=\bar{x}_i(\bar{v}_i,\bar{u},\bar{y},\bar{a})$ such that for any test sequence $(h_n)_{n<\omega}:G\rightarrow\F$ for $\mathcal{T}(G,\F)$ 
for which there exists $\{(\bar{b}_n)_{n<\omega}\}$ with $\F\models\phi(\bar{b}_n,h_n(\bar{y}),\bar{a})$, 
there exists $i\leq k$ and a subsequence of $(h_n)_{n<\omega}$ extends to 
a test sequence $(h'_n)_{n<\omega}:G^{cl_i}\rightarrow\F$ for 
$cl_i(\mathcal{T}(G,\F))$ with  $\F\models\phi(h'_n(\bar{x}_i),h'_n(\bar{y}),\bar{a})$.
\end{theorem}
\begin{proof}
We consider the solid limit groups on which the graded towers of the envelope of $\{\phi(\bar{x},\bar{y},\bar{a})\}$ are based on. By Theorem \ref{RigidGradedTower} 
the tuple $\bar{x}$ of each graded tower belongs to $\hat{H}_i$ of each $\hat{JSJ(Sld_i)}$. By the properties of the graded envelope every 
test sequence for $\mathcal{T}(G,\F)$ that can be extended to a sequence of solutions for $\phi$, has a subsequence that 
extends to a sequence of morphisms of $G*_{H_i}Sld_i$, thus we can now apply Theorem \ref{PrepareExtFormalSolutions} in order to conclude.
\end{proof}

\section{Main proof}
In this final section we bring everything together in order to prove the main result of this paper: no infinite field is definable in a non abelian free group. 
We have split the proof in two parts. Assuming that $X$ is an infinite definable set, in the first subsection \ref{AbCase} we tackle the case where  
$X$ is co-ordinated by a finite set of centralizers. In this case we have already proved that $X$ is one-based thus we cannot give it a definable field structure. 
In the third subsection \ref{GenCas} we prove that in any other case $X$ cannot be given definably the structure of an abelian group. 

We have inserted a subsection between subsections \ref{AbCase} and \ref{GenCas} that is free of certain technical problems in order to 
make the ideas of our proof more transparent.

\subsection{Abelian case}\label{AbCase}

In this subsection we tackle the special case where the diophantine envelope of a definable set $X$ consists essentially of the ground floor $\F$ (the coefficient group) 
and only abelian floors are glued over centralizers of elements in $\F$. The prototypical case being that $X$ is a product of centralizers. Since centralizers 
are pure groups, their theory is one-based, thus $X$ cannot be given definably a field structure. In the general case we can show that $X$ is internal to 
a product of centralizers, thus still one-based. 

\begin{theorem}\label{AbelianCase}
Let $X:=\phi(\bar{x},\bar{a})$ be a first order formula over $\F$ with $\abs{X}=\infty$. Let  
$\{(\mathcal{T}(G_i,$ $\F))_{i\leq k}\}$, where $G_i:=\langle \bar{u}_i,\bar{x},\bar{a} \ | \ \Sigma_i(\bar{u}_i,\bar{x},\bar{a})\rangle$ 
be a diophantine envelope for $X$. Suppose, for each $i\leq k$,  the tuple $\bar{x}$ belongs to the sub-tower of $\mathcal{T}_i$ consisting 
of the ground floor $\F$ and the abelian flats glued over centralizers of elements in $\F$. 

Then $X$ cannot be given definably the structure of a field.
\end{theorem}

\begin{proof}
By Corollary \ref{Envelope} (i), we have that $X$ is a subset of the diophantine set $D(\bar{x}):=\exists\bar{u}(\bigvee_{i=1}^{k}\Sigma_i(\bar{u},\bar{x},\bar{a})=1)$. 
But then by the hypothesis we have that there are finitely many words $\bar{w}_i(\bar{z},\bar{a})$ so that for each element $\bar{b}$ of $D$ there exist some $i$ 
and some elements, $\bar{c}_1,\ldots,\bar{c}_l$ from centralizers of elements in $\F$, such that $\bar{b}={w}_i(\bar{c}_1,\ldots,\bar{c}_l,\bar{a})$. 

Thus, $D$ is internal to a set of centralizers and by Corollary \ref{PureCentra} this centralizers are one-based sets. Thus $D$ is one-based, in particular no 
infinite definable subset of $D$ can be given definably a field structure. 
\end{proof}

\subsection{Hyperbolic case}\label{HypCase}
In this subsection we tackle a case, in which a diophantine envelope for a definable set $X$ contains a hyperbolic tower. 
In contrast to the abelian case, this case is not strictly needed but it will serve us as a toy example for the general case. 

\begin{theorem}
Let $X:=\phi(\bar{x},\bar{a})$ be a first order formula over $\F$ with $\abs{X}=\infty$. Let $\{(\mathcal{T}_i(G_i,$ $\F))_{i\leq k}\}$, 
where $G_i:=\langle \bar{u},\bar{x},\bar{a} \ | \ \Sigma_i\rangle$,  
be a diophantine envelope for $X$. Suppose, for some $i\leq k$, $\mathcal{T}_i$ is hyperbolic.

Then $X$ cannot be given definably an abelian group structure.
\end{theorem}

\begin{proof}
Suppose for the sake of contradiction that $(\phi(\bar{x},\bar{a}),\oplus(\bar{x},\bar{y},\bar{z},\bar{a}))$ is an abelian group. 
Let $\mathcal{T}(G,\F)$, where $G:=\langle\bar{u},\bar{x},\bar{a} \ |\ \Sigma(\bar{u},\bar{x},\bar{a})\rangle$ be the hyperbolic tower contained in a 
diophantine envelope of $\phi(\bar{x},\bar{a})$.
 
We consider the twin tower $\mathcal{T}\#\mathcal{T}(G,\F)$, its corresponding group $G*_{\F}G':=\langle\bar{u},\bar{x},\bar{u}',\bar{x}',\bar{a}\ | \ $ $
\Sigma(\bar{u},\bar{x},\bar{a}), \Sigma(\bar{u}',\bar{x}',\bar{a})\rangle$ and a 
test sequence $(h_n)_{n<\omega}:G*_{\F}G'\rightarrow\F$ for $\mathcal{T}\#\mathcal{T}$ with respect to the twin tower and its natural order.

We note that by the properties of test sequences the restrictions $(h_n)_{n<\omega}\upharpoonright G$, $(h_n)_{n<\omega}\upharpoonright G'$ 
are both test sequences for $\mathcal{T}(G,\F)$ (for the natural order), thus by Theorem \ref{FactSela} 
the sequence of couples $(h_n(\bar{x}),h_n(\bar{x}'))$ belongs to $X\times X$. 
Therefore, there exists a sequence of tuples $(\bar{c}_n)_{n<\omega}$ of $\F$ such that $\F\models\oplus(h_n(\bar{x}),h_n(\bar{x}'), \bar{c}_n)$. 
We can now use Theorem \ref{ExtMerzlyakovFree} in order to obtain a formal solution 
$\bar{w}(\bar{u},\bar{x},\bar{u}',\bar{x}',\bar{a})$, i.e. a tuple of words in $G*_{\F}G'$, such that 
$\F\models\oplus(h_n(\bar{x}),h_n(\bar{x}'),h_n(\bar{w}(\bar{u},\bar{x},\bar{u}',\bar{x}',$ $\bar{a})),\bar{a})$. 

We claim that there is an element in the tuple $\bar{w}(\bar{u},\bar{x},\bar{u}',\bar{x}',\bar{a})$ that 
does not belong to $\F$. If, for a contradiction, $\bar{w}(\bar{u},\bar{x},\bar{u}',\bar{x}',\bar{a})=\bar{b}$ for some 
tuple $\bar{b}$ of $\F$, then we can replace $(h_n)\upharpoonright G'$
with a proper subsequence, this will still be a test sequence for $\mathcal{T}\#\mathcal{T}$, call it $(g_n)_{n<\omega}$. 
By Theorem \ref{FactSela} we get $\F\models\oplus(g_n(\bar{x}),g_n(\bar{x}'),\bar{b},\bar{a})$, but $g_n(\bar{x})=h_n(\bar{x})$ 
while for $n$ large enough $g_n(\bar{x}')\neq h_n(\bar{x}')$, thus $h_n(\bar{x})\oplus h_n(\bar{x}')\neq h_n(\bar{x})\oplus g_n(\bar{x}')$, 
a contradiction since both should be equal to $\bar{b}$.

We next consider exchanging the role of $(h_n)_{n<\omega}\upharpoonright G$ with $(h_n)_{n<\omega}\upharpoonright G'$ 
since $G'$ is a copy of $G$ this is again a sequence of morphisms from $G*_{\F}G'$ to $\F$ which we call $(f_n)_{n<\omega}$. 
The sequence $(f_n)_{n<\omega}$ is a test sequence for $\mathcal{T}\#\mathcal{T}$ but for a different order, namely the floors that 
consist the group $G'$ come before the floors that consist the group $G$. Thus,  we can still 
apply Theorem \ref{FactSela} so that $\F\models \oplus(f_n(\bar{x}),f_n(\bar{x}'),f_n(\bar{w}(\bar{u},\bar{x},\bar{u}',\bar{x}',$ $\bar{a})),\bar{a})$. 
By the definition of $(f_n)_{n<\omega}$ it follows that $h_n(\bar{x})=f_n(\bar{x}')$ and $h_n(\bar{x}')=f_n(\bar{x})$. In particular, 
since the group operation $\oplus$ is abelian we observe that 

$$h_n(\bar{w}(\bar{u},\bar{x},\bar{u}',\bar{x}',\bar{a}))=f_n(\bar{w}(\bar{u},\bar{x},\bar{u}',\bar{x}',\bar{a}))=h_n(\bar{w}(\bar{u}',\bar{x}',\bar{u},\bar{x},\bar{a}))$$ 

In particular the equation $\bar{w}(\bar{u},\bar{x},\bar{u}',\bar{x}',\bar{a}))=\bar{w}(\bar{u}',\bar{x}',\bar{u},\bar{x},\bar{a}))$ 
holds for a test sequence with respect to $\mathcal{T}\#\mathcal{T}$. But a test sequence witnesses that $G*_{\F}G'$ is a limit group, 
thus the above equation is a formal relation in the corresponding group $G*_{\F}G'$. Now since $\bar{w}(\bar{u},\bar{x},\bar{u}',\bar{x}',\bar{a}))$ 
does not live in $\F$ some element in $\bar{w}(\bar{u},\bar{x},\bar{u}',\bar{x}',\bar{a}))$ and the corresponding 
element in $\bar{w}(\bar{u}',\bar{x}',\bar{u},\bar{x},\bar{a}))$ have different normal forms with respect to the amalgamated free product 
$G*_{\F}G'$, a contradiction.
\end{proof}

\subsection{General case}\label{GenCas}

\begin{theorem}
Let $X:=\phi(\bar{x},\bar{a})$ be a first order formula over $\F$ with $\abs{X}=\infty$. 
Then $X$ cannot be given definably a field structure.
\end{theorem}
\begin{proof}
Let $\{(\mathcal{T}(G_i,\F))_{i\leq k}\}$  
be a diophantine envelope of $\phi(\bar{x},\bar{a})$. By Theorem \ref{AbelianCase} we may assume that there exists a tower $\mathcal{T}(G,\F)$ (where 
$G:=\langle\bar{u},\bar{x},\bar{a} \ |\ \Sigma(\bar{u},\bar{x},\bar{a}\rangle$) in the diophantine envelope such that either 
the first floor is not an abelian floor or the tuple of elements $\bar{x}$ does not live in the group 
that corresponds to the first floor. We next prove that in either case $X$ cannot be given definably an abelian group structure.  
Assume for a contradiction that $(X,\oplus(\bar{x},\bar{x}',\bar{z}))$ for some first order formula $\oplus(\bar{x},\bar{x}',\bar{z})$ is an abelian group.

We first consider the case where the first floor, $\mathcal{G}(G^1,\F)$, of the tower $\mathcal{T}(G,\F)$ is not abelian.  
Let $G':=\langle \bar{u}',\bar{x}',\bar{a} \ |\ \Sigma(\bar{u}',\bar{x}',\bar{a})\rangle$ be a copy of $G$. 
We consider the twin tower $\mathcal{T}\#\mathcal{T}(G*_{\F}G',\F)$, where $G*_{\F}G':=\langle\bar{u},\bar{x},\bar{u}',\bar{x}',\bar{a}\ | \ 
\Sigma(\bar{u},\bar{x},\bar{a}), \Sigma(\bar{u}',\bar{x}',\bar{a})\rangle$ and the finite set of closures 
$cl_1(\mathcal{T}\#\mathcal{T}(G*_{\F}G',\F), \ldots, cl_k(\mathcal{T}\#\mathcal{T}(G*_{\F}G',\F)$ of the twin tower provided by 
applying Theorem \ref{ExtFormalSolutions} to the twin tower $\mathcal{T}\#\mathcal{T}(G*_{\F}G',\F)$ and the first order formula 
$\oplus(\bar{x},\bar{x}',\bar{z})$ be seen as a parametric family with respect to $\bar{x},\bar{x}'$. 

For each abelian flat $E\oplus\Z^m$ that appears in the tower $\mathcal{T}(G,\F)$, and each closure $cl_i(\mathcal{T}\#\mathcal{T}(G*_{\F}G',\F)$ of the twin tower 
$\mathcal{T}\#\mathcal{T}(G*_{\F}G',\F)$ (after passing to its symmetric closure still denoted $cl_i(\mathcal{T}\#\mathcal{T}(G*_{\F}G',\F)$) 
there exist two corresponding closure embeddings $f^i_1,f^i_2$. 
To each of these embeddings we may assign a subgroup $U^i_1, U^i_2$ of $\mathbb{A}^m$ according to Remark \ref{MorphismExtensions}. Now 
we choose a test sequence for the twin tower $\mathcal{T}\#\mathcal{T}(G*_{\F}G',\F)$ such that $(h_n(\bar{x}),h_n(\bar{x}'))$ 
belongs to $X\times X$ and moreover the values of the powers of the generators for each abelian flat belong to the same coset 
modulo all $U^i_1,U^i_2$ for $i\leq k$. 

Such a test sequence must extend to a test sequence $(h'_n)_{n<\omega}$ of some of the closures say $cl_1(\mathcal{T}\#\mathcal{T}(G*_{\F}G',\F)$, 
moreover reversing the role of $(h'_n)_{n<\omega}\upharpoonright G^{cl_1}$ 
with $(h_n)_{n<\omega}\upharpoonright G'^{cl_1}$ is still a test sequence for the closure. We can now apply the same argument 
as in the case of the hyperbolic tower.

Finally, we are left with the case where the first floor $\mathcal{G}(G^1,\F)$ is abelian. In this case the twin tower has a slightly different 
form (see Proposition \ref{TwinAbel}), as a group it is the amalgamated free product of $G_{Db}$ with $_fG_{Db}$ over the double $G^1_{Db}$ of $G^1$ with respect to the 
floor $\mathcal{G}(G^1,\F)$. In the first step of our previous argument we applied Theorem \ref{ExtFormalSolutions} in order to obtain 
finitely many closures and formal solutions in these closures that their images under a test sequence gives us the product operation. 
In order to conclude as in the previous case it is enough to show that these formal solutions do not live in the closure of $G^1_{Db}$ for each closure 
of the twin tower. Indeed, the values of such a formal solution would depend only on the values a test sequence 
$(h_n)_{n<\omega}:G_{Db}*_{G^1_{Db}}(_fG_{Db}\rightarrow\F$ of the twin tower gives to the elements of $G^1_{Db}$. But as we have assumed that 
$\bar{x}$ does not live in $G^1$, we can refine the test sequence $(h_n)_{n<\omega}$ as follows: 
$(h'_n)_{n<\omega}\upharpoonright G_{Db}=(h_n)_{n<\omega}\upharpoonright G_{Db}$ and $(h'_n)_{n<\omega}\upharpoonright\ _fG_{Db}$ 
is a sequence that eventually gives to $\bar{x}'$ different values than $(h_n)_{n<\omega}$. Now, as in the hyperbolic case, since the 
product of $h_n(\bar{x}),h_n(\bar{x}')$ only depends on the values $h_n$ gives to the elements of the group $G^1_{Db}$, there is a contradiction. 
\end{proof}

\bibliography{biblio}
\end{document}